\documentclass[10pt]{article}

\usepackage{amsthm}
\usepackage{thmtools}
\usepackage[colorlinks=true,linkcolor=blue]{hyperref}
\usepackage{enumitem}
\usepackage{mathrsfs}
\usepackage{bbm}
\usepackage[margin=1.6in]{geometry}
\usepackage{marvosym}
\usepackage{graphicx}

\usepackage{tikz}
\usepackage{url}

\newcommand\spiral{%
  \begin{tikzpicture}[y=.009ex,x=.009ex,yscale=-1]
    \path[fill] (74.2500,379.5469) -- (62.4375,373.3594) .. controls
      (83.4375,365.2031) and (99.6563,356.1563) .. (111.0938,346.2188) ..
controls
      (122.5313,336.2813) and (131.0625,324.9375) .. (136.6875,312.1875) ..
controls
      (142.3125,299.4375) and (145.1250,284.4375) .. (145.1250,267.1875) ..
controls
      (145.1250,256.6875) and (141.9375,247.5000) .. (135.5625,239.6250) ..
controls
      (129.1875,231.7500) and (121.7813,226.3125) .. (113.3438,223.3125) ..
controls
      (104.9063,220.3125) and (95.6250,218.8125) .. (85.5000,218.8125) ..
controls
      (75.3750,218.8125) and (66.0937,220.3125) .. (57.6562,223.3125) ..
controls
      (49.2187,226.3125) and (41.8125,231.7500) .. (35.4375,239.6250) ..
controls
      (29.0625,247.5000) and (25.8750,256.6875) .. (25.8750,267.1875) ..
controls
      (25.8750,277.6875) and (27.3750,286.9687) .. (30.3750,295.0312) ..
controls
      (33.3750,303.0937) and (38.1563,308.9063) .. (44.7188,312.4688) ..
controls
      (51.2813,316.0313) and (59.2500,317.8125) .. (68.6250,317.8125) ..
controls
      (73.6875,317.8125) and (77.5782,316.7813) .. (80.2969,314.7188) ..
controls
      (83.0156,312.6563) and (85.2187,309.7500) .. (86.9062,306.0000) ..
controls
      (88.5937,302.2500) and (89.4375,297.3750) .. (89.4375,291.3750) ..
controls
      (89.0625,288.3750) and (88.4063,285.6563) .. (87.4688,283.2188) ..
controls
      (86.5313,280.7813) and (85.0313,278.9063) .. (82.9688,277.5938) ..
controls
      (80.9063,276.2813) and (78.3750,275.6250) .. (75.3750,275.6250) ..
controls
      (72.9375,275.6250) and (70.5469,277.1250) .. (68.2031,280.1250) ..
controls
      (65.8594,283.1250) and (63.7500,288.4219) .. (61.8750,296.0156) --
      (57.6562,295.8750) .. controls (55.5937,295.8750) and (54.0937,295.4063)
..
      (53.1562,294.4688) .. controls (52.2187,293.5313) and (51.7500,292.5000)
..
      (51.7500,291.3750) .. controls (51.7500,286.5000) and (52.6875,281.7187)
..
      (54.5625,277.0312) .. controls (56.4375,272.3437) and (59.2500,268.5000)
..
      (63.0000,265.5000) .. controls (66.7500,262.5000) and (70.8750,261.0000)
..
      (75.3750,261.0000) .. controls (80.6250,261.0000) and (85.4063,262.1250)
..
      (89.7188,264.3750) .. controls (94.0313,266.6250) and (97.4063,270.0937)
..
      (99.8438,274.7812) .. controls (102.2813,279.4687) and (103.5000,285.0000)
..
      (103.5000,291.3750) .. controls (103.5000,299.6250) and
(102.3750,306.5625) ..
      (100.1250,312.1875) .. controls (97.8750,317.8125) and (94.0313,322.5937)
..
      (88.5938,326.5312) .. controls (83.1563,330.4687) and (76.5000,332.4375)
..
      (68.6250,332.4375) .. controls (57.3750,332.4375) and (47.0625,329.8125)
..
      (37.6875,324.5625) .. controls (28.3125,319.3125) and (21.5625,311.5313)
..
      (17.4375,301.2188) .. controls (13.3125,290.9063) and (11.2500,279.5625)
..
      (11.2500,267.1875) .. controls (11.2500,254.0625) and (14.2500,241.9687)
..
      (20.2500,230.9062) .. controls (26.2500,219.8437) and (35.1563,211.3125)
..
      (46.9688,205.3125) .. controls (58.7813,199.3125) and (71.6250,196.3125)
..
      (85.5000,196.3125) .. controls (99.3750,196.3125) and (112.2187,199.3125)
..
      (124.0312,205.3125) .. controls (135.8437,211.3125) and
(144.7500,219.8437) ..
      (150.7500,230.9062) .. controls (156.7500,241.9687) and
(159.7500,254.0625) ..
      (159.7500,267.1875) .. controls (159.7500,287.4375) and
(156.5625,304.7813) ..
      (150.1875,319.2188) .. controls (143.8125,333.6563) and
(134.0625,346.1250) ..
      (120.9375,356.6250) .. controls (107.8125,367.1250) and (92.2500,374.7656)
..
      (74.2500,379.5469) -- cycle;
  \end{tikzpicture}%
}

\renewcommand{\spiral}{sp}

\usepackage{amsmath}
\usepackage{amsxtra}
\usepackage{amssymb}
\usepackage{lineno}
\usepackage{turnstile}
\usepackage{verbatim}
\usepackage{enumitem}
\usepackage[retainorgcmds]{IEEEtrantools}
\usepackage{accents}
\usepackage[all]{xy}

\declaretheoremstyle[bodyfont=\sl]{slanted}

\swapnumbers
\declaretheorem[name=Definition,style=definition,qed=$\dashv$,
numberwithin=section]{dfn}
\declaretheorem[name=Definition,style=definition,numbered=no,qed=$\dashv$]{dfn*}
\declaretheorem[name=Definition,style=definition,numbered=no]{dfnnoqed*}
\declaretheorem[name=Example,style=definition,sibling=dfn]{exm}
\declaretheorem[name=Theorem,style=slanted,sibling=dfn]{tm}
\declaretheorem[name=Theorem,style=slanted,numbered=no]{tm*}
\declaretheorem[name=Lemma,style=slanted,sibling=dfn]{lem}
\declaretheorem[name=Corollary,style=slanted,sibling=dfn]{cor}
\declaretheorem[name=Corollary,style=slanted,numbered=no]{cor*}
\declaretheorem[name=Remark,style=definition,sibling=dfn]{rem}

\swapnumbers
\declaretheoremstyle[headfont=\scshape]{claimstyle}

\declaretheorem[name=Claim,style=claimstyle]{clmtwo}

\declaretheorem[name=Claim,style=claimstyle]{clmfour}
\declaretheorem[name=Claim,style=claimstyle]{clmfive}
\declaretheorem[name=Claim,style=claimstyle]{clmsix}

\declaretheorem[name=Claim,style=claimstyle,numbered=no]{clm*}

\declaretheorem[name=Subclaim,style=claimstyle,numberwithin=clmfive]{sclmfive}
\declaretheorem[name=Subclaim,style=claimstyle,numberwithin=clmsix]{sclmsix}

\declaretheorem[name=Subclaim,style=claimstyle,numbered=no]{sclm*}

\declaretheorem[name=Subsubclaim,style=claimstyle,numberwithin=sclmtwo]{ssclmtwo
}
\declaretheorem[name=Subsubclaim,style=claimstyle,numberwithin=sclmthree]{
ssclmthree}
\declaretheorem[name=Subsubclaim,style=claimstyle,numberwithin=sclmfour]{
ssclmfour}
\declaretheorem[name=Subsubclaim,style=claimstyle,numberwithin=sclmfive]{
ssclmfive}
\declaretheorem[name=Subsubclaim,style=claimstyle,numberwithin=sclmsix]{ssclmsix
}
\declaretheorem[name=Subsubclaim,style=claimstyle,numberwithin=sclmseven]{
ssclmseven}
\declaretheorem[name=Subsubclaim,style=claimstyle,numberwithin=sclmeight]{
ssclmeight}
\declaretheorem[name=Subsubclaim,style=claimstyle,numberwithin=sclmnine]{
ssclmnine}
\declaretheorem[name=Subsubclaim,style=claimstyle,numberwithin=sclmten]{ssclmten
}
\declaretheorem[name=Subsubclaim,style=claimstyle,numbered=no]{ssclm*}

\declaretheoremstyle[headfont=\scshape]{casestyle}

\declaretheorem[name=Case,style=casestyle]{case}
\declaretheorem[name=Case,style=casestyle]{casetwo}

\title{Ordinal definability in $L[\es]$}
\author{Farmer Schlutzenberg\footnote{afirstname dot alastname at tuwien dot ac dot at, afirstname dot alastname at gmail dot com,
\url{\myurl}}
}

\newcommand{\tame}{\mathrm{t}}
\newcommand{\passive}{\mathrm{pv}}
\newcommand{\moon}{{\text{\tiny\spiral}}}
\renewcommand{\Moon}{{\text{\footnotesize\spiral}}}
\newcommand{\Hh}{\mathcal{H}}
\newcommand{\lambdabar}{{\bar{\lambda}}}
\newcommand{\deltabar}{{\bar{\delta}}}
\newcommand{\eps}{\varepsilon}
\newcommand{\Pbar}{{\bar{P}}}
\newcommand{\Coll}{\mathrm{Coll}}
\newcommand{\ZF}{\mathrm{ZF}}
\newcommand{\wlim}{\mathrm{wlim}}
\newcommand{\cs}{\mathrm{cs}}
\newcommand{\sJs}{\mathrm{sJs}}
\newcommand{\ml}{\mathrm{ml}}
\newcommand{\lgcd}{\mathrm{lgcd}}
\renewcommand{\pm}{\mathrm{pm}}
\newcommand{\trancl}{\mathrm{trancl}}
\newcommand{\stack}{\mathrm{stack}}

\newcommand{\QQ}{\mathbb Q}
\newcommand{\RR}{\mathbb R}
\newcommand{\PP}{\mathbb P}
\newcommand{\BB}{\mathbb B}
\newcommand{\sub}{\subseteq}
\newcommand{\cross}{\times}

\newcommand{\inter}{\cap}

\newcommand{\om}{\omega}
\newcommand{\pow}{\mathcal{P}}
\newcommand{\OR}{\mathrm{OR}}
\newcommand{\Hull}{\mathrm{Hull}}
\newcommand{\cut}{\backslash}

\newcommand{\Tt}{\mathcal{T}}
\newcommand{\Ss}{\mathcal{S}}
\newcommand{\Uu}{\mathcal{U}}

\newcommand{\rg}{\mathrm{rg}}
\newcommand{\dom}{\mathrm{dom}}
\newcommand{\ins}{\trianglelefteq}
\newcommand{\nins}{\ntrianglelefteq}
\newcommand{\pins}{\triangleleft}
\newcommand{\npins}{\ntriangleleft}
\newcommand{\crit}{\mathrm{cr}}

\newcommand{\union}{\cup}
\newcommand{\rest}{\!\upharpoonright\!}
\newcommand{\com}{\circ}

\newcommand{\lh}{\mathrm{lh}}
\newcommand{\Ult}{\mathrm{Ult}}

\newcommand{\sats}{\models}
\newcommand{\elem}{\preccurlyeq}
\newcommand{\J}{\mathcal{J}}

\newcommand{\PS}{\mathrm{PS}}

\newcommand{\HOD}{\mathrm{HOD}}
\newcommand{\HC}{\mathrm{HC}}
\newcommand{\ZFC}{\mathrm{ZFC}}

\newcommand{\es}{\mathbb{E}}

\newcommand{\iso}{\cong}

\newcommand{\core}{\mathfrak{C}}

\newcommand{\her}{\mathcal{H}}

\newcommand{\pred}{\mathrm{pred}}

\newcommand{\un}{\union}

\newcommand{\id}{\mathrm{id}}

\newcommand{\conc}{\ \widehat{\ }\ }

\newcommand{\forces}{\dststile{}{}}

\newcommand{\bfSigma}{\undertilde{\Sigma}}

\newcommand{\rSigma}{\mathrm{r}\Sigma}

\DeclareMathOperator{\Th}{Th}

\DeclareMathOperator{\card}{card}
\DeclareMathOperator{\cof}{cof}

\newcommand{\OD}{\mathrm{OD}}

\newcommand{\bfrSigma}{\undertilde{\rSigma}}

\newcommand{\psub}{\subsetneq}

\newcommand{\Xx}{\mathcal{X}}

\newcommand{\alphavec}{\vec{\alpha}}
\newcommand{\cHull}{\mathrm{cHull}}

\newcommand{\lpole}{\left\lfloor}
\newcommand{\rpole}{\right\rfloor}
\newcommand{\univ}[1]{\lpole #1\rpole}
\newcommand{\tu}{\textup}

\newcommand{\R}{\mathcal{R}}

\newcommand{\eqdef}{=_{\mathrm{def}}}
\newcommand{\pvec}{\vec{p}}

\newcommand{\successor}{\mathrm{succ}}

\newcommand{\betavec}{\vec{\beta}}
\newcommand{\canM}{\mathbbm{m}}
\newcommand{\canN}{\mathbbm{n}}
\newcommand{\canP}{\mathbbm{p}}

\newcommand{\Momone}{\canM}
\newcommand{\dropset}{\mathscr{D}}
\newcommand{\branch}{\mathrm{branch}}
\newcommand{\Vop}{\mathrm{Vop}}

\newcommand{\measlim}{\mathrm{ml}}
\newcommand{\leedsurl}{https://conferences.leeds.ac.uk/lc2016/slides/}
\newcommand{\gironaurlone}{https://ivv5hpp.uni-muenster.de}
\newcommand{\gironaurltwo}{/u/rds/girona_meeting_2018.html}
\newcommand{\gironaurl}{\gironaurlone\gironaurltwo}
\newcommand{\irvineurlone}{https://www.math.uci.edu}
\newcommand{\irvineurltwo}{/~mzeman/CMI/cmi-2016.html}
\newcommand{\irvineurl}{\irvineurlone\irvineurltwo}
\newcommand{\myurl}{https://sites.google.com/site/schlutzenberg/home-1}
\begin{document}
\maketitle
\begin{abstract}
Let $M$ be a tame mouse modelling $\ZFC$. We show that
$M$ satisfies ``$V=\HOD_x$ for some real $x$'',
and that the restriction $\es^M\rest[\om_1^M,\OR^M)$ of the
extender sequence $\es^M$ of $M$ to indices above $\omega_1^M$ is definable
without parameters over the universe of $M$.
We show that  $M$ has universe
$\HOD^M[X]$, where $X=M|\om_1^M$ is the initial segment of $M$ of
height $\om_1^M$ (including $\es^M\rest\om_1^M$),
and that $\HOD^M$ is the universe of a premouse
over some $t\sub\om_2^M$. We also show that $M$ has no proper grounds
via strategically $\sigma$-closed forcings.

We then extend some of these results partially to
non-tame mice, including a proof that many natural
$\varphi$-minimal mice model ``$V=\HOD$'', assuming a certain fine structural
hypothesis whose proof has almost been given elsewhere.
\end{abstract}

\section{Introduction}

Let $M$ be a mouse. We write $\es^M$ for the extender sequence of $M$, not
including the active extender $F^M$ of $M$, and $\Momone^M=M|\om_1^M$ for the initial segment of $M$ of height $\om_1^M$ (incorporating $\es^M\rest\om_1^M$).\footnote{See
\S\ref{subsec:terminology}
for (a reference to) more terminology.}
It was shown in
\cite[Theorem 3.11]{V=HODX_pub} that if $M$ has no largest cardinal
(in fact more generally than this) then $\es^M$
is definable over the universe $\univ{M}$ of $M$
from the parameter $M|\om_1^M$.
We
consider here the following questions:
\begin{enumerate}[label=--]
 \item Is $\es^M$ definable over $\univ{M}$ from a real parameter?
  \item How much of the iteration strategy $\Sigma_{\Momone}$ for
$\Momone=M|\om_1^M$ is known to
$M$?
 \item What can be said about the structure of $\HOD^{\univ{M}}$? How close is
$\HOD^{\univ{M}}$ to $M$?
\end{enumerate}
We will see that these questions are interrelated.

We write $\es_+^M=\es^M\conc F^M$.
Recall that a premouse $M$ is \emph{non-tame} iff there is
$E\in\es_+^M$ and $\delta$ such that $\crit(E)<\delta<\lh(E)$ and
$M|\lh(E)\sats$ ``$\delta$ is Woodin as witnessed by $\es$''.
The power set axiom is denoted $\PS$.

Ralf Schindler and John Steel's paper \cite{sile} is very relevant to our considerations here. There they established that tame mice do compute significant fragments of their own iteration strategy. In particular,
their proof shows that if $M$ is a tame mouse modelling ZFC, then there is $\alpha<\om_1^M$ such that $M\sats$ ``$M|\om_1^M$ is above-$\alpha$, $(0,\om_1)$-iterable'' (in fact they show more). Further, their research leading to that paper was ``motivated by the question whether every
iterable tame extender model thinks that there is a well-ordering of $\RR$ which is
ordinal definable from a real parameter'' (see \cite[p.~752, immediately following Theorem 0.2]{sile}).
Although \cite{sile} made significant progress toward answering this question,
the question itself was left unresolved:
$(0,\om_1)$-iterability is not enough to execute  the usual proof that comparisons of countable premice terminate.
Our first result answers the question affirmatively.
The proof owes much to the methods employed in \cite{sile}:

\begin{tm}\label{thm:tame_e_def_from_x}
 Let $M$ be a $(0,\om_1+1)$-iterable tame premouse satisfying ``$\om_1$
exists''. Let $\delta=\om_2^M$ and $\her=\her_{\delta}^M$. Then $\Momone^M$ is definable over  $\her$ from a
real parameter, and
in fact, $\{\Momone^M\}$ is  $\Sigma_2^\her(\{x\})$
for some $x\in\RR^M$.
\end{tm}

In the preceding theorem, if $\om_1^M$ is the largest cardinal of $M$
then $\her_\delta^M$ denotes $\univ{M}$.
Combining
\cite[Theorem 1.1]{V=HODX_pub} and Theorem
\ref{thm:tame_e_def_from_x} above, we have:
\begin{cor}\label{cor:tame_V=HODx}
Let $M$ be a tame, $(0,\om_1+1)$-iterable premouse such that
$\univ{M}\sats\ZFC$.
Then $\univ{M}\sats$ ``$V=\HOD_{x}$ for some $x\in\RR$''.
\end{cor}

We also use a variant of the proof to yield
some information regarding grounds of tame mice,
relating to a question of Miedzianowski and Goldberg
\cite{gironaconfproblems};
see Theorem \ref{tm:tame_grounds}.

In \S\S\ref{sec:HOD_tame_mice},\ref{sec:HOD_in_non-tame} we prove some facts
regarding
$\HOD^{\univ{M}}$,
for mice $M$ satisfying $\ZFC$. (In fact, the assumption that $M\sats\ZFC$
is only stated in order the the usual definition of $\HOD$ works.
The results
do not depend
very strongly on this.)

Let $N=L[M_1^\#]$, which is of course tame. It is well known that
$N\sats$ ``$V\neq\HOD$'' (see
\ref{rem:L[M_1^sharp]}). Therefore $\es^N$ is not definable over $\univ{N}$
without parameters.
However, we will show that any tame mouse $M$ satisfying $\PS$ can
\emph{almost} define $\es^M$
from no parameters.
In the statements of the next two theorems, \emph{tractability}  and \emph{strong tractability}
are just fine structural
requirements, which hold
  if $M\sats$ ``$\om_2$ exists''
(see \ref{dfn:tractable}).
And $\mathscr{P}^M$
(see \ref{dfn:candidate,Ppp})
is roughly the collection of premice $N\in M$ such that $\univ{N}=\HC^M$ and $N$ ``eventually agrees''
with $M|\om_1^M$.
\begin{tm}\label{thm:E_almost_def_tame_L[E]}
  Let $M$ be a $(0,\om_1+1)$-iterable tame tractable premouse satisfying
``$\om_1$ exists'' and $\delta=\om_2^M$ \tu{(}with $\delta=\OR^M$
if $\om_1$ is the largest cardinal of $M$\tu{)}.
  Then:
  \begin{enumerate}[label=--]
   \item $\es^M\rest[\om_1^M,\OR^M)$ is $\univ{M}$-definable without parameters,
  and
   \item $\mathscr{P}^M$ is definable
   over $\her_\delta^M$ without
parameters.
   \end{enumerate}
\end{tm}

The results above concern tame mice. We now turn to (short-extender)
mice in general with no smallness restriction. All of our results here rely on
a technical
hypothesis, STH ($\star$-translation hypothesis, see
\ref{dfn:Q^*_when_Q_correct}),
which is almost proved in \cite{closson}, but not quite, and
which should be routine to verify with basically
the methods of \cite{closson}.
We give the key definitions in \S\ref{sec:star},
but a proof of STH is beyond the scope of this paper.
Many typical $\varphi$-minimal mice are \emph{transcendent}
(see \ref{dfn:transcendent}), including
for example $M_1^\#$, and assuming STH,
$M_{\wlim}^\#$ (the sharp
for a Woodin limit of Woodins), the least mouse with an active
superstrong extender (in MS-indexing, so this is not $0^\#$),
and many more.

\begin{tm}\label{tm:V=HOD_in_transcendent_mice}
Assume STH.
 Let $M$ be a transcendent strongly tractable $(\om,\om_1+1)$-iterable $\om$-sound premouse
 such that $\rho_\om^M=\om$ and
 $M\sats$ ``$\om_1$ exists''.
 Let $\delta=\om_2^M$.
 Then $\Momone^M$ is definable without parameters over $\her_\delta^M$.
 Therefore if $N\pins M$ with $M|\om_2^M\ins N$ and $N\sats\PS$
 or $N\sats\ZFC^-$,
 then $\es^N$ is definable without parameters
 over $\univ{N}$, so if $N\sats\ZFC$ then $\univ{N}\sats$ ``$V=\HOD$''.
\end{tm}

We finally consider the question of the structure of $\HOD^{L[\es]}$.
Our results here only give information ``above
$\delta$''
where $\delta=\om_2^{L[\es]}$
if $L[\es]$ is tame and
$\delta=\om_3^{L[\es]}$ otherwise. The question of the nature of
$\HOD^{L[\es]}$ below $\delta$
appears to be much more subtle, and  relates to the question of the
nature of $\HOD^{L[x]}$ for a cone of reals $x$.\footnote{See \cite[\S8.2]{lcfd} for
partial results,
and \cite{zhu_iterates}, \cite{a_long_comparison},
 \cite{local_mantles_of_Lx} for possibly related issues.}
For by considering arbitrary mice, we are
including examples like
$L[x]=L[M_n^\#]$.

Before we  state the results, we make a coarser remark.
Let $M$ be a mouse modelling $\ZFC$
 and $\Momone=\Momone^M$.
 By \cite{V=HODX_pub},
 $\univ{M}=\HOD^{\univ{M}}_{\Momone}$.
So letting $H=\HOD^{\univ{M}}$ and $\PP\in H$ be Vopenka for
adding subsets of
$\om_1^M$ (as
computed in
$M$)
 and $G_{\Momone}$ the generic for $\Momone$,
 standard facts on Vopenka forcing give
 \[H[G_{\Momone}]=\HOD_{\Momone}^{\univ{M}}=\univ{M}\]
 (cf.~Footnote \ref{ftn:Vopenka_extension} for some explanation).
 In $M$, there are  only
$\om_3^M$-many
 subsets of $(\her_{\om_2})^M$, so $\card^M(\PP)\leq\om_3^M$.
 In fact, this Vopenka has the $\om_3^M$-cc in $H$, because
 the antichains correspond in $M$ to partitions of $\pow(\om_1)^M$.
 Therefore $M$ and $H$ have the same cardinals $\geq\om_3^M$.
 Therefore $\PP$ is in fact equivalent in $H$
  to a forcing $\sub\om_3^M$. (Actually, arguing as in the proof of
  Lemma \ref{lem:tame_Vopenka},
  one can also prove this directly, and show that there is such a
  $\PP\sub\om_3^M$ which is
definable without parameters over $\univ{M}$.)
In particular, there is $X\in\pow(\om_3)^M$ such that $H[X]=\univ{M}$.
One can ask whether this is optimal. In fact,
it can be somewhat improved:

\begin{dfn}
We say that a premouse $M$ is \emph{below a Woodin limit of Woodins}
iff there is no segment of $M$ satisfying ``There is a Woodin limit of
Woodins''.
\end{dfn}

$\BB_{\measlim,\delta}$ (Definition
\ref{dfn:meas-lim_extender_algebra}) denotes a simple
variant of the extender algebra at $\delta$.
\begin{tm}\label{tm:HOD_non-tame_mouse}
Assume STH.
 Let $M$ be a $(0,\om_1+1)$-iterable premouse satisfying $\ZFC$ with
 $H=\HOD^{\univ{M}}\psub M$, $\delta=\om_3^M$ and
 $\Momone=\Momone^M$. Then:
 \begin{enumerate}
\item\label{item:HOD_non-tame_mouse_H[M|delta]} $H[M|\delta]=\univ{M}$,
\item\label{item:HOD_non-tame_mouse_when_M_below_WLW}  If $M$ is below a Woodin limit of Woodins
then there is $X\sub\om_2^M$ with $H[X]=\univ{M}$.
\end{enumerate}
Moreover, there is $W\sub M$ which is definable over $\univ{M}$  without parameters, such that:
 \begin{enumerate}[resume*]
\item $W$ is a premouse satisfying ``$\delta$ is Woodin''.\footnote{We do not claim that $\delta$ is the least
Woodin of $W$, nor even that $\delta$ is  a cutpoint of $W$.},
\item\label{item:HOD_non-tame_mouse_H=univ(W)[t]} $H=\univ{W}[t]$ where
$t=\Th_{\Sigma_2}^{\her}(\delta)$ and $\her=\her_\delta^M$,
and $t$ is
$(W,\BB_{\measlim,\delta}^W)$-generic.
\item\label{item:W|delta_is_seg_of_Sigma-iterate}  If either:
\begin{enumerate}[label=--]\item $\om_2^M<\om_1$
and $\Sigma=\Sigma_{\Momone^M}$,
or
\item $\Momone^M$ is $(0,\om_2+1)$-iterable
and $\Sigma$ is the unique $(0,\om_2+1)$-strategy for $\Momone^M$,
\end{enumerate}
then $W|\delta$ is a segment of an iterate
of $\Momone^M$ via $\Sigma$.\footnote{\label{ftn:om_2^M<om_2}The proof will also show that $\om_2^M<\om_2$, even without either of the extra assumptions
of part \ref{item:W|delta_is_seg_of_Sigma-iterate}.}
\end{enumerate}
\end{tm}

Assuming further that $M$ is tame,
and that $M$ is pointwise definable for part \ref{item:W|delta_is_seg_of_Sigma-iterate},
 we can state a tighter relationship between $M$ and $W$,
$\om_3^M$ is reduced to $\om_2^M$,
and we get $H[\Momone^M]=\univ{M}$.
But we defer the full statement (see Theorem
\ref{tm:HOD_tame_mouse}).

Some of the methods developed in this paper and \cite{V=HODX_pub} have since become useful in the study of \emph{Varsovian models}; in particular, related  methods have been employed in \cite{vm2_v2}.

The author thanks the organizers of the WWU M\"unster set theory seminar
for providing the opportunity to present
Theorems
\ref{thm:tame_e_def_from_x}, \ref{thm:E_almost_def_tame_L[E]},
\ref{thm:tame_HOD_x}
and  \ref{tm:HOD_tame_mouse}
in  a series of talks there
in the summer semester of 2016.
He thanks the organizers of the UC Irvine conference in
inner model theory 2016,
for the opportunity to present Theorems \ref{thm:tame_e_def_from_x},
\ref{thm:tame_HOD_x} and
a summary of some other results
(handwritten notes
at \cite{irvine_talk_notes}). He would also like to thank the organizers of the
 Leeds Logic Colloquium 2016
(slides at \cite{leeds_talk_slides})
for the opportunity to present work on the topic. However,
the very last theorem presented in the latter talk (on slide 46) was overstated:
the known proof only works with $\om_3^M$ replacing
$\om_2^M$, as it is stated here in Theorem \ref{tm:HOD_non-tame_mouse}.
The author apologises for the oversight.
(Theorems \ref{tm:tame_grounds}
and \ref{tm:V=HOD_in_transcendent_mice}
were found some time later than the results just mentioned.)

\subsection{Conventions and Notation}\label{subsec:terminology}

For a summary of terminology  see
\cite[\S1.3]{fsfni_v4}.  We just mention a few non-standard and key points here.
We deal with premice $M$ with Mitchell-Steel indexing and fine structure,
except that we allow superstrong extenders on the extender sequence $\es_+^M$
and use the modifications to the fine structure explained in
\cite[\S5]{V=HODX_pub}.

Let $M$ be a premouse (possibly proper class).
We say $M\in\pm_n$ iff $M\sats$ ``$\om_n$ exists''.\footnote{Here the ``$\om_n$''
is
not supposed to refer to $\om_n^V$; we just mean
 that $M\sats$ ``There are at least $(n+1)$ infinite cardinals''.}
An \emph{$\om$-premouse}
 is a sound premouse $N$ with $\rho_\om^N=\om$;
an \emph{$\om$-mouse} is an
$(\om,\om_1+1)$-iterable $\om$-premouse.
The \emph{degree} $\deg(N)$ of an $\om$-premouse $N$ is the least $n<\om$
such that $\rho_{n+1}^N=\om$.
If $N$ is an $\om$-mouse, we write $\Sigma_N$ for the unique
$(\om,\om_1+1)$-strategy for $N$.
We write $\Momone^M$ for $M|\om_1^M$.

Suppose $M$ is $k$-sound where $k<\om$.
 We say that $M$ satisfies \emph{$(k+1)$-condensation}
 iff it satisfies the conclusion of
 \cite[Theorem 5.2]{premouse_inheriting}.
Let $\dot{p}\in V_\om\cut\om$ be some fixed constant.
Then for $\rho_{k+1}^M\leq\alpha\leq\rho_k^M$,
  $t^M_{k+1}(\alpha)$ denotes the theory
 given by replacing $\pvec_{k+1}^M$ in
$\Th_{\rSigma_{k+1}}^M(\alpha\cup\{\pvec_{k+1}^M\})$
with $\dot{p}$, and write $t^M_{k+1}=t^M_{k+1}(\rho_{k+1}^M)$.

For a limit length iteration tree $\Tt$
on an $\om$-premouse and  a $\Tt$-cofinal branch $b$,
$Q(\Tt,b)$ denotes
the Q-structure $Q\ins M^\Tt_b$ for $\delta(\Tt)$,
if this exists, and otherwise
$Q=M^\Tt_b$.

\section{Local branch definability}

\begin{lem}\label{lem:Q_computes_b}
Let $\Tt$ be a limit length $\om$-maximal tree on an $\om$-premouse and $b$ a
$\Tt$-cofinal branch
with $M^\Tt_b$ being $\delta(\Tt)$-wellfounded and $Q=Q(\Tt,b)$ wellfounded.
Let $\delta=\delta(\Tt)$, $t=t^Q_{q+1}(\delta)$
where $Q$ is $q$-sound and $\rho_{q+1}^Q\leq\delta\leq\rho_q^Q$,
and $X=\trancl(\{\Tt,t\})$.
Then:
\begin{enumerate}[label=\tu{(}\roman*\tu{)}]
 \item\label{item:b_in_J(X)} $b\in\J(X)$, and
 \item\label{item:b_from_params_unif} $b$ is $\Sigma_1^{\J(X)}(\{\Tt,t\})$,
uniformly in $(\Tt,t)$.
\end{enumerate}
\end{lem}
\begin{proof}
Part \ref{item:b_in_J(X)}: If $Q=M^\Tt_b$ then we can just use standard
calculations using core maps
(done in the codes given by the theory $t$, however)
to find a tail sequence of extenders used along $b$, and hence,
find $b$ itself, from $(\Tt,Q)$.
So suppose $Q\pins M^\Tt_b$, so $\rho_\om^Q=\delta$
and $Q$ is fully
sound.

\begin{case}\label{case:Q_sing_delta} $Q$ singularizes $\delta$.

Let
$f:\theta\to\delta$ be
cofinal in $\delta$,
with
$\theta=\cof^{M^\Tt_b}(\delta)$, $f$ the least such
which is definable over $Q$ (without parameters).
Let $\alpha\in b$ be such that $(\alpha, b]^\Tt$ does not drop and
$\delta\in\rg(i^\Tt_{\alpha b})$
(so $Q,f\in\rg(i^\Tt_{\alpha b})$) and
$\theta<\kappa=\crit(i^\Tt_{\alpha b})$. Let
$i^\Tt_{\alpha b}(\bar{\delta},\bar{f})=(\delta,f)$.
For $\gamma<\theta$, let $\beta_\gamma$ be the least
$\beta\in[\alpha,\lh(\Tt))$ such that $\alpha\leq^\Tt\beta$ and
$(\alpha,\beta]^\Tt$ does not drop
and
$i^\Tt_{\alpha\beta}(\bar{f}(\gamma))=f(\gamma)<\lambda(E^\Tt_\beta)$.
Then $\beta_\gamma\in b$. (Suppose not.
Let $\xi+1\in b$ be least such that $\xi\geq\beta_\gamma$.
Let $\varepsilon=\pred^\Tt(\xi+1)$. So $\alpha\leq^\Tt\varepsilon<\beta_\gamma\leq\xi$,
so by the minimality of $\beta_\gamma$,
\[ \crit(E^\Tt_\xi)=\crit(i^\Tt_{{\varepsilon} b})<\nu(E^\Tt_\varepsilon)\leq\lambda(E^\Tt_\varepsilon)\leq
f(\gamma)<\lambda(E^\Tt_{\beta_\gamma})\leq\lambda(E^\Tt_\xi). \]
But then
$f(\gamma)\notin\rg(i^\Tt_{{\varepsilon} b})$, so
$f(\gamma)\notin\rg(i^\Tt_{\alpha b})$,
a contradiction.)

So $b$ is appropriately computable from $(\Tt,t)$ and the parameter
$(\alpha,\bar{\delta})$.
But if we define another branch $b'$ from $(\Tt,t)$, in the same manner,
but from some other parameter $(\alpha',\bar{\delta}')$, with $\alpha'\notin
b$, then by the Zipper Lemma
\cite[Theorem 6.10]{outline} and variants thereof,
$Q(\Tt,b')\neq Q(\Tt,b)$, and this fact is first-order over $(Q,t)$,
because we can compute the corresponding
theory $t'$ of $Q(\Tt,b')$ by consulting the
theories of
the models along
$b'$. So by demanding
that the selected parameter
results in a Q-structure whose theory agrees with $t$,
we can actually compute the correct $b$
from $(\Tt,t)$ without the extra parameter.
\end{case}

\begin{case}\label{case:Q_not_sing_delta}
$Q$ does not singularize $\delta$.

Let $A\sub\delta$ be definable over $Q$ without parameters,
such that no $\kappa<\delta$ is ${<\delta}$-$A$-reflecting.
Let $C$ be the set of all limit cardinals $\lambda<\delta$ of $Q$ such that for
all
$\kappa<\lambda$,
$\kappa$ is not ${<\lambda}$-$A$-reflecting. Then $C$ is club in $\delta$
because
$Q$ does not singularize $\delta$.
Let $\alpha\in b$ be such that $[\alpha,b)^\Tt$ is non-dropping and $\delta\in\rg(i^\Tt_{\alpha b})$.
Let $i^\Tt_{\alpha b}(\bar{C})=C$. For $\gamma\in C$, let
$\beta_\gamma$ be the least $\beta\in[\alpha,\lh(\Tt))$ such that
$\gamma<\lh(E^\Tt_\beta)$.
Then $\beta_\gamma\in b$. (Suppose not, and let   $\xi\geq\beta_\gamma$ be least
with $\xi+1\in b$.
Let $\varepsilon=\pred^\Tt(\xi+1)$, so $\alpha\leq^\Tt\varepsilon<\beta_\gamma\leq\xi$.
So
\[
\kappa=\crit(E^\Tt_\xi)<\lh(E^\Tt_{\varepsilon})\leq\gamma<\lh(E^\Tt_{\beta_\gamma}
)\leq\lh(E^\Tt_\xi) \]
and $\gamma\leq\nu(E^\Tt_\xi)$ since $\gamma$ is a $Q$-cardinal.
But since $i^\Tt_{\alpha b}(\bar{A})=A$,
we have
\[ i_{E^\Tt_\xi}(A\inter\kappa)\inter\gamma=A\inter\gamma, \]
so by the ISC, restrictions of $E^\Tt_\xi$ witness
the fact that
$\kappa$ is ${<\gamma}$-$A$-strong
in $Q$, so $\gamma\notin C$, contradiction.) So $b$ is computable from $(\Tt,t)$
and the parameter $(\alpha,\bar{\delta})$, and like before, we actually
therefore get a
computation from $(\Tt,t)$ without the extra parameter.
\end{case}

Part \ref{item:b_from_params_unif}: It seems we can't quite uniformly
tell which of
the above three cases holds. But the calculations used in the case that
$Q\pins M^\Tt_b$ still work when $Q=M^\Tt_b$ and $\delta$ is not
$\bfrSigma_k^Q$-Woodin,
but $\rho_{k+10}^Q=\delta$. So our $\Sigma_1$ formula
seeks either some $k<\om$ such that $Q$ is not $k$-sound,
and applies the procedure for when $Q=M^\Tt_b$,
or some $k<\om$ such that $Q$ is $(k+10)$-sound and $\rho_{k+10}^Q=\delta$,
but $\delta$ is not $\bfrSigma_k^Q$-Woodin,
and then uses the procedure for when $Q\pins M^\Tt_b$ (with complexity say
$\bfrSigma_{k+5}$).
We have enough information in some $\Ss_n(X)$ to verify all the relevant
computations,
including that $Q$ is the correct direct limit of certain substructures
appearing along the branch $b$.
This yields the desired uniform computation for \ref{item:b_from_params_unif}.
\end{proof}
\begin{dfn}
 Let $\Tt$ be as above and $Q$ be a (wellfounded) Q-structure for $M(\Tt)$,
 and $t$ as above for $Q$.
 Then $\branch(\Tt,Q)$
 or $\branch(\Tt,t)$ is the unique $\Tt$-cofinal branch $b$ computed from
$(\Tt,Q)$ as above
 (as the output of our $\Sigma_1^{\J(X)}(\{\Tt,Q\})$ procedure)
 if it exists, and is otherwise undefined.
\end{dfn}

\section{Self-iterability and definability}\label{sec:self-it}

We begin with some basic examples which provide some context for the paper.

\begin{tm}\label{thm:1-small_pc_V=HOD}
 Let $M$ be a proper class, $1$-small, $(0,\om_1+1)$-iterable premouse.
 Then $\es^M$ is definable over $\univ{M}$, so $\univ{M}\sats$ ``$V=\HOD$''.
\end{tm}
\begin{proof}
 By \cite[Theorem 3.11(b)]{V=HODX_pub}, it suffices to see that
 $\Momone^M$ is definable over $\univ{M}$. But
  because $M$ is proper class, and trees $\Tt$ on $\Momone^M$ in $M$ are guided
by Q-structures of the form $\J_\alpha(M(\Tt))$, we
get $M\sats$ ``$\Momone^M$ is $(\om,\om_1+1)$-iterable'', so
$\Momone^M$ is outright definable over $\univ{M}$,
and hence so is $\es^M$.
\end{proof}

In particular $M_1\sats$ ``$V=\HOD$'', a fact first proven by Steel, via other
means. On the other hand:
\begin{rem}\label{rem:L[M_1^sharp]}Assume that $M_1^\#$ exists (and is
$(\om,\om+1)$-iterable) and let $N=L[M_1^\#]$. Note that $N$ is an
$(\om,\om_1+1)$-iterable tame premouse. Standard
descriptive set theoretic observations show that
$\univ{N}\sats$ ``$V\neq\HOD$'', and in fact, that $\om_1^{N}$ is
measurable in $\HOD^{\univ{N}}$. (So by Theorem \ref{thm:1-small_pc_V=HOD},
$N$ is the least such proper class mouse.)

For the record, we give the proof that $\om_1^{N}$ is measurable in
$\HOD^{\univ{N}}$.
It suffices to see that $N\sats\Delta^1_2$-determinacy,
for then $N\sats\OD$-determinacy (by Kechris-Solovay
\cite[Corollary
6.8]{lcfd}),
and hence $\om_1^N$ is measurable in $\HOD^{\univ{N}}$
 by the effective version of Solovay's result
 (see \cite[Theorem 2.15]{lcfd}).
 (Further, $\om_2^N$ is Woodin in $\HOD^{\univ{N}}$ by Woodin
 \cite[Theorem
6.10]{lcfd}.)

So let $g\in N$ be $M_1$-generic for $\Coll(\om,\delta)$
 where $\delta$ is Woodin in $M_1$
  (note $\delta^{+M_1}<\om_1^N$, so such a $g$ exists).
  By Neeman
  \cite[Corollary 6.12]{detLR},
  $M_1[g]\sats\Delta^1_2$-determinacy.
  Let $X\in N$ be $\Delta^1_2$, and $\varphi,\psi$
  be $\Pi^1_2$ formulas such that
  \[ X=\{x\in\RR^N\mid N\sats\varphi(x)\}\text{ and }Y=\RR^N\cut
X=\{x\in\RR^N\mid N\sats\psi(x)\}.\]
 Let $\bar{X}=X\inter M_1[g]$ and $\bar{Y}=Y\inter M_1[g]$.
 By absoluteness,
 \[ \bar{X}=\{x\in \RR^{M_1[g]}\mid M_1[g]\sats\varphi(x)\}\text{ and }
\bar{Y}=\{x\in\RR^{M_1[g]}\mid M_1[g]\sats\psi(x)\},\]
 so $\bar{X}$ is $\Delta^1_2$ in $M_1[g]$.
Let $\sigma\in M_1[g]$ be a winning strategy for the game
$\mathcal{G}^{M_1[g]}_{\bar{X}}$.
 The fact that $\sigma$ is winning is a $\Pi^1_2$ assertion (for either player),
 so $\sigma$ is still winning in $N$. This verifies that $N$ satisfies
$\Delta^1_2$-determinacy.

This proof relies heavily on descriptive set theory. Is there an inner model
theoretic proof that $\univ{N}\sats$ ``$V\neq\HOD$''? There \emph{is}
such a proof that $L[x]\sats$ ``$V\neq\HOD$'' for a cone of reals $x$
(assuming $M_1^\#$); see \cite{a_long_comparison}.
 \end{rem}

\begin{rem}
 Note that we have not ruled out the possibility of set-sized mice $N$ which
 model $\ZFC$ and are $1$-small, and such that $N\sats$ ``$V\neq\HOD$''.
 Let $M$ be the least mouse satisfying $\ZFC$ + ``There is a Woodin cardinal''.
 Then $M$ is pointwise definable and $\J(M)$ is sound, $\rho_1^{\J(M)}=\om$
and $p_1^{\J(M)}=\{\OR^M\}$.
 Let $N$ be the least mouse with $M\pins N$ and $N\sats\ZFC$; so
$N=\J_\alpha(M)$
for some $\alpha\in\OR$, and $N\pins M_1$
and $N$ is pointwise definable and $\J(N)$ is sound and $\rho_1^{\J(N)}=\om$.
Then genericity iterations can be used to show that $N\sats$ ``$M$ is not
$(\om,\om_1+1)$-iterable'', and the
 author does not know whether $\univ{N}\sats$ ``$V=\HOD$''.
\end{rem}

\begin{rem}Considering again $N=L[M_1^\#]$, clearly
$\univ{N}\sats$ ``$V=\HOD_x$ for some real
$x$''.
Steel and Schindler showed that if $M$ is a tame mouse satisfying
$\ZFC^-+$``$\om_1$ exists'',
then there is $\alpha<\om_1^M$ such that $M\sats$ ``$\Momone^M$ is
above-$\alpha$,
$(\om,\om_1)$-iterable''. We next show that this
cannot be improved to
``above $\alpha$, $(\om,\om_1+1)$-iterable''. So we cannot use
$(\om_1+1)$-iterability to prove Theorem
\ref{thm:tame_e_def_from_x}.\end{rem}

\begin{dfn}\label{dfn:meas-lim_extender_algebra}
 Working in a premouse $M$, the \emph{meas-lim extender algebra at $\delta$},
written $\BB_{\ml,\delta}$, is the version of the $\delta$-generator extender
algebra at $\delta$
 in which we only induce axioms with extenders $E\in\es^M$ such that $\nu_E$ is
an inaccessible limit of measurable cardinals of $M$.
 And $\BB^{\geq\alpha}_{\ml,\delta}$
 denotes the variant using only extenders $E$ with $\crit(E)\geq\alpha$.
\end{dfn}

\begin{exm}\label{exm:M_1^sharp-closed_reals}
Let $S$ be the least active mouse such that $S|\om_1^S$ is closed under the
$M_1^\#$-operator and let $N=L[S|\om_1^S]$.
Note that $N\sats$ ``I am $\om_1$-iterable'',
and in fact, letting $\Sigma$ be the correct strategy for $N$,
then $\Sigma\rest\HC^N$ is definable over $N$.
We claim that, however,
\[ N\sats\neg\exists\alpha<\om_1\ [\Momone^N\text{ is
above-}\alpha,\ (\om,\om_1+1)\text{-iterable}].\]

For let $P\pins N|\om_1^N$ project to $\om$.
We will construct tree $\Tt\in N$, on $R=M_1(P)$, above $P$,
of length $\om_1^N$, via the correct strategy,
such that $\Tt$ has no cofinal branch in $N$.
Since $M_1^\#(P)\pins N$ and $P$ can be taken arbitrarily high below $\om_1^N$,
this suffices.

Let $\BB=(\BB^{\geq\OR^P}_{\ml,\delta^R})^R$.
We define $\Tt$ by $\es^{N|\om_1^N}$-genericity iteration with respect to $\BB$
(and its images), interweaving short linear iterations
at successor measurables, as follows. Work in $N$.
The tree $\Tt$ will be nowhere dropping.
We define a continuous sequence $\left<\eta_\alpha\right>_{\alpha<\om_1^N}$
where $\eta_\alpha$ is either $0$ or a limit ordinal ${<\om_1^N}$,
and define $\Tt\rest(\eta_\alpha+1)$, by induction on $\alpha$.
Set $\eta_0=0$.
Suppose we have defined $\Tt\rest(\eta_\alpha+1)$ and it is short;
so
\[ i^\Tt_{0\eta_\alpha}(\delta^R)>\delta=\delta(\Tt\rest\eta_\alpha)\]
(where $\delta(\Tt\rest 0)=0$).
Let $G=G^\Tt_{\eta_\alpha}$ be the least \emph{bad} extender
$G\in\es(M^\Tt_{\eta_\alpha})$;
that is, it induces an axiom of $i^\Tt_{0\eta_\alpha}(\BB)$
which is false for $\es^{N|\om_1^N}$ (or set $G^\Tt_{\eta_\alpha}=\emptyset$
if there is no such; in fact there will be one).
By induction, we will have
$\delta\leq\nu_G<\lh(G)$ (assuming $G$ is defined).
By definition of $\BB$, $\nu_G$
is a limit of  $M^\Tt_{\eta_\alpha}$-measurables.

Suppose $\nu_G>\delta$ (or $G$ is undefined). Let $\mu$ be the least
$M^\Tt_{\eta_\alpha}$-measurable
with $\delta<\mu$, and let $D\in\es(M^\Tt_{\eta_\alpha})$
be the (unique) total  measure on $\mu$. Note that $\lh(D)<\lh(G)$, if $G$ is
defined. Let $Q\pins N$ be least such that
$Q=M_1^\#(S)$ for some $S\pins N$ with $\rho_\om^S=\om$ and $\mu<\OR^S$.
Let $\eta_{\alpha+1}=\OR^Q$.
Then $\Tt\rest[\eta_\alpha,\eta_{\alpha+1}+1]$ is given by linearly iterating
with $D$ and its images.

Now suppose instead that $\nu_G=\delta$.
Then we set $\eta_{\alpha+1}=\eta_\alpha+\om$, set $E^\Tt_{\eta_\alpha}=G$,
and letting $\mu$ be the least  $M^\Tt_{\eta_\alpha+1}$-measurable with $\mu>\delta$ and
$D\in\es(M^\Tt_{\eta_\alpha+1})$
the total measure on $\mu$, let $\Tt\rest[\eta_\alpha+1,\eta_{\alpha+1}+1]$
be given by linear iteration with $D$ and its images.

Note that in both cases, because $\mu$ is a successor measurable, this does not
leave any bad extender algebra axioms induced by extenders $G\in\es(M^\Tt_{\eta_{\alpha+1}})$ such that
\[ \delta<\lh(G)<\delta(\Tt\rest\eta_{\alpha+1}).\]

So it is straightforward to see that $\Tt$ is normal and is nowhere dropping.
We set $\Tt=\Tt\rest\eta_\alpha$ where
$\alpha$ is least such that either $\alpha=\om_1^N$ or
$\Tt\rest\eta_\alpha$
is maximal (non-short).
Note that $\Tt\in N$ and $\Tt$ is via the correct strategy,
so it suffices to verify:

\begin{clm*} $\lh(\Tt)=\om_1^N$ and $N$ has no $\Tt$-cofinal branch.\end{clm*}

\begin{proof} Suppose $\lh(\Tt)=\om_1^N$ but $N$ has a $\Tt$-cofinal branch $b$. Note that $\eta_\alpha$ is defined and $\eta_\alpha<\om_1^N$ for every $\alpha<\om_1^N$. Working in $N$,
we do the usual reflection argument, and get an elementary $\pi:M\to N|\gamma$
for some countable $M$ and large $\gamma$, with $\Tt,b\in\rg(\pi)$.
Let $\kappa=\crit(\pi)$. Let $\beta+1=\min(b\cut(\kappa+1))$.
Because $\Tt$ is normal and by the usual proof that genericity iterations
terminate,
it suffices to see that $E^\Tt_\beta=G^\Tt_{\eta_\alpha}$ for some $\alpha$.
So fix $\alpha<\om_1^N$ such that $\beta\in[\eta_\alpha,\eta_{\alpha+1})$. Then, noting that $\crit(E^\Tt_\beta)=\kappa=\eta_\kappa=\delta(\Tt\rest\eta_\kappa)$,
we have $\alpha\geq\kappa$.
But then if $E^\Tt_\beta\neq G^\Tt_{\eta_\alpha}$,
then $E^\Tt_\beta$
is one of the linear iterates of the order $0$ measure $D$ from stage $\alpha$, but then $\crit(D)=\mu>\delta(\Tt\rest\eta_\alpha)\geq\alpha\geq\kappa$, contradiction.

Now suppose that $\lh(\Tt)<\om_1$; then $\Tt=\Tt\rest\eta_\alpha$ is maximal
with some $\alpha<\om_1^N$. Note that $\alpha$ is a limit.
Let $b$ be the correct $\Tt$-cofinal branch, chosen in $V$.
So
\[ i^\Tt(\delta^R)=\delta=\delta(\Tt\rest\eta_\alpha)\text{ is Woodin in
}M^\Tt_b,\]
and $\delta<\om_1^N$. Let $Q$ result from linearly iterating out
the sharp of $M^\Tt_b$.
Then $N|\delta$ is $Q$-generic for $i^\Tt_b(\BB)$,
and since $\alpha$ is a limit ordinal and because of the linear iterations
inserted in $\Tt$,
$N|\delta$ is closed under the $M_1^\#$-operator.
But $\delta$ is regular in $Q[N|\delta]$,
hence regular in $L[N|\delta]$.
This easily contradicts the minimality of $N$.\end{proof}
\end{exm}

\section{Ordinal-real definability in tame mice}\label{sec:tame_V=HOD_x}
In this section we prove some results for tame mice,
including Theorem \ref{thm:tame_e_def_from_x}, which has the consequence
that every tame mouse satisfying $\ZFC$
satisfies ``$V=\HOD_x$ for some real $x$'', and also that every tame
mouse satisfying ``$\om_1$ exists'' satisfies ``there is a wellorder of
$\RR$ definable over $\her_{\om_2}$ from a real parameter'' (the wellorder is
just the canonical one of $\Momone^M$). As mentioned in the introduction, this answers
the (implicit) question of Schindler and Steel from \cite[p.~752]{sile}.
The methods are, moreover, very similar to those of
\cite{sile}.

\begin{dfn}\label{dfn:Sigma_mmalpha}
 For an $\om$-mouse $M$, or for a mouse $M$ satisfying ``$V=\HC$'',
 $\Sigma_M$ denotes the unique $(\om,\om_1+1)$-iteration strategy for $M$.

 Let $M$ be a $(0,\om_1+1)$-iterable premouse satisfying ``$\om_1$ exists''.
Let
$\Momone=\Momone^M$ and $\alpha<\om_1^M$.
 Then
 $\Sigma_{\Momone\Momone\alpha}$
 denotes the restriction of $\Sigma_{\Momone}$
 to above-$\alpha$ trees in $\Momone$ (in particular, the trees in the domain of this strategy have countable length in $M$).
\end{dfn}

\begin{tm}\label{thm:tame_HOD_x}
 Let $M$ be a tame mouse satisfying ``$\om_1$ exists'' and $\Momone=\Momone^M$.
 Then there is an $\alpha<\om_1^M$ such that:
 \begin{enumerate}
  \item
 $\Sigma_{\Momone\Momone\alpha}\in M$; in fact,
 this strategy is definable over $\Momone$ from parameter $\alpha$,
 \item\label{item:om_1_iter_enough} For every sound tame $\om$-premouse $R$
with $M|\alpha\ins R\in M$, if $M\sats$ ``$R$ is above-$\alpha$,
$(\om,\om_1)$-iterable''
then $R\pins\Momone$.
  \end{enumerate}
Therefore, by \cite[Theorem 3.11]{V=HODX_pub}:
\begin{enumerate}[label=--]
 \item
$\Momone$ is definable over $(\her_{\om_2^M})^M$ from the
parameter $M|\alpha$, and
\item if $M\sats\PS$ or $\univ{M}\sats\ZF^-$ then $\es^M$ is definable over
$\univ{M}$ from $M|\alpha$.
\end{enumerate}
\end{tm}
\begin{proof}
By \cite[Theorem 0.2]{sile},\footnote{\label{ftn:sile_hyp}The hypothesis of \cite[Theorem 0.2]{sile} is that ``$L[E]$'' is a ``tame extender model''. That article does not appear to specify exactly what is meant by an ``extender model'', and of course usually the notation ``$L[E]$'' would mean that the model is proper class. But actually, the proof is very local, and does not depend on the model being proper class, and in fact, it works to give what we claim here under our present hypotheses.}
we may fix $\bar{R}\pins M|\om_1^M$
such that $\rho_\om^{\bar{R}}=\om$ and $M\sats$ ``$\Momone^M$ is above-$\OR^{\bar{R}}$
$(0,\om_1)$-iterable'',
as witnessed by the restriction of the correct strategy $\Sigma_{\Momone}$.
That is, $\Sigma_{\Momone\Momone\alpha}\in M$ where
$\alpha=\OR^{\bar{R}}$. Given $R$ such that $\bar{R}\ins R\pins\Momone$,
$\Sigma_R^M$ denotes the restriction of this strategy to trees on $R$.

We say that $(R,S)\in M$ is a \emph{conflicting pair}
iff:
\begin{enumerate}[label=--]
 \item $R$ and $S$ are  tame $\om$-premice,
 \item $\bar{R}\pins R\pins\Momone$ and $\bar{R}\pins S$
 and $R|\om_1^{R}=S|\om_1^{S}$ but $R\neq S$, and
 \item $M\sats$ ``$S$ is $\om_1$-iterable''.
\end{enumerate}
If part \ref{item:om_1_iter_enough} of the theorem fails for every
$\alpha<\om_1^M$, then
note that for every such $\alpha$ there is a conflicting pair $(R,S)$
with $\alpha<\om_1^R=\om_1^S$. However, for the present we just assume that we have
some conflicting pair and work with this,
without assuming that part \ref{item:om_1_iter_enough} fails for every $\alpha$.

So fix a conflicting pair $(R_0,S_0)$.
Let $\Gamma_0$ be an $\om_1^M$-strategy
for $S_0$ in $M$. Working in $M$, we attempt to compare $R_0,S_0$, via
$\Sigma^M_{R_0},\Gamma_0$, folding in extra extenders to ensure that for every
limit stage
$\lambda$ of the comparison, letting
\begin{enumerate}[label=--]
 \item $\delta_\lambda=\delta((\Tt,\Uu)\rest\lambda)$ and
\item $N_\lambda=M((\Tt,\Uu)\rest\lambda)$,
\end{enumerate}
we have that
\begin{enumerate}[label=\tu{(}$*$\arabic*\tu{)}]
 \item $M|\delta_\lambda$ is generic for the meas-lim extender algebra of
$N_\lambda$ at $\delta_\lambda$ and
 \item\label{item:limit_definability} if $N_\lambda$ is not a Q-structure for
$\delta_\lambda$
 then $(\Tt,\Uu)\rest\lambda\sub M|\delta_\lambda$ and $(\Tt,\Uu)\rest\lambda$
 is definable over $M|\delta_\lambda$ from parameters (and therefore so is
$N_\lambda$).\footnote{The statement that $(\Tt,\Uu)\rest\lambda\sub
M|\delta_\lambda$
is to be interpreted that for each $\alpha<\lambda$ we have
$(\Tt,\Uu)\rest\alpha\in M|\delta_\lambda$,
where $(\Tt,\Uu)\rest\alpha$ incorporates all models $M^\Tt_\beta$
and embeddings $i^\Tt_{\beta\gamma}$ for $\beta\leq\gamma<\alpha$,
and the tree structure ${<_\Tt}\rest\alpha$, etc, and likewise for $\Uu$.
The definability condition adds the requirement that the sequence
$\left<(\Tt,\Uu)\rest\alpha\right>_{\alpha<\lambda}$ is definable.}
\end{enumerate}

Note, however, that there need not actually be Woodin cardinals in $R_0,S_0$,
and the trees might drop in model at points.
To deal with this correctly, the folding
in of extenders for genericity iteration (and other purposes)
is done much as in \cite{backcomp},
and also in
\cite[Definition 5.4]{iter_for_stacks}.
We clarify below exactly how this is executed,
along with
 ensuring the definability condition \ref{item:limit_definability}.

 We will define the comparison $(\Tt,\Uu)$ in certain blocks, during some of
which
we fold in short linear iterations.
In order to ensure the definability condition \ref{item:limit_definability}
above,
initially we must linearly iterate to the point in $M$
which constructs $(R_0,S_0)$, and
following certain limit stages $(\Tt,\Uu)\rest(\eta+1)$ ($\eta$ a limit
ordinal) of the comparison,
we will fold in a linear iteration out to a segment of $M$
which constructs $(\Tt,\Uu)\rest(\eta+1)$.
Overall,  we  will define a strictly increasing, continuous sequence
$\left<\eta_\alpha\right>_{\alpha<\om_1^M}$
of ordinals $\eta_\alpha$ such that either $\eta_\alpha=0$ or $\eta_\alpha$ is
a limit,
and simultaneously define
$(\Tt,\Uu)\rest(\eta_\alpha+1)$.

Also, we will define $(\Tt,\Uu)\rest(\eta+1)$ by induction on $\eta<\om_1^M$ (refining the recursive construction of  blocks just mentioned).
Given $(\Tt,\Uu)\rest(\eta+1)$, if this constitutes a successful comparison
(that is, $M^\Tt_\eta\ins M^\Uu_\eta$ or vice versa),
we stop at stage $\eta$ (and will then derive a contradiction via Claim
\ref{clm:comparison_doesnt_terminate} below).
Now suppose otherwise and let $F^\Tt_\eta,F^\Uu_\eta$ be the extenders
witnessing least disagreement
between $M^\Tt_\eta,M^\Uu_\eta$ (as explained below, we might not use these
extenders in $\Tt,\Uu$, however).
We have $F^\Tt_\eta\neq\emptyset$
or $F^\Uu_\eta\neq\emptyset$.
Let $\ell_\eta=\lh(F^\Tt_\eta)$ or $\ell_\eta=\lh(F^\Uu_\eta)$, whichever is
defined, and
$K_\eta=M^\Tt_\eta||\ell_\eta=M^\Uu_\eta||\ell_\eta$.
If $\eta$ is a limit, let
$Q^\Tt_\eta$ be the Q-structure $Q$
for $\delta_\eta$  with $Q\ins M^\Tt_\eta$
(so if $\Tt\rest\eta$ is not eventually only padding, then $Q^\Tt_\eta=Q(\Tt\rest\eta,[0,\eta)_\Tt)$),
and likewise $Q^\Uu_\eta$ the Q-structure $Q$ for $\delta_\eta$ with $Q\ins M^\Uu_\eta$.\footnote{\label{ftn:padding_Q-str}$\Tt$ and $\Uu$ can be padded, but for all $\alpha$, if $E^\Tt_\alpha=\emptyset$ then $E^\Uu_\alpha\neq\emptyset$. It seems it might be that one of $\Tt\rest\eta$ or $\Uu\rest\eta$
consists of  eventually only padding. Say $\Tt\rest\eta$ is eventually only padding.
Then $\Uu\rest\eta$ is not,
so $\delta_\eta=\delta(\Uu\rest\eta)$, and $Q^\Uu_\eta=Q(\Uu\rest\eta,[0,\eta)^\Uu)$. We have $M^\Tt_\eta=M^\Tt_\alpha$ for all sufficiently large $\alpha<\eta$,
and $M^\Tt_\eta|\delta_\eta=M(\Uu\rest\eta)=M^\Uu_\eta|\delta_\eta$.
So $Q^\Tt_\eta$ and $Q^\Uu_\eta$ are still Q-structures for $M^\Tt_\eta|\delta_\eta=M^\Uu_\eta|\delta_\eta$.}
These
exist as $R_0,S_0$ project to $\om$ and are sound.
Also let $Q^\Tt_0=R_0$ and $Q^\Uu_0=S_0$, and let
$N_0=R_0|\om_1^{R_0}=S_0|\om_1^{S_0}$.

We set $\eta_0=0$.  Suppose we have defined $(\Tt,\Uu)\rest(\eta_\alpha+1)$, so
$\eta_\alpha<\om_1^M$
and $\eta_\alpha=0$ or is a limit. We next define $\eta_{\alpha+1}$
and $(\Tt,\Uu)\rest\eta_{\alpha+1}$, and hence
$(\Tt,\Uu)\rest(\eta_{\alpha+1}+1)$.
In the definition we literally assume that we reach no $\eta<\eta_{\alpha+1}$
such that $(\Tt,\Uu)\rest(\eta+1)$ is a successful comparison;
if we do reach such an $\eta$ then we stop the construction there. There are
three cases to consider.

\begin{casetwo}\label{case:Qs_inequal} $Q^\Tt_{\eta_\alpha}\neq
Q^\Uu_{\eta_\alpha}$ (note this holds in case $\alpha=0$).

If $F^\Tt_{\eta_\alpha}$ is defined then
$\lh(F^\Tt_{\eta_\alpha})\leq\OR(Q^\Tt_{\eta_\alpha})$,
 and likewise for $F^\Uu_{\eta_\alpha}$.
 Thus, note that by tameness (or otherwise if $\alpha=0$),
$\delta_{\eta_\alpha}$
 is a strong cutpoint of $Q^\Tt_{\eta_\alpha}$ and of $Q^\Uu_{\eta_\alpha}$.
 Now $\Tt\rest[\eta_\alpha,\infty)$
 will be based on $Q^\Tt_{\eta_\alpha}$ and above $\delta_{\eta_\alpha}$,
 and likewise $\Uu\rest[\eta_\alpha,\infty)$.

 We want to insert a short linear iteration past the point where $M$
 constructs
$Q^\Tt_{\eta_\alpha},Q^\Uu_{\eta_\alpha}$,
and hence  \label{pg:using_limit_def}(by \ref{item:limit_definability} and Lemma \ref{lem:Q_computes_b}), constructs the branches
$[0,\eta_\alpha)_\Tt$ and $[0,\eta_\alpha)_\Uu$, if $\alpha>0$.
 Let $\eta_{\alpha+1}$ be the least limit ordinal $\eta<\om_1^M$ such that
 $Q^\Tt_{\eta_\alpha},Q^\Uu_{\eta_\alpha}\in M|(\eta+\om)$
  (clearly if $\alpha>0$ then
 $\eta_\alpha\leq\delta_{\eta_\alpha}<\eta_{\alpha+1}$).

 Now $(\Tt,\Uu)\rest[\eta_\alpha,\eta_{\alpha+1})$ is given as follows:
 Let $\eta\in[\eta_\alpha,\eta_{\alpha+1})$ and suppose we have defined
$(\Tt,\Uu)\rest(\eta+1)$. Recall that $K_\eta$ was defined above.
 If $K_\eta$ has a ($K_\eta$-total) measurable $\mu>\delta_{\eta_\alpha}$ then
 letting $\mu$ be least such, $E^\Tt_\eta=E^\Uu_\eta$
 is the unique normal measure
 on $\mu$ in $\es^{K_\eta}$.
 Otherwise, $E^\Tt_\eta=F^\Tt_\eta$ and $E^\Uu_\eta=F^\Uu_\eta$.

 Note  that if
$\alpha=0$ then
$\delta_{\eta_\alpha}=\om_1^{N_{\eta_{\alpha+1}}}$,
and if $\alpha>0$ then $N_{\eta_{\alpha+1}}\sats$ ``$\delta_{\eta_\alpha}$
is Woodin'', and in either case,
 $N_{\eta_{\alpha+1}}\sats$ ``there are no measurables or Woodins
 $>\delta_{\eta_\alpha}$''.
So $Q^\Tt_{\eta_{\alpha+1}}=N_{\eta_{\alpha+1}}=Q^\Uu_{\eta_{\alpha+1}}$,
 and (by tameness) $N_{\eta_{\alpha+1}}$ has no extenders inducing meas-lim
extender algebra axioms
 with index in $[\delta_{\eta_\alpha},\delta_{\eta_{\alpha+1}}]$.
\end{casetwo}

\begin{casetwo}\label{case:pc_Woodins} $N_{\eta_\alpha}\sats$ ``There is a
proper
class of Woodins'' (so
$Q^\Tt_{\eta_\alpha}=N_{\eta_\alpha}=Q^\Uu_{\eta_\alpha}$).

By tameness, it follows that $\delta_{\eta_\alpha}$ is a cutpoint (maybe not
strong cutpoint)
of either $M^\Tt_{\eta_\alpha}$ or $M^\Uu_{\eta_\alpha}$.\footnote{That is,
either $\Tt$ or $\Uu$ uses non-empty extenders cofinally below
$\eta_\alpha$; if $\Tt$ does then $\delta_{\eta_\alpha}$
is a limit of Woodins of $M^\Tt_{\eta_\alpha}$, and likewise for $\Uu$.}
Here $(\Tt,\Uu)\rest[\eta_\alpha,\infty)$ will be above
$\delta_{\eta_\alpha}$.

In this case we want to insert a short linear iteration past the point in $M$
which constructs $(\Tt,\Uu)\rest\eta_\alpha$
(we will have $\alpha=\eta_\alpha=\delta_{\eta_\alpha}$
and already have
\[ (\Tt,\Uu)\rest\eta\in M|\eta_\alpha \]
for every $\eta<\eta_\alpha$, but it is not clear that
$(\Tt,\Uu)\rest\eta_\alpha$
is actually definable over $M|\eta_\alpha$, as it is not clear that the branch
choices of $\Uu$ are appropriately definable).

So let $\eta<\om_1^M$ be the least limit ordinal such that
$(\Tt,\Uu)\rest\eta_\alpha\in M|(\eta+\om)$ (we have
$(\Tt,\Uu)\rest\eta_\alpha\in\HC^M$ by assumption).
Note then that
\[ [0,\eta_\alpha)^\Tt,[0,\eta_\alpha)^\Uu\in M|(\eta+\om) \]
by tameness. We set $\eta_{\alpha+1}=\max(\eta,\eta_\alpha+\om)$.

Now $(\Tt,\Uu)\rest[\eta_\alpha,\eta_{\alpha+1})$ is constructed as in the
previous case,
and note that again, $N_{\eta_{\alpha+1}}$ has no measurables
$>\delta_{\eta_\alpha}$.
(Maybe $\delta_{\eta_\alpha}$ itself is measurable.
In order to ensure that we get a useful comparison,
it is important here that we do not iterate at $\delta_{\eta_\alpha}$
itself during the interval $[\eta_\alpha,\eta_{\alpha+1})$.)
\end{casetwo}

\begin{casetwo}\label{case:Qs_match} $Q^\Tt_{\eta_\alpha}=Q^\Uu_{\eta_\alpha}$
and
$N_{\eta_\alpha}\sats$ ``There is not a proper class of Woodins''.

We set $\eta_{\alpha+1}=\eta_\alpha+\om$. Let
$\eta\in[\eta_{\alpha},\eta_{\alpha+1})$
and suppose we have defined $(\Tt,\Uu)\rest(\eta+1)$.
If $\eta=\eta_\alpha$ (and in fact in general),
\begin{equation}\label{eqn:Qs_match}
Q^\Tt_{\eta_\alpha}=Q^\Uu_{\eta_\alpha}\pins K_\eta.\end{equation}
If there is any $E\in\es^{K_\eta}$ such that $\nu_E$ is an $K_\eta$-inaccessible
limit of $K_\eta$-measurables,
and $E$ induces an extender algebra axiom
which is false of $\es^M$, then set $E^\Tt_\eta=E^\Uu_\eta=$
 the least such $E$. Otherwise set $E^\Tt_\eta=F^\Tt_\eta$
and $E^\Uu_\eta=F^\Uu_\eta$. \label{pg:We_will_have}We will have
\[
\OR(Q^\Tt_{\eta_\alpha})=\OR(Q^\Uu_{\eta_\alpha})<\ell_{\eta_\alpha}\leq\ell_{
\eta} \]
basically by line (\ref{eqn:Qs_match}), tameness and since $Q^\Tt_{\eta_\alpha}$
projects to $\delta_{\eta_\alpha}$.\footnote{We will also have $\OR(Q^\Tt_{\eta_\alpha})<\lh(E^\Tt_\eta),\lh(E^\Uu_\eta)$, but this is only established in Claim \ref{clm:Tt,Uu_normal,etc} after completing the definition of $(\Tt,\Uu)$.}
\end{casetwo}

This completes all cases. Of course, limit stages ${<\om_1^M}$ are taken care
of by our strategies.
This completes the definition of the comparison.

\begin{clmtwo}\label{clm:Tt,Uu_normal,etc}
 $\Tt,\Uu$ are normal, and moreover, if $\alpha<\beta$
 and [$E=E^\Tt_\alpha\neq\emptyset$ or $E=E^\Uu_\alpha\neq\emptyset$]
 and [$F=E^\Tt_\beta\neq\emptyset$ or $F=E^\Uu_\beta\neq\emptyset$],
 then $\lh(E)<\lh(F)$.
\end{clmtwo}
\begin{proof}
 This is a straightforward induction.
\end{proof}

\begin{clmtwo}\label{clm:comparison_doesnt_terminate}
 For each $\alpha<\om_1^M$, either $F^\Tt_\alpha$ or $F^\Uu_\alpha$ is defined,
 and hence, we get a comparison $(\Tt,\Uu)$ of length $\om_1^M$.
\end{clmtwo}
\begin{proof}
Suppose not and let $\alpha$ be least such. So $M^\Tt_\alpha=M^\Uu_\alpha$,
since $R_0,S_0$ are both sound and project to $\om$.
So letting
$C=\core_\om(M^\Tt_\alpha)=\core_\om(M^\Uu_\alpha)$,
there is $\beta+1<\lh(\Tt,\Uu)$ such that
 $\beta<^\Tt\alpha$
and letting $\varepsilon=\successor^\Tt(\beta,\alpha]$,
$(\varepsilon,\alpha]_\Tt\inter\dropset^\Tt=\emptyset$
and $C=M^{*\Tt}_\varepsilon\ins M^\Tt_\beta$,
and $E^\Tt_\beta\in\es_+^C$, and likewise
there is $\gamma+1<\lh(\Tt,\Uu)$ such that $\gamma<^\Uu\alpha$
and letting $\eta=\successor^\Uu(\gamma,\alpha]$,
we have $(\eta,\alpha]_\Uu\inter\dropset^\Uu=\emptyset$
and $C=M^{*\Uu}_\eta\ins M^\Uu_\gamma$ and $E^\Uu_\gamma\in\es_+^C$.

Since $R_0\neq S_0$ but $R_0|\om_1^{R_0}=S_0|\om_1^{S_0}$,
we have $C\neq R_0$ and $C\neq S_0$,
so in fact, $C\pins M^\Tt_\beta$ and $C\pins M^\Uu_\gamma$.

Now since $E^\Tt_\beta\in\es_+^C$ and $E^\Uu_\gamma\in\es_+^C$,
but $E^\Tt_\beta$ is the least disagreement between $C$ and $M^\Tt_\alpha$,
and $E^\Uu_\gamma$ is the least disagreement between $C$ and
$M^\Uu_\alpha=M^\Tt_\alpha$,
we must have $\beta=\gamma$ and $E^\Tt_\beta=E^\Uu_\beta$.
Therefore $E=E^\Tt_\beta=E^\Uu_\beta$ was chosen either for genericity
iteration purposes,
or for short linear iteration purposes.
We have $\beta<\alpha$, so either $F^\Tt_\beta$ or $F^\Uu_\beta$ is defined;
suppose $F=F^\Tt_\beta$ is defined. Since this is least disagreement between
$M^\Tt_\beta,M^\Uu_\beta$,
but $C\pins M^\Tt_\beta$ and $C\pins M^\Uu_\beta$,
we have $\OR^C<\lh(F)$. We also have $\lh(E)\leq\OR^C$,
and note that by how we chose $E$, $\nu_E$ is a cardinal of
$K_\beta=M^\Tt_\beta||\lh(F)$.
But $\crit(E^\Tt_{\varepsilon-1})<\nu_E$,
and therefore $E^\Tt_{\varepsilon-1}$ is total over $K_\beta$,
so
\[ M^{*\Tt}_\varepsilon=C\pins M^\Tt_\beta|\lh(F)\ins M^{*\Tt}_\varepsilon,\]
a contradiction.
\end{proof}

Let $N=N_{\om_1^M}=M(\Tt,\Uu)$, so $\OR^N=\om_1^M$.

\begin{clmtwo}\label{clm:no_pc_Woodins} $N\sats$ ``There is not a proper class
of
Woodins''.\end{clmtwo}
\begin{proof} Suppose otherwise. By tameness, we get $\Tt$- and $\Uu$-cofinal branches
$b,c$.
That is, for each $\delta<\om_1^M$ such that $\delta$ is Woodin in $N$,
$\delta$ is a cutpoint of $(\Tt,\Uu)$,
meaning that there is no extender used in $(\Tt,\Uu)$ which overlaps $\delta$.
But then letting $W$ be the set of all such $\delta$,
$b=\bigcup_{\delta\in W}[0,\delta]^\Tt$
is a $\Tt$-cofinal branch, and likewise for $\Uu$.

Now working in $M$, we argue much as in the
usual proof of termination of comparison/genericity iteration, with one extra
observation. \label{pg:simplicity_assumption_no_alpha}
For simplicity, let us assume that
there is no $\alpha$ such that $\OR^M=\alpha+\om$; in the contrary case, one needs some minor refinements of the discussion to follow.
We get some $\lambda<\OR^M$
and some sufficiently elementary $\pi:\bar{M}\to M|\lambda$ with the relevant
objects in $\rg(\pi)$
and $\bar{M}$ countable.\footnote{\label{ftn:without_simplicity_assumption_no_alpha}{For example, if the simplicity assumption failed and  we had instead $\OR^M=\om_1^M+\om$,}
one would instead choose $n<\om$
and some $x\in\Momone$ such that $(\Tt,\Uu)$ is definable from
$x$ over $\Momone$, and let $\bar{M}$ be a countable $\Sigma_n$-elementary hull
of $\Momone$ including $x$.} Let $\kappa=\crit(\pi)$.
Then $\kappa=\eta_\kappa=\delta_{\eta_\kappa}$.
Let
\[ \beta+1=\successor^\Tt(\kappa,\om_1^M)\text{ and
}\gamma+1=\successor^\Uu(\kappa,\om_1^M).\]
Then $E^\Tt_\beta,E^\Uu_\gamma$ are compatible through
$\min(\nu(E^\Tt_\beta),\nu(E^\Uu_\gamma))$.
The usual arguments for termination of comparison/genericity iteration
show that $\beta=\gamma=\kappa$ and $E=E^\Tt_\kappa=E^\Uu_\kappa$
was chosen for short linear iteration purposes, and $\crit(E)=\kappa$.
Since $N\sats$ ``There is a proper class of Woodins'',
$N_\kappa=N|\kappa$ satisfies the same.
But since $\kappa=\eta_\kappa=\delta_{\eta_\kappa}$
and by the rules of choosing $E^\Tt_\kappa$,
we therefore have $\crit(E)>\kappa$,  contradiction.
\end{proof}
Using Claim \ref{clm:no_pc_Woodins} we may fix $\eta^*<\om_1^M$ with $\eta^*$
above all Woodins of $N$.

\begin{clmtwo}\label{clm:Qs_eventually_match} For all limits $\lambda<\om_1^M$
such that $\delta_\lambda>\eta^*$, we have
$Q^\Tt_\lambda=Q^\Uu_\lambda$ and $Q^\Tt_\lambda\pins M^\Tt_\lambda$ and
$Q^\Uu_\lambda\pins M^\Uu_\lambda$.
\end{clmtwo}
\begin{proof}If $Q^\Tt_\lambda\neq Q^\Uu_\lambda$ then comparison would force
us to use some extender
within the Q-structures, and this would mean that $\delta_\lambda$
is Woodin in $N$,
contradicting the choice of $\eta^*$. So $Q^\Tt_\lambda=Q^\Uu_\lambda$.
If say $Q^\Tt_\lambda=M^\Tt_\lambda$ then $M^\Tt_\lambda=
M^\Uu_\lambda$,
contradicting Claim \ref{clm:comparison_doesnt_terminate}.
\end{proof}

\begin{clmtwo}\label{clm:no_cofinal_branches}
 There is no $\Tt$-cofinal branch $b\in M$, and no $\Uu$-cofinal branch $c\in
M$.
\end{clmtwo}
\begin{proof}
 If both such $b$ and $c$ exist in $M$ then we can reach a contradiction
 much as in the proof of Claim \ref{clm:no_pc_Woodins}.

 Now suppose that we have such a branch $b\in M$, but not $c$.
 Let $Q=Q(\Tt,b)$. If $Q=M^\Tt_b$ then working in $M$,
 we can take a hull, and letting $\kappa$ be the resulting critical point,
 with $\eta^*<\kappa$, note that $Q^\Tt_\kappa=M^\Tt_\kappa$, contradicting
 Claim \ref{clm:Qs_eventually_match}.
 So $Q^\Tt_b\pins M^\Tt_b$. We claim that
 \[ \branch(\Uu,Q^\Tt_b)\text{
 yields a }\Uu\text{-cofinal branch }c\in M,\]
 a contradiction.
 For assuming not, again working in $M$ we can take a hull,
 and letting $\kappa$ be the resulting critical point,
 note that
 \[ \branch(\Uu\rest\kappa,Q^\Tt_\kappa)\text{
 does not yield a }\Uu\rest\kappa\text{-cofinal branch,}\]
 contradicting Claim \ref{clm:Qs_eventually_match} and Lemma
\ref{lem:Q_computes_b}.

 If instead we have $c\in M$ but no such $b\in M$, it is symmetric.
\end{proof}

We will now give a more thorough analysis of stages of the comparison,
and how they relate to the Woodins of the final common model $N$
and segments of $M$ which project to $\om$.
Let $\left<\beta_\gamma\right>_{\gamma<\Omega}$
enumerate the Woodin cardinals of $N$ in increasing order,
and let $\beta_\Omega=\om_1^M$.
Let
\[ \alpha_\gamma=\sup_{\gamma'<\gamma}\beta_{\gamma'},\]
so $\alpha_\gamma<\beta_\gamma$
and either $\alpha_\gamma=0$ or $\alpha_\gamma$ is Woodin or a limit of Woodins
in $N$,
and $\alpha_\Omega$ is the supremum of all Woodins of $N$.
We will show below that for each $\gamma$, we have
[if $\gamma>0$ then
$\alpha_\gamma=\eta_{\alpha_\gamma}=\delta_{\eta_{\alpha_\gamma}}$],
and either:
\begin{enumerate}[label=--]
 \item $\gamma=\alpha_\gamma=0$ (and recall that Case \ref{case:Qs_inequal}
 attains at stage $0$ of the comparison) and let $\chi_0=\eta_1$, that is,
$\chi_0$ is the least $\chi$
 such that $(R_0,S_0)\in M|(\chi+\om)$, or
 \item $\gamma$ is a successor (so $\alpha_\gamma=\beta_{\gamma-1}$ is Woodin
in $N$),
 and Case \ref{case:Qs_inequal} attains at stage $\alpha=\alpha_\gamma$ of the
comparison,
 and let $\chi_\gamma=\eta_{\alpha_\gamma+1}$, that is, $\chi_\gamma$
 is the least $\chi$ such that
$Q^\Tt_{\eta_{\alpha_\gamma}},Q^\Uu_{\eta_{\alpha_\gamma}}\in M|(\chi+\om)$, or
 \item $\gamma$ is a limit (so $\alpha_\gamma$ is a limit of Woodins of $N$),
 and Case \ref{case:pc_Woodins} attains at stage $\alpha=\alpha_\gamma$ of the
comparison,
 and let $\chi_\gamma=\eta_{\alpha_\gamma+1}$, that is,
 $\chi_\gamma=\max(\chi,\eta_{\alpha_\gamma}+\om)$
 where $\chi$ is least such that $(\Tt,\Uu)\rest\eta_{\alpha_\gamma}\in
M|(\chi+\om)$.
\end{enumerate}

\begin{clmtwo}\label{clm:thorough_analysis}
Let $\gamma\leq\Omega$. Then we have:
\begin{enumerate}
 \item\label{item:alpha_gamma_closure} $\alpha_\gamma=\eta_{\alpha_\gamma}$ and
if $\gamma>0$ then $\alpha_\gamma=\delta_{\eta_{\alpha_\gamma}}$
\item\label{item:case_for_alpha_gamma} Case \ref{case:Qs_inequal} or Case
\ref{case:pc_Woodins} attains at stage $\alpha_\gamma$
 of the comparison, according to the discussion above,
  \item\label{item:param_projs_to_om} $M|\chi_\gamma$ projects to $\om$,
  and if $\gamma>0$ then $M|\alpha_\gamma$ has largest cardinal $\om$,
\item\label{item:beta_gamma_closure}
$\beta_\gamma=\eta_{\beta_\gamma}=\delta_{\eta_{\beta_\gamma}}$
\item\label{item:case_for_beta_gamma} if $\gamma<\Omega$ then Case
\ref{case:Qs_inequal} attains at stage $\beta_\gamma$
 of the comparison,
\end{enumerate}
and for every limit $\zeta\in[\alpha_\gamma,\beta_\gamma]$,
if $N_\zeta$ is not a Q-structure for $\delta_\zeta$
then:
\begin{enumerate}[resume*]
\item\label{item:zeta_closure} $\zeta=\eta_\zeta=\delta_{\eta_\zeta}$
and if $\zeta>\alpha_\gamma$ then $\zeta>\chi_\gamma$,
\item\label{item:M|zeta_lgcd_om} $M|\zeta\sats\ZFC^-$ and has largest cardinal
$\om$,
\item $(\Tt,\Uu)\rest\zeta\sub M|\zeta$,
\item if $\zeta>\alpha_\gamma$ then
$x=(\Tt,\Uu)\rest(\alpha_\gamma+1)\in M|\zeta$ and
 $(\Tt,\Uu)\rest\zeta$ is definable over
$M|\zeta$ from the parameter
$x$,
\item\label{item:M|zeta_generic} $M|\zeta$ is $N_\zeta$-generic for the
meas-lim extender algebra of $N_\zeta$ at $\zeta$,
\item\label{item:Q_is_P-con_of_projector} $Q^\Tt_\zeta$
is the output of the P-construction (see \cite{sile})
of $M|\xi$ over $N_\zeta$,
where $\xi$ is least such that $\xi\geq\zeta$ and $\rho_\om^{M|\xi}=\om$
(so in fact $\xi>\zeta$),
\item\label{item:Qs_matching_below_beta_gamma} if
$\alpha_\gamma<\zeta<\beta_\gamma$ then $Q^\Tt_\zeta=Q^\Uu_\zeta\pins N$,
\item\label{item:Qs_not_matching_at_beta_gamma} if $\zeta=\beta_\gamma<\om_1^M$
then $Q^\Tt_\zeta\neq Q^\Uu_\zeta$
and $Q^\Tt_\zeta,Q^\Uu_\zeta\nins N$.
\end{enumerate}
\end{clmtwo}
\begin{proof}By induction on $\gamma$, with a sub-induction on $\zeta$.
Also note that if $N_\zeta$ is a Q-structure for $\delta_\zeta$
then $[0,\zeta)_\Tt$ and $[0,\zeta)_\Uu$ are easily
definable from $(\Tt,\Uu)\rest\zeta$.

Note then that parts \ref{item:alpha_gamma_closure}
and \ref{item:case_for_alpha_gamma} follow easily by induction
from parts \ref{item:beta_gamma_closure} and \ref{item:case_for_beta_gamma}
(we have $0=\alpha_0=\eta_{\alpha_0}$ by definition, and
$\delta_0=\om_1^{R_0}$).
Consider part \ref{item:param_projs_to_om}.
If $\gamma=0$ this is just because $R_0,S_0$ are sound and project to $\om$.
Suppose $\gamma>0$. Then $M|\alpha_\gamma$ has largest cardinal $\om$
by induction.
Clearly
$\rho_\om(M|\chi_\gamma)\leq\alpha_\gamma$,
so suppose that
$\rho_\om(M|\chi_\gamma)=\alpha_\gamma$.
Then $\alpha_\gamma=\om_1^{\J(M|\chi_\gamma)}$
and by Lemma \ref{lem:Q_computes_b} we have
\[ (\Tt\rest\alpha_\gamma,b),(\Uu\rest\alpha_\gamma,c)\in\J(M|\chi_\gamma) \]
where $b=[0,\alpha_\gamma)^\Tt$ and $c=[0,\alpha_\gamma)^\Uu$. But then working
inside $\J(M|\chi_\gamma)$, we can
use parts of the proofs of Claims \ref{clm:no_pc_Woodins}
and \ref{clm:no_cofinal_branches}
to reach a contradiction.

Now it suffices to verify parts
\ref{item:zeta_closure}--\ref{item:Qs_not_matching_at_beta_gamma} for each
limit $\zeta\in[\alpha_\gamma,\beta_\gamma]$,
since then parts \ref{item:beta_gamma_closure} and
\ref{item:case_for_beta_gamma}
follow from parts \ref{item:zeta_closure} and
\ref{item:Qs_not_matching_at_beta_gamma}.

If $\zeta=\alpha_\gamma$ then the required facts already hold by induction
if $\gamma$ is a successor
(as then $\beta_{\gamma-1}=\alpha_\gamma$),
and trivially if $\gamma$ is a limit. (In the limit case,  $N_\zeta\sats$ ``There is a proper
class of Woodins'',
so $N_\zeta$ is a Q-structure for itself.)

So suppose $\zeta>\alpha_\gamma$ and that $N_\zeta$ is not a Q-structure for
itself;
that is, $N_\zeta\sats\ZFC$ and $\J(N_\zeta)\sats$ ``$\delta_\zeta=\OR(N_\zeta)$
is Woodin''.
So $N_\zeta\sats$ ``There is a proper class of measurables'',
so note $\zeta=\eta_\varphi$ for some $\varphi>0$,
and since we integrated genericity iteration into $(\Tt,\Uu)$,
part \ref{item:M|zeta_generic} (genericity of $M|\delta_\zeta$) holds, so
$\delta_\zeta$ is regular
in $\J(N_\zeta)[M|\delta_\zeta]$, hence regular in $\J(M|\delta_\zeta)$,
so $M|\delta_\zeta\sats\ZFC^-$.

Let us verify that $(\Tt,\Uu)\rest\zeta\sub M|\delta_\zeta$
and $(\Tt,\Uu)\rest\zeta$ is definable over $M|\delta_\zeta$
from the parameter $x=(\Tt,\Uu)\rest(\alpha_\gamma+1)$.
We have $\chi_\gamma<\delta_\zeta$ because $N_\zeta$ is not a Q-structure for
itself.
So $x\in M|\delta_\zeta$.
But then working in $M|\delta_\zeta$, which satisfies $\ZFC^-$,
we can define $(\Tt,\Uu)\rest\zeta$, because the extender selection algorithm
can be executed in $M|\delta_\zeta$, in particular since we only need to make
$M|\delta_\zeta$ generic in that interval,
and at non-trivial limit stages $\zeta'\in(\alpha_\gamma,\zeta)$ (when
$N_{\zeta'}$
is not a Q-structure for itself) we use the inductively established
fact that $Q^\Tt_{\zeta'}=Q^\Uu_{\zeta'}$ is computed by P-construction
from some proper segment of $M$, and in fact some proper segment of
$M|\delta_\zeta$
(as $Q^\Tt_{\zeta'}\pins N$ in this case).

Now since $M|\delta_\zeta\sats\ZFC^-$ and $\delta_\zeta=\delta_{\eta_\varphi}$,
it follows that $\varphi=\delta_{\eta_\varphi}=\delta_\zeta=\zeta$.
It also follows that $\om$ is the largest cardinal of $M|\zeta=M|\delta_\zeta$,
as otherwise working in $M|\zeta$, which then satisfies ``$\om_1$
exists'',
we can establish a contradiction to termination of comparison/genericity
iteration much as before.

So $(\Tt,\Uu)\rest\zeta\sub M|\zeta$ and both $(\Tt,\Uu)\rest\zeta$ and
$N_\zeta$ are definable from $x$
over $M|\zeta$, and we have extender algebra genericity as stated earlier. So we have established parts \ref{item:zeta_closure}--\ref{item:M|zeta_generic} for $\zeta$,
and are now in a position to form the P-construction of segments of $M$
over $N_\zeta$.

Let $\xi$ be least such that $\xi\geq\zeta$ and $\rho_\om^{M|\xi}=\om$.
Given $\eta\in[\zeta,\xi]$ let $P_\eta$ be the P-construction
of $M|\eta$ over $N_\zeta$, if it exists.

Now if $P_\xi$ exists then it must be a Q-structure for $\zeta$.
For otherwise, we have that $\zeta$ is Woodin in $\J(P_\xi)$,
and $M|\zeta$ is generic for the meas-lim extender algebra of $\J(P_\xi)$ at
$\zeta$,
so $\zeta$ is regular in $\J(P_\xi)[M|\zeta]$,
but then $\zeta$ is regular in $\J(M|\xi)$, contradicting  that
$\rho_\om^{M|\xi}=\om$.

Now suppose there is
$\eta<\xi$ such that $P_\eta$ exists and either projects $<\zeta$
or is a Q-structure for $N_\zeta$.
If $P_\eta$ projects $<\zeta$ then note that $P_\eta=M^\Tt_\zeta$.
But then working in $M|\xi$, noting that $\zeta=\om_1^{M|\xi}$,
we can reach a contradiction as in the proof of Claim
\ref{clm:no_cofinal_branches}.
So $\rho_\om^{P_\eta}=\zeta$ and $P_\eta$ is a Q-structure for $\zeta$.
But here we also reach a contradiction as in the proof of Claim
\ref{clm:no_cofinal_branches}.

It follows that $P_\xi$ exists and $P_\xi=Q^\Tt_\zeta$,
giving part \ref{item:Q_is_P-con_of_projector} for $\zeta$.

Finally, if $\zeta<\beta_\gamma$ then $Q^\Tt_\zeta=Q^\Uu_\zeta\pins N$,
since $\zeta$ is not Woodin in $N$ by assumption;
and if $\zeta=\beta_\gamma<\om_1^M$ then $Q^\Tt_\zeta\neq Q^\Uu_\zeta$ and
hence, $Q^\Tt_\zeta\nins N$,
since if $Q^\Tt_\zeta=Q^\Uu_\zeta$ then Case \ref{case:Qs_match} would attain
at stage $\zeta$
(recall $\alpha_\gamma<\zeta$)
and then we would have $Q^\Tt_\zeta\pins N$, contradicting the fact that
$\beta_\gamma$ is Woodin in $N$.
This establishes parts \ref{item:Qs_matching_below_beta_gamma} and \ref{item:Qs_not_matching_at_beta_gamma} for $\zeta$, completing the induction.
\end{proof}

\begin{clmtwo}\label{clm:not_Q_and_ZFCmin,V=HC}
 Let $\gamma\leq\Omega$ and $\zeta\in(\alpha_\gamma,\beta_\gamma]$ be a limit.
 Then the following are equivalent:
 \begin{enumerate}[label=\tu{(}\roman*\tu{)}]
  \item\label{item:not_Q-structure} $\J(N_\zeta)\sats$ ``$\delta_\zeta$ is
Woodin'' (equivalently, $N_\zeta$ is not a Q-structure for $\delta_\zeta$),
 \item\label{item:ZFCmin_and_closure} $M|\zeta\sats\ZFC^-\wedge\text{``}V=\HC$'' and
$\zeta>\chi_\gamma$,
 \item\label{item:ZFC_min_etc} $M|\zeta\sats\ZFC^-\wedge\text{``}V=\HC$'' and
$\zeta=\eta_\zeta=\delta_{\eta_\zeta}$.
 \end{enumerate}
\end{clmtwo}
\begin{proof}
We have that \ref{item:ZFC_min_etc} implies \ref{item:ZFCmin_and_closure},
because $\eta_{\alpha_\gamma+1}\geq\chi_\gamma$.
And  \ref{item:not_Q-structure} implies \ref{item:ZFC_min_etc}
 by the previous claim. So it suffices to see that \ref{item:ZFCmin_and_closure}
implies \ref{item:not_Q-structure}.

So suppose $\zeta>\chi_\gamma$
and $M|\zeta\sats\ZFC^-$, where $\zeta\in(\alpha_\gamma,\beta_\gamma]$.
Then as in the proof of Claim \ref{clm:thorough_analysis},
and because $M|\zeta\sats$ ``$V=\HC$'',
$(\Tt,\Uu)\rest\zeta\sub M|\zeta$ and $(\Tt,\Uu)\rest\zeta$ is definable
from the parameter $(\Tt,\Uu)\rest(\alpha_\gamma+1)$ over $M|\zeta$
(the fact that $M|\zeta\sats$ ``$V=\HC$'' ensures that at each non-trivial
intermediate limit stage $\zeta'$, the P-construction
computing $Q^\Tt_{\zeta'}$
is performed by a proper segment of $M|\zeta$,
using part \ref{item:Q_is_P-con_of_projector} of Claim
\ref{clm:thorough_analysis}).
So $\delta_\zeta=\zeta$ and $N_\zeta$ is a class of $M|\zeta$.

Now if $\J(N_\zeta)\sats$ ``$\zeta$ is not Woodin''
then $[0,\zeta)_\Tt$ and $[0,\zeta)_\Uu$ are in $M|(\zeta+\om)$,
but $\zeta=\om_1^{M|(\zeta+\om)}$ since $M|\zeta\sats\ZFC^-$, so we can again
run the usual proof
working in $M|(\zeta+\om)$ for a contradiction.
\end{proof}

Write $\Tt_0=\Tt$ and $\Uu_0=\Uu$.
We now enter a proof by contradiction, by assuming
that part \ref{item:om_1_iter_enough} of the theorem fails for every
$\alpha<\om_1^M$.
Then we can fix a conflicting pair $(R_1,S_1)$ with
$\chi_\Omega<\zeta\eqdef\om_1^{R_1}=\om_1^{S_1}$.
So $R_1\pins M$ and $R_1|\zeta=M|\zeta=S_1|\zeta$.
We have
\[ M|\zeta\sats\ZFC^-\wedge \text{``}V=\HC\text{''},\]
so by Claim \ref{clm:not_Q_and_ZFCmin,V=HC}, $N_\zeta$ is not a Q-structure for
itself,
so the conclusions of Claim \ref{clm:thorough_analysis} hold for $\zeta$.

Repeat the foregoing comparison with $(R_1,S_1)$ replacing $(R_0,S_0)$,
producing trees $\Tt_1$ on $R_1$ and $\Uu_1$ on $S_1$.
Continue in this manner, producing a sequence
\[ \left<R_n,\Tt_n,S_n,\Uu_n\right>_{n<\om}. \]
(It is not relevant whether the sequence is in $M$.)

Now $\Tt_1$ is a tree on $R_1$, above $\zeta=\om_1^{R_1}$.
By Claim \ref{clm:thorough_analysis}, $Q^{\Tt_0}_\zeta$ is the output of the
P-construction
of $R_1$ above $N_\zeta$. So we can translate $\Tt_1$ into a tree $\Tt_1'$ on
$Q^{\Tt_0}_\zeta$;
note this tree is above $\zeta$.
So
\[ \Xx_1=\Tt_0\rest(\zeta+1)\conc\Tt_1' \]
is a correct normal
tree on $R_0$, $\zeta$ is a strong cutpoint of $\Xx_1$,
and $\Xx_1$ drops in model at $\zeta+1$, as
$\Xx_1\rest[\zeta,\infty)$ is based on $Q^\Tt_\zeta$
and  $Q^\Tt_\zeta\pins M^\Tt_\zeta$ by Claim \ref{clm:thorough_analysis}.

Continue recursively, defining $\left<\Xx_n\right>_{n<\om}$,
by setting $\zeta_n=\om_1^{R_{n+1}}$,
and translating $\Tt_{n+1}$ (on $R_{n+1}$) into a tree $\Tt_{n+1}'$
on $Q(\Xx_n,[0,\zeta_n)^{\Xx_n})\pins M^{\Xx_n}_{\zeta_n}$ (there
is a natural finite sequence of intermediate translations between
$\Tt_{n+1}$ and $\Tt_{n+1}'$), and setting
\[ \Xx_{n+1}=\Xx_n\rest(\zeta_n+1)\conc\Tt_{n+1}'.\]

Let $\Xx=\liminf_{n<\om}\Xx_n$. Then $\Xx$ is a correct normal tree on $R_0$.
But it has a unique cofinal branch, which drops in model infinitely often,
a contradiction. This completes the proof of the theorem.
\end{proof}

\begin{dfn}
 Let $N$ be a premouse, $\Tt$ a limit length iteration tree on some $M\pins N$,
 and $\OR^M<\delta\leq\OR^N$. We say that $\Tt$ is \emph{$N|\delta$-optimal}
 iff:
 \begin{enumerate}[label=--]
  \item $\delta=\delta(\Tt)$,
  \item $\Tt\sub N||\delta$ and $\Tt$ is definable from parameters over
$N||\delta$, and
  \item $N||\delta$ is generic for the $\delta$-generator meas-lim
extender algebra of $M(\Tt)$.\qedhere
 \end{enumerate}
\end{dfn}

\begin{dfn}\label{dfn:iterability-good}
 Let $N$ be a tame premouse satisfying $\ZFC^-\wedge\text{``}V=\HC$''.
 $\Lambda_{\tame}^N$ ($\tame$ for \emph{tame}) denotes the partial putative
$(\om,\OR^N)$-iteration strategy
$\Lambda$
for $N$, defined over $N$ as follows.
We define $\Lambda$ by induction on the length of trees.
Let $\Tt\in N$. We say that $\Tt$ is \emph{necessary} if $\Tt$ is an iteration
tree via
$\Lambda$, of limit length, and letting $\delta=\delta(\Tt)$,
either
\begin{enumerate}[label=--]\item $M(\Tt)$ is a Q-structure
for itself, or
\item $\Tt$ is $N|\delta$-optimal
and either $\lgcd(N|\delta)=\om$ or
$\lgcd(N|\delta)=\om_1^{N|\delta}$.\footnote{The restriction
on $\lgcd(N|\delta)$ could be reduced,
but it slightly simplifies some considerations, and suffices for our purposes.
Note that it ensures that $\delta$ is a strong
cutpoint of $N$. We will make use of the possibility that $\lgcd(N|\delta)=\om_1^{N|\delta}$ in the proof of Theorem \ref{tm:HOD_tame_mouse}.}
\end{enumerate}
Every $\Tt\in\dom(\Lambda)$ is necessary. Let $\Tt$ be necessary and $\delta=\delta(\Tt)$.
Then $\Lambda(\Tt)=b$ iff $b\in N$ and either $Q(\Tt,b)=M(\Tt)$ or there is
$R\pins N$
such that $\delta$ is a strong cutpoint of $R$ and $Q(\Tt,b)$
is the output of the P-construction of $R$ above $M(\Tt)$.\footnote{That is,
the P-construction $Q$
of $R$ above $M(\Tt)$ is defined,  $\OR^Q=\OR^{R}$
and $Q=Q(\Tt,b)$.}

We say that $N$ is \emph{tame-iterability-good} iff all putative trees via
$\Lambda^N_{\tame}$ have wellfounded models,
and $\Lambda^N_{\tame}(\Tt)$ is defined for all necessary $\Tt$.
\end{dfn}

\begin{rem}\label{rem:Q_existing_for_necessary_trees}
Note that because $N\sats$ ``$V=\HC$'',
every  tree $\Tt$ on $N$ drops immediately to some proper segment,
and $Q(\Tt,b)$ exists for every limit length $\om$-maximal tree $\Tt$ on $N$ and $\Tt$-cofinal branch $b$
with $M^\Tt_b$ wellfounded.
By Lemma \ref{lem:Q_computes_b} (local branch definability),
$\{b=\Lambda(\Tt)\}$ is uniformly
$\Sigma_1^{\J(Q^*)}(\{\Tt\})$,
where $Q=Q(\Tt,b)$ and either $Q^*$ is the least segment of
$N$ such that $\Tt$ is definable from parameters over $Q^*$,
when $Q=M(\Tt)$,
or $Q^*$ is the  segment of $N$
whose P-construction above $M(\Tt)$ is $Q$, when $M(\Tt)\pins Q$.
In particular, $\Lambda$ is $\Sigma_1$-definable over $N$,
and \emph{tame-iterability-good} is
expressed by a first-order formula $\varphi$ (modulo $\ZFC^-$).
\end{rem}
The following lemma is proved as in
\cite[\S1]{sile}:

\begin{lem}\label{lem:Momone^M_is_tame-it-good}
 Let $M$ be a $(0,\om_1+1)$-iterable tame premouse satisfying either $\ZFC^-$
or ``$\om_1$ exists''.
 Then $\Momone^M$ is tame-iterability-good and
$\Lambda_{\tame}^{\Momone^M}\sub\Sigma_{\Momone^M}$.
\end{lem}

Gabriel Goldberg and
Stefan Miedzianowski  asked  \cite{gironaconfproblems} about the
nature of grounds
of mice via specific kinds of forcings,
in particular $\sigma$-closed
and $\sigma$-distributive.
One result in this regard
was established in
\cite[Theorem 12.1]{fsfni_v4}, and we now improve
this for tame mice modelling $\ZFC$:
\begin{tm}\label{tm:tame_grounds}
Let $M$ be a $(0,\om_1+1)$-iterable   tame premouse
modelling $\ZFC$. Then $M$ has no proper grounds
$W$ via forcings $\PP\in W$ with $W\sats$ ``$\PP$ is strategically
$\sigma$-closed''.
\end{tm}
\begin{proof}
 By \cite[Theorem 12.1]{fsfni_v4}, it suffices to see that $\canM=\Momone^M\in W$,
 and of course we already have $\Momone\sub W$.
 So suppose $\canM\notin W$. We will reach a contradiction
 via a slight variant of the construction
 for Theorem \ref{thm:tame_HOD_x}, so we just give a sketch.
Recall that by \cite{sile},
 we can fix $\xi<\om_1^M$
 such that $\Sigma_{\canM\canM\xi}\in M$ (see also Definition \ref{dfn:Sigma_mmalpha} and Theorem \ref{thm:tame_HOD_x}).
 Also, $M$ is the inductive condensation stack of $M$ above $\canM$
 (see \cite[Theorem 3.11, Definition 3.12]{V=HODX_pub}).
So we can fix a name $\dot{\canM}\in W$ for $\canM$ and
 a name $\dot{\Sigma}\in W$ for $\Sigma_{\canM\canM\xi}$,
 and may assume that
 in $W$, $\PP$ forces ``the universe
 is that of a tame premouse $N$ such that $\Momone^N=\dot{\Momone}$
is tame-iterability-good, $\dot{\Sigma}$ is an
above-$\check{\xi}$-$(\om,\om_1)$-strategy
for $\dot{\Momone}$,
$\dot{\Sigma}$ is consistent with
$\Lambda_{\tame}^{\dot{\Momone}}$,
 $N$ is the inductive condensation stack of $N$ above $\dot{\Momone}$,
 $N$ satisfies various first order facts
 established in this paper and elsewhere for tame mice,
 and $\dot{\canM}\notin\check{V}$''.

 Recall $\HC^W=\HC^M$. Work in $W$. Fix a strategy $\Psi$ witnessing that $\PP$ is strategically
$\sigma$-closed.
Pick some
$(p_0,q_0)\in\PP\cross\PP$
 and some conflicting pair $(R_0,S_0)$ such that
 $p_0\forces_\PP$``$R_0\pins\dot{\Momone}$''
 and $q_0\forces_\PP$``$S_0\pins\dot{\Momone}$'' (where \emph{conflicting pair}
 is defined like before, but with $R_0|\xi=S_0|\xi$ and
$\xi<\om_1^{R_0}=\om_1^{S_0}$).
Let $p'_0=\Psi(\left<p_0\right>)$.
 Let $G\cross H$ be $(W,\PP\cross\PP)$-generic
 with $(p_0',q_0)\in G\cross H$.
 Note that $\PP\forces$``$\check{\PP}$ is strategically $\sigma$-closed,
 as witnessed by $\check{\Psi}$''.

 Work in $W[G,H]$.  It follows that $\HC^{W[G,H]}=\HC^{W[G]}=\HC^{W[H]}=\HC^M$,
and therefore
 $\dot{\Sigma}_G,\dot{\Sigma}_H$
 are above-$\xi$-$(\om,\om_1)$-strategies in $W[G,H]$.
Compare $R_0,S_0$ in the manner of the previous proof,
producing trees $(\Tt,\Uu)$,
 via $\dot{\Sigma}_G,\dot{\Sigma}_H$, except that we fold in
$\dot{\canM}_G$-genericity
instead of $\Momone$-genericity. As before, the comparison
lasts $\om_1^M$ stages and $M(\Tt,\Uu)$ has boundedly many Woodins.
Therefore there is some  $\xi_1<\om_1^M$ after which
$\Tt,\Uu$ agree about all Q-structures,
and these Q-structures are given by P-construction
using proper segments of $\dot{\canM}_G$.
Let $\dot{\Tt}_0,\dot{\Uu}_0\in W$ be $\PP\cross\PP$-names for $\Tt,\Uu$. For $\PP$-names $\tau$, let $\tau_{\mathrm{lt}}$
and $\tau_{\mathrm{rt}}$
be the $\PP\cross\PP$-names
for the interpretation of $\tau$
with respect to the left and right projections of the generic filter, respectively.

Work in $W$. Let $(p_0'',q_0'')\leq(p_0',q_0)$ be such that
$(p_0'',q_0'')\forces_{\PP\cross\PP}$ ``$\dot{\Tt}_0$ is a
tree via $\dot{\Sigma}_{\mathrm{lt}}$ of length $\om_1$,
and for every $\delta\in(\check{\xi}_1,\om_1)$,
if $\dot{\Momone}_{\mathrm{lt}}|\delta\sats\ZF^-\wedge\text{``}V=\HC$''
then $\delta=\delta(\dot{\Tt}_0\rest\delta)$
and $Q=Q(\dot{\Tt}_0\rest\delta,[0,\delta)^{\dot{\Tt}_0})$
is given by P-construction of $Q'$
above $M(\dot{\Tt}_0\rest\delta)$,
where $Q'\pins\dot{\Momone}_{\mathrm{lt}}$ is the least $\om$-premouse
such that $\delta\leq\OR^{Q'}$,
and moreover, $Q\pins M^{\dot{\Tt}_0}_\delta$''. Pick some
$(p_1^-,q_1)\in\PP\cross\PP$
and some $(R_1,S_1)$
such that $p_1^-,q_1\leq p_0''$, and $(R_1,S_1)$ is
a conflicting pair with $R_1|\xi_1=S_1|\xi_1$
and $\xi_1<\om_1^{R_1}=\om_1^{S_1}$
and $p_1^-\forces_\PP$``$\check{R_1}\pins\dot{\Momone}$''
and $q_1\forces_\PP$``$\check{S_1}\pins\dot{\Momone}$''.
So $(p_1^-,q_0'')\forces_{\PP\cross\PP}$``With $\delta=\om_1^{\check{R_1}}$,
and $Q'\pins\dot{\canM}_{\mathrm{lt}}$ as above, then $Q'=\check{R_1}$'', and likewise
$(q_1,q_0'')$ and $S_1$.
Let us now extend $p_1^-$, so as to instantiate some of these (countable) objects in $W$.
Let $\delta=\om_1^{R_1}=\om_1^{S_1}$.
Let $(p_1,q_0''')\leq(p_1^-, q_0'')$
and $\bar{\Tt}_0\in\HC^W$
be such that $\bar{\Tt}_0$ is an above-$\om_1^{R_0}$ tree on $R_0$ of length $\delta+1$ with $\delta=\delta(\bar{\Tt}_0\rest\delta)$, and
$(p_1,q_0''')\forces_{\PP\cross\PP}$ ``$\check{\bar{\Tt}}_0\ins\dot{\Tt}_0$''. Note then that $Q=Q(\bar{\Tt}_0\rest\delta,[0,\delta)^{\bar{\Tt}_0})$ is given by the P-construction
of $R_1$ above $M(\bar{\Tt}_0\rest\delta)$,
and moreover, $Q\pins M^{\bar{\Tt}_0}_\delta$, and $p_1\forces_{\PP}$ ``$\bar{\Tt}_0$ is via $\dot{\Sigma}$''.
Let $p_1'=\Psi(p_0,p_0',p_1)$.

Carry on in this way, much as before, but
also producing the sequence $\left<p_n,p_n'\right>_{n<\om}$
via $\Psi$. We can then find $p_\om\in\PP$ with $p_\om\leq p_n$ for all
$n<\om$.
Much as in the proof of Theorem \ref{thm:tame_HOD_x}, we can extend $\bar{\Tt}_0$ to a tree $\bar{\Tt}_\om\in\HC^W$ such that  $p_\om\forces_{\PP}$ ``$\check{\bar{\Tt}}_\om$ is via $\dot{\Sigma}$, but the only cofinal branch of $\check{\bar{\Tt}}_\om$ has infinitely many drops'',
a contradiction.
\end{proof}

\section{Candidates and their extensions}\label{sec:candidates}

We now prepare for the proof of Theorem \ref{thm:E_almost_def_tame_L[E]}.
The proof will use a combination of the methods of the previous section
with those of \cite{V=HODX_pub}. But nothing in this section
requires tameness, and what we establish will also be used
in  \S\ref{sec:HOD_in_non-tame}.

\begin{dfn}\label{dfn:candidate,Ppp}
 Let $M\in\pm_1$.
We say that $\canN$ is an \emph{$M$-candidate} iff  $\canN\in M$, $\canN$ is a
premouse with $\univ{\canN}=\HC^M$,
and every initial segment of $\canN$ satisfies $(k+1)$-condensation
for every $k<\om$. Let $\canP,\canN$ be $M$-candidates and $\alpha<\om_1^M$.
 We say that $\canP,\canN$ \emph{converge at $\alpha$} iff:
 \begin{enumerate}[label=--]
  \item $\univ{\canP|\alpha}=\univ{\canN|\alpha}$ (hence
$\om_1^{\canP|\alpha}=\om_1^{\canN|\alpha}$),
   \item $\canP|\alpha,\canN|\alpha$ are inter-definable
from parameters (that is, $\es_+^{\canP|\alpha}$ is definable over
$\canN|\alpha$ from
parameters
and likewise $\es_+^{\canN|\alpha}$ over $\canP|\alpha$),
 \item $\rho_\om^{\canP|\alpha}\leq\om_1^{\canP|\alpha}$ (and note
$\rho_\om^{\canN|\alpha}=\rho_\om^{\canP|\alpha}$).
\end{enumerate}
We say that $\canP,\canN$ \emph{$\om$-converge at $\alpha$}
iff $\canP,\canN$ converge at $\alpha$ and $\rho_\om^{\canP|\alpha}=\om$.

Let $\canP,\canN$ be $M$-candidates. We write $\canP\sim_\alpha \canN$ iff
$\canP,\canN$
converge at $\alpha$ and
\[ \es^\canP\rest(\alpha,\om_1^M)=\es^\canN\rest(\alpha,\om_1^M).\]
Let $\mathscr{P}^M=\{\canN\in M\bigm|\canN\text{ is an }M\text{-candidate
and }\exists\alpha<\om_1^M\ [\canN\sim_\alpha\Momone^M]\}$.
\end{dfn}

Note that if $M\in\pm_1$ then $\mathscr{P}^M\in M$,
and for each $N\in\mathscr{P}^M$ , we have $\univ{N}=\HC^M$
and $N$ is $\Sigma_1$-definable from parameters over $\Momone^M$.

\begin{dfn}\label{dfn:cs}
Let $M\in\pm_1$ with either
$\univ{M}\sats\PS$ or $\univ{M}\sats\ZFC^-$. Work in $\univ{M}$ and let $\canN$ be a
candidate.
If the inductive condensation stack $S$ above $\canN$
(see
\cite[Definition 3.12]{V=HODX_pub}) has universe $V$,
then we define $\cs(\canN)=S$; otherwise $\cs(\canN)$ is undefined.
\end{dfn}

\begin{dfn}\label{dfn:standard_condensation}A sound premouse $N$ satisfies \emph{standard condensation}
iff  $N$ satisfies $(n+1)$-condensation for every $n<\om$.\end{dfn}

\begin{lem}\label{lem:tame_es_above_om_1_def_from_Ppp}
Let $M\in\pm_1$ be $(0,\om_1+1)$-iterable, such that $\univ{M}$ satisfies $\ZFC^-$ or $\PS$.
Then:
\begin{enumerate}
 \item\label{item:elements_of_scrP} For all $\canN\in\mathscr{P}^M$,
$\cs(\canN)^M$
 is well-defined, so has universe $\univ{M}$,  the proper segments of
$\cs(\canN)$
 satisfy standard condensation,
 and
 $\es^{\cs(\canN)}\rest[\om_1^M,\OR^M)=\es^M\rest[\om_1^M,\OR^M)$.
\item\label{item:es_equiv_above_om_1} $\es^M\rest[\om_1^M,\OR^M)$ is definable over $\univ{M}$
from the parameter $\mathscr{P}^M$.
\end{enumerate}
\end{lem}
\begin{proof}
Part \ref{item:elements_of_scrP}:
 Let $\canN\in\mathscr{P}^M$. Work in $\univ{M}$.

Let $\alpha<\om_1$ be such that $\canN\sim_\alpha\Momone^M$  and $\rho_\om^{M|\alpha}=\om$ (such an $\alpha$ exists because if $\canN\sim_\beta\Momone^M$ then $\canN\sim_\alpha\Momone^M$ for each $\alpha\in[\beta,\om_1)$).
So $\canN|\alpha$ and $M|\alpha$ project to $\om$ and are inter-definable from
parameters.
Fix a real $x$ coding the pair $(\canN|\alpha,M|\alpha)$, $x$ definable over
$M|\alpha$.
Let $M_x,\canN_x$ be the translations of $M,\canN$ to $x$-premice.
Then  $\canN_x=M_x|\om_1$. So by the relativization to $x$ of
\cite[3.11,
3.12]{V=HODX_pub},
$\cs(\canN_x)$ is defined and has universe $V$. (If $\univ{M}$ has a largest cardinal then $\univ{M_x}\sats\ZFC^-$ by assumption, which implies that $M_x$ is \emph{tractable} in the sense of \cite[3.10]{V=HODX_pub} (relativized to $x$),
as it satisfies clause (vi) of that definition; note that property is coarse, just dependent on $\univ{M_x}=\univ{M}$,
not $\es^{M_x}$.)
In fact $\cs(\canN_x)=\cs(M_x|\om_1)=M_x$,
\[ \es^{\cs(\canN_x)}\rest[\om_1,\OR^M)=\es^M\rest[\om_1,\OR^M),\]
and since $\cs(\canN_x)=M_x$ is iterable (as an $x$-mouse),
its proper segments satisfy standard condensation
(for $x$-mice).

Let $\widetilde{\canN}$ be the translation of $\cs(\canN_x)$ to a standard
premouse
extending $\canN|\alpha$.
So $\widetilde{\canN}$ has universe $V$. We claim that
$\widetilde{\canN}=\cs(\canN)$. Most of the defining properties for
$\cs(\canN)$ (see
\cite[3.12]{V=HODX_pub}) just carry over from $\cs(\canN_x)$.
However, some of the required properties are not quite immediate,
because we can have hulls of segments of $\widetilde{\canN}$
which do not include $x$ in them, so do not correspond to hulls of  segments of
$\cs(\canN_x)$.

So let $R\pins\widetilde{\canN}$ and let $\bar{R}$ be countable
and $\pi:\bar{R}\to R$ be elementary. We claim that there is some $S\pins
\canN$
and $\sigma:\bar{R}\to S$ such that $\sigma$ is elementary. \label{pg:assume_om_1<=OR^R}For we may assume that $\om_1\leq\OR^R$, since otherwise $S=R$ works.
Let $R_x$ be the translation of $R$ to an $x$-premouse.
Let $\bar{R}^+_x$ be countable and $\pi^+_x:\bar{R}^+_x\to R_x$
be elementary (with respect to the language of $x$-premice) with $\rg(\pi)\sub\rg(\pi^+_x)$.
So there is some $S_x\pins \canN_x$ and $\sigma_x^+:\bar{R}^+_x\to S_x$
which is elementary (for $x$-premice),  and so  $x\in\rg(\sigma_x^+)$.
Let $S\pins \canN$ be the translation of $S_x$ to a standard premouse.
Let $\tau:\bar{R}\to\bar{R}^+_x$ be $\tau=(\pi_x^+)^{-1}\com\pi$.
Then $\sigma^+_x\com\tau:\bar{R}\to S$ is elementary (with respect to the language of standard premice), as desired.

Standard condensation for proper segments of $\widetilde{\canN}$ (which is
used both in the proof that $\widetilde{\canN}=\cs(\canN)$,
and also otherwise for part
\ref{item:elements_of_scrP}) now follows
easily: supposing $R\pins\widetilde{\canN}$ fails some
condensation fact,
let $\pi:\bar{R}\to R$ be elementary with $\bar{R}$ countable
and $\pi,\bar{R}$ in $M$, and let $S\pins \canN$ and $\sigma\in \canN$
with $\sigma:\bar{R}\to S$ elementary. Then the failure of condensation
reflects into $S$, contradicting our assumptions about $\canN$. This completes the proof of part \ref{item:elements_of_scrP}.

Part \ref{item:es_equiv_above_om_1}: This follows immediately from part \ref{item:elements_of_scrP}.\end{proof}

By the lemma, to prove Theorem \ref{thm:E_almost_def_tame_L[E]},
it suffices to see that $\mathscr{P}^M$ is definable over $(\her_{\om_2})^M$
without parameters.
For this we will use a comparison argument very much like that of the proof of
Theorem \ref{thm:tame_e_def_from_x}.

\begin{dfn}\label{dfn:tractable_pre}
Let $P\in\pm_1$ with $P\sats$ ``$\om_1$ is the largest cardinal''. We say that
$P$ satisfies
\emph{$(1,\om_1)$-condensation}
iff for every premouse $\Pbar$ with $\eta=\om_1^\Pbar<\om_1^P$,
if $\Pbar$ is $\eta$-sound and $\rho_1^\Pbar\leq\eta$
and $\pi:\Pbar\to P$ is a near $0$-embedding with
 $\crit(\pi)=\eta=\om_1^\Pbar$ (so $\pi(\eta)=\om_1^P$)
then $\Pbar\pins P$.
\end{dfn}

\begin{dfn}\label{dfn:Jensen_ext}
Let $M\in\pm_1$. Work in
$M$.
Let $\canN$ be a candidate. A \emph{Jensen extension} of $\canN$
is a sound premouse $\canN'$ such that:
\begin{enumerate}[label=--]
\item $\canN\ins\canN'$,
\item there is $k<\om$ such that
$\rho_{k+1}^{\canN'}=\om_1^M$ and  $(k+1)$-condensation
holds for $\canN'$, and
\item if $\canN'\sats$ ``$\om_1$ is the largest cardinal'' then $(1,\om_1)$-condensation hold
for $\canN'$.
\end{enumerate}
An \emph{$\Ss$-Jensen extension} of $\canN$
is a structure of the form $\Ss_m(\canN')$,
where $\canN'$ is a Jensen extension of $\canN$ and $m<\om$.
\end{dfn}

\begin{lem}\label{lem:Jensen_stack}
 Let $M\in\pm_1$. Work in
$M$.
 Let $\canN$ be a candidate. Then:
 \begin{enumerate}
 \item\label{item:segs_satisfy_cond} For each Jensen extension $S$ of
$\canN$, all
segments of $S$ satisfy standard condensation.
  \item\label{item:compatibility_Jensen_exts} For all
Jensen extensions $S_0,S_1$ of $\canN$,
either $S_0\ins S_1$ or $S_1\ins S_0$.
 \end{enumerate}
\end{lem}
\begin{proof}
Part \ref{item:segs_satisfy_cond}: All segments of $\canN$ satisfy standard
condensation, as $\canN$ is a candidate.
 But by the assumed condensation for $S$,
 we can reflect segments of $S$ down to segments of $\canN$,
 with a $\Sigma_m$-elementary map, with $m<\om$ arbitrarily high.

Part \ref{item:compatibility_Jensen_exts}: One can run Jensen's standard proof
(for example,
\cite[Fact 3.1]{V=HODX_pub}) inside $M$, unless $M=\J(M')$ for some
$M'$.
In the latter case, we get $S_0,S_1\in\Ss_n(M')$ for some $n<\om$.
But then for any $m<\om$, in $M$ we can form $\Sigma_m$-elementary
substructures
of $\Ss_n(M')$
whose transitive collapse $\bar{\Ss}$ is in $M|\om_1^M$, with the
uncollapse
map
$\pi:\bar{\Ss}\to\Ss_n(M')$ in $M$, and such that
$S_0,S_1\in\rg(\pi)$ and
$\rg(\pi)\inter\om_1^M=\alpha$
for some
$\alpha<\om_1^M$. By condensation, we get a
contradiction as in
Jensen's proof.
\end{proof}

Lemma \ref{lem:Jensen_stack} gives that the stack $\sJs(\canN)$ defined below
is a premouse
extending $\canN$:
\begin{dfn}\label{dfn:strong_can}
Let $M\in\pm_1$.
Work in $M$.
Let $\canN$ be a candidate. Then
$\sJs(\canN)$ denotes the stack of all $\Ss$-Jensen extensions
of $\canN$. We often write $\canN^+=\sJs(\canN)$.
Say $\canN$ is  \emph{strong} if
\begin{enumerate}[label=(\roman*)]
\item $\sJs(\canN)$ has universe $\her_{\om_2}$, and
\item  if $M\sats$ ``$\om_1$ is the largest cardinal''
then $\sJs(\canN)$ satisfies $(1,\om_1)$-condensation.\qedhere
\end{enumerate}
\end{dfn}

\begin{dfn}\label{dfn:tractable}
 A premouse $M$ is \emph{tractable}
if
 $M\in\pm_1$,
 all proper segments of $M$ satisfy standard condensation,
and if $M\sats$ ``$\om_1$ is the largest cardinal''
  then
  \begin{enumerate}[label=--]
  \item $\om<\rho_1^M$ and
  \item $M$ satisfies $(1,\om_1)$-condensation.
  \end{enumerate}

  A premouse $M$ is \emph{strongly tractable}
  if it is tractable and if $M\sats$ ``$\om_1$ is the largest cardinal''
  then $\Hull_1^M(\{x\})$ is bounded in $\OR^M$ for all $x\in M$.
\end{dfn}

\begin{lem}\label{lem:iterable_tract_implies_strong}
If $M$ is an $(0,\om_1+1)$-iterable tractable premouse then
 $\Momone^M$ is a strong candidate in $M$.
\end{lem}
\begin{proof}
By standard condensation facts, $\sJs(\Momone^M)=M|\om_2^M$, which easily implies the lemma.
\end{proof}

\begin{dfn}
 Let $M\in\pm_1$.
 Let $\canP,\canN$ be candidates of $M$.
 Let $\eps<\om_1^M$. We say  $(\canP,\canN)$ \emph{diverges at $\eps$}
 iff there is $\gamma<\eps$ such that
$(\canP,\canN)$ converges at $\gamma$ and $\eps$ is least ${>\gamma}$ such that
$\es^\canP_\eps\neq\es^\canN_\eps$.
We say $(\canP,\canN)$ \emph{$\om$-diverges at $\eps$} iff $(\canP,\canN)$
diverge at $\eps$
and there is $\gamma$ as above such that $(\canP,\canN)$ $\om$-converges at
$\gamma$.
If $(\canP,\canN)$ $\om$-diverges at $\eps$ then $\delta^{\canP,\canN}_\eps$
denotes $\om_1^{\canP|\eps}=\om_1^{\canN|\eps}$ (so by $\om$-divergence,
$\gamma<\delta^{\canP,\canN}_\eps$).
\end{dfn}

Note that if $(\canP,\canN)$ converges at $\gamma$
then $\gamma$ is a strong cutpoint of $\canP,\canN$, and if also
$\es^\canP\rest(\gamma,\eps)=\es^\canN\rest(\gamma,\eps)$,
then $\univ{\canP|\eps}=\univ{\canN|\eps}$ and $\canP||\eps,\canN||\eps$
are inter-definable from $\gamma$ and parameters in $\canP|\gamma$,
uniformly in $\eps$ in a $\Delta_1$ fashion,
and likewise for $\canP|\eps,\canN|\eps$ if also $F^{\canP|\eps}=F^{\canN|\eps}$.
Note also that if $(\canP,\canN)$ diverges at $\eps$ and $\gamma$ is as above,
then $\gamma<\om_2^{\canP|\eps}=\om_2^{\canN|\eps}$ (we have
$\om_2^{\canP|\eps}=\om_2^{\canN|\eps}<\eps$
as either $\canP|\eps$ is active or $\canN|\eps$ is active).

\begin{lem}\label{lem:cofinal_converge_points}
 Let $M$ be a $(0,\om_1+1)$-iterable tractable premouse. Let $\canP\in M$
be a strong
candidate in $M$.
 Then  $(\Momone^M,\canP)$
$\om$-converges at unboundedly many $\gamma<\om_1^M$.
\end{lem}
\begin{proof}
We consider primarily the case that either $M\in\pm_2$
or there is no $M'\pins M$ such that $M=\J(M')$,
and then sketch the modifications needed for the other case.
Write $\canN^+=\sJs^M(\canN)$ for $M$-candidates $\canN$.

Let $\canM_0=\canM=\Momone^M$ and $\canP_0=\canP$. (So
$\canM^+=M|\om_2^M$.) Given $\canM_n,\canP_n$, let
$\canM_{n+1}$ be the least
$\canM'\pins \canM^+$
with $\canM_n\pins \canM'$ and $\canP_n\in \canM'$ and
$\rho_\om^{\canM'}=\om_1^M$;
and define $\canP_{n+1}$ symmetrically from $\canP_n,\canM_n,\canP^+$. Let
$\widetilde{\canM}$ be
the stack of all $\canM_n$,
and $\widetilde{\canP}$
likewise.
Note that $\widetilde{\canM},\widetilde{\canP}$ have the same universe $U$,
and $\widetilde{\canP}$ is $\Sigma_1^U(\{\canP_0\})$ (as $\widetilde{\canP}$ is
the stack
of all Jensen extensions of $\canP_0$ in $U$), and likewise
for $\widetilde{\canM}$ from $\canM_0$, so in particular,
$\widetilde{\canP},\widetilde{\canM}$ are inter-definable from parameters.
Also,
$\widetilde{\canM}\ins M|\om_2^M$.

Note that $\left<\canM_n,\canP_n\right>_{n<\om}$ is also
$\Sigma_1^U(\{(\canM_0,\canP_0)\})$,
so $\rho_1^{\widetilde{\canM}}=\rho_1^{\widetilde{\canP}}\leq\om_1^M$.
In fact $\rho_1^{\widetilde{\canM}}=\om_1^M$, for if $\rho_1^{\widetilde{\canM}}=\om$
then $M=\widetilde{\canM}\sats$ ``$\om_1$ is the largest cardinal'', so by tractability, $\om<\rho_1^M$, a contradiction.

We claim that $\widetilde{\canP}$
is $1$-sound. If $\OR^{\widetilde{\canP}}<\om_2^M$,
then since $\canP$ is a strong candidate,
there is a Jensen extension $\canP'$ of $\canP$ with $\widetilde{\canP}\ins\canP'$,
and Jensen extensions are sound by assumption,
which suffices for this.
So suppose $\OR^{\widetilde{\canP}}=\om_2^M$,
so $\widetilde{\canP}=\sJs(\canP)$, which has universe that of $(\her_{\om_2})^M=\univ{M}$ under these circumstances. Since $\rho_1^{\widetilde{\canP}}=\rho_1^{\widetilde{\canM}}=\om_1^M$
and $H=\Hull_1^{\widetilde{\canP}}(\rho_1^{\widetilde{\canP}}\cup\{p_1^{\widetilde{\canP}}\})$ is cofinal in $\OR^{\widetilde{\canP}}$, it follows that $H=\widetilde{\canP}$. So it suffices to see that $\widetilde{\canP}$ is $1$-solid, so assume $p_1^{\widetilde{\canP}}\neq\emptyset$
and let $\alpha=\max(p_1^{\widetilde{\canP}})$.
Then $\alpha\geq\om_1^M$.
Let $H'=\Hull_1^{\widetilde{\canP}}(\alpha)$.
Then $H'=\widetilde{\canP}|\alpha\in\widetilde{\canP}$,
because otherwise note that $\alpha\in H'$,
contradicting the minimality of $p_1^{\widetilde{\canP}}$. Finally note now that $p_1^{\widetilde{\canP}}=\{\alpha\}$,
since $\alpha+1\sub\Hull_1^{\widetilde{\canP}}(\rho_1^{\widetilde{\canP}}\cup\{\alpha\})$. So $\widetilde{\canP}$ is $1$-solid, as desired.

Let $\eta_0<\om_1^M$ be the least $\eta$ such that
\[
p_1^{\widetilde{\canM}},p_1^{\widetilde{\canP}},
w_1^{\widetilde{\canM}},w_1^{\widetilde{\canP}},\canM_0,\canP_0\in
\Hull_1^{\widetilde{\canM}}(\eta\un\{p_1^{\widetilde{\canM}}\}
)\inter\Hull_1^{\widetilde{\canP}}(\eta\un\{p_1^{\widetilde{\canP}}\})\]
(where $w_1$ denotes the set of $1$-solidity witnesses).
For $\eta\in[\eta_0,\om_1^M)$, note that because of the definability of
$\widetilde{\canP}$ from $\canP_0$ and $\widetilde{\canM}$ from $\canM_0$,
\[ \Hull_1^{\widetilde{\canM}}(\eta\un\{p_1^{\widetilde{\canM}}\})\text{ and
}\Hull_1^{\widetilde{\canP}}(\eta\un\{p_1^{\widetilde{\canP}}\})\text{ have the
same
elements.}\] Let
$H_\eta,H'_\eta$ be the transitive
collapses of these hulls
respectively,
$\pi_\eta:H_\eta\to \widetilde{\canM}$ and
$\pi'_\eta:H'_\eta\to \widetilde{\canP}$
the uncollapse maps. Note $\pi_\eta=\pi'_\eta$ and
$H_\eta,H'_\eta$ have the same
universe and are inter-definable from parameters.
Let $C$ be the set of all $\eta\in[\eta_0,\om_1^M)$ with
$\eta=\om_1^{H_\eta}=\crit(\pi_\eta)$.
Note that if  $\OR^U=\om_2^M$, i.e.~$U=(\her_{\om_2})^M$,
then $M\sats$ ``$\om_1$ is the largest cardinal'', so $\om<\rho_1^M$ by tractability. So
in any case, $C$ is club in $\om_1^M$ (but it seems the ordertype of $C$ might only be $\om$\label{pg:C_low_ot}).
Let $\eta\in C$.
Then $H_\eta,H'_\eta$ are
$\eta$-sound, so by $(1,\om_1)$-condensation,  $H_\eta\pins\Momone$ and
 $H'_\eta\pins \canP$.
So $(\Momone,\canP)$ converges at $\OR^{H_\eta}$.
Now let $\eta<\xi$ be consecutive elements of $C$.
Then $\rho_1^{H_\xi}=\rho_1^{H'_\xi}=\om$, because note that
\[
\xi=\om_1^M\inter\Hull^{\widetilde{\canM}}((\eta+1)\cup\{p_1^{\widetilde{\canM
}}\} ),
\]
so $H_\xi=\Hull_1^{H_\xi}(\{q\})$ where
$\pi_{\xi}(q)=\{\eta,p_1^{\widetilde{\canM}}\}$,
and likewise for $H_\xi'$.
So $(\Momone,\canP)$ $\om$-converges at $\xi$.
Since this holds for cofinally many $\xi<\om_1^M$, we are done.

If instead $M\sats$ ``$\om_1$ is the largest cardinal'' and
$M=\J(M')$ (so
$\rho_\om^{M'}=\om_1^M$) then proceed similarly,
but define $\left<\canM_n,k_n,\canP_n,\ell_n\right>_{n<\om}$
with $k_0=\ell_0=0$ and $\canM_{n+1}$ is the least
$\canM'\pins M$ such that $\canM_n\ins \canM'$ and there is $k<\om$
such that $\canP_n\in\Ss_k(\canM')$ and $k_n,\ell_n<k$, and then let $k_{n+1}$
be the least witnessing $k$, and define $\canP_{n+1},\ell_{n+1}$
symmetrically. Define $\widetilde{\canM},\widetilde{\canP}$ in the obvious
manner
from this sequence (and once again, they have a common universe $U$,
and now $\widetilde{\canP}=\sJs^U(\canP_0)$, etc). Now
proceed much as before.
\end{proof}

\begin{dfn}\label{dfn:ptilde(m)}
 In the above context, let $\widetilde{\canP}(\canM)$ denote
$\widetilde{\canP}$ and $\widetilde{\canM}(\canP)$ denote
$\widetilde{\canM}$.
\end{dfn}

\section{Tail definability of $\es$ in tame mice}\label{sec:tail_def_tame}
For this section and the next, we restrict our attention to tame mice.
\begin{dfn}
 Let $M\in\pm_1$ be tame.
 We say that an $M$-candidate $\canN$  is \emph{tame-good} iff
 $\canN$ is strong and tame-iterability-good (see \ref{dfn:strong_can} and \ref{dfn:iterability-good}) in $M$.
We write $\mathscr{G}_{\tame}^M$,
for the set of tame-good candidates of
$M$. For the most part we abbreviate $\mathscr{G}_{\tame}$
with $\mathscr{G}$.\end{dfn}

\begin{rem}\label{rem:Ggg_def}
$\mathscr{G}^M_{\tame}$
is $\Pi_2^{\her_\delta^M}$ where $\delta={\om_2^M}$ (the definability is <without parameters).
\end{rem}

\begin{lem}\label{lem:def_of_Ppp_in_tame_mice}
Let $M$ be a $(0,\om_1+1)$-iterable tractable tame premouse.
Then $\Momone^M\in\mathscr{G}_{\tame}^M\sub\mathscr{P}^M$.
Therefore $\mathscr{P}^M$ is definable over $\her_{\om_2^M}^M$ without
parameters.
\end{lem}
\begin{proof}
Write $\mathscr{G}^M=\mathscr{G}_{\tame}^M$. The ``therefore'' clause follows
from the rest, as
given any $M$-candidate $\canN$, we get $\canN\in\mathscr{P}^M$
iff $\canN\sim_\alpha\canM$ for some $\canM\in\mathscr{G}^M$ and
 some $\alpha<\om_1^M$.
 And $\Momone^M\in\mathscr{G}^M$
 by Lemmas \ref{lem:iterable_tract_implies_strong} and \ref{lem:Momone^M_is_tame-it-good}.

So let $\canN\in\mathscr{G}^M$;
we show $\canN\in\mathscr{P}^M$. For this, we use a comparison
argument
very much like in the proof of Theorem \ref{thm:tame_HOD_x} (but only its first
round,
which produced the trees $\Tt_0,\Uu_0$ there), so we only
outline enough to explain the differences.

It suffices to see find some $\gamma<\om_1^M$ such that
$\canN,\Momone^M$
$\om$-converge at $\gamma$, and do not diverge at any $\eps>\gamma$.
So suppose we cannot, and for each  $\gamma$ such that $\canN,\Momone^M$ $\om$-converge at $\gamma$, let $\eps'_\gamma$ be the
least
$\eps>\gamma$
such that $(\canN,\Momone^M)$ diverge at $\eps$.
Let $C'$ be the set of all $\gamma<\om_1^M$ such that $(\canN,\Momone^M)$
$\om$-converges at $\gamma$.
By \ref{lem:cofinal_converge_points}, $C'$ is cofinal in $\om_1^M$,
and clearly $0\in C'$.
Define a sequence $\left<\gamma_\alpha\right>_{\alpha<\om_1^M}$
by $\gamma_0=0$,
and given $\left<\gamma_\alpha\right>_{\alpha<\lambda}$
with $\lambda<\om_1^M$, then
$\gamma_{\lambda}$ is the least $\gamma\in C'$ with
$\gamma\geq\sup_{\alpha<\lambda}\eps'_{\gamma_\alpha}$.
So
$\left<\gamma_\alpha\right>_{\alpha<\om_1^M}$ is cofinal in $\om_1^M$.
Now let $\eps_\alpha=\eps'_{\gamma_\alpha}$ and
\[ \delta_\alpha=\om_1^{M|\eps_\alpha}=\om_1^{\canN|\eps_\alpha}.\]
Let $R_\alpha$ be the least $R\pins M$ with $M|\eps_\alpha\ins R$
and $\rho_\om^R=\om$, and $S_\alpha\pins \canN$ likewise.
So
 \[
\gamma_\alpha<\delta_\alpha=\om_1^{R_\alpha}=\om_1^{S_\alpha}<\eps_\alpha\leq
  \OR^{R_\alpha},\OR^{S_\alpha}\leq\gamma_{\alpha+1}.
 \]
Note that $\gamma_0=0$ and $\eps_0$ indexes the least disagreement between
$M,\canN$.

We will define a length $\om_1^M$ comparison/genericity iteration $(\Tt,\Uu)$
of $(R_0,S_0)$,
via $(\Lambda_{\tame}^{\Momone^M},\Lambda_{\tame}^\canN)$,
such that $\left<\delta_\alpha\right>_{0<\alpha<\om_1^M}$
are exactly the Woodin cardinals of $M(\Tt,\Uu)$.
Then as in the proof of \ref{thm:tame_HOD_x},
because $M(\Tt,\Uu)$ has a proper class of Woodins and $M,\canN$ are tame,
we will have $\Tt$-cofinal and $\Uu$-cofinal branches,
and this will give a contradiction.

Given $(\Tt,\Uu)\rest(\alpha+1)$, let
$F^\Tt_\alpha,F^\Uu_\alpha$ be the least
disagreement
between $(M^\Tt_\alpha,M^\Uu_\alpha)$, write $\ell_\alpha=\lh(F^\Tt_\alpha)$
or $\ell_\alpha=\lh(F^\Uu_\alpha)$, whichever is defined,
and $K_\alpha=M^\Tt_\alpha||\ell_\alpha=M^\Uu_\alpha||\ell_\alpha$.

We first define $(\Tt_1,\Uu_1)=(\Tt,\Uu)\rest(\delta_1+1)$;
this will yield $\delta((\Tt_1,\Uu_1)\rest\delta_1)=\delta_1$ and
$M((\Tt_1,\Uu_1)\rest\delta_1)$
will be definable from parameters over $M|\delta_1$, and equivalently,
over $\canN|\delta_1$ (note that $M|\delta_1,\canN|\delta_1$  are
inter-definable
from parameters).

We have $\OR^{R_0},\OR^{S_0}\leq\gamma_1$.
We construct $(\Tt_1,\Uu_1)$ by comparison subject to
folding in meas-lim genericity iteration
and short linear iterations, much as in the proof of \ref{thm:tame_HOD_x}.
Now $(\Tt_1,\Uu_1)$ has two phases. In the first we
fold in a short linear iteration at the least measurable of $K_\alpha$
(that is, if $K_\alpha$ has a least measurable cardinal $\mu$, then we set
$E^{\Tt_1}_\alpha=E^{\Uu_1}_\alpha=$ the least normal measure
on $\mu$, and otherwise $E^{\Tt_1}_\alpha=F^\Tt_\alpha$ and
$E^{\Uu_1}_\alpha=F^\Uu_\alpha$),
until we reach the least $\alpha$ such that $\gamma_1\leq\ell_\alpha$.
In the second phase, we fold in meas-lim extender algebra violations
for making $(\es^M,\es^\canN)$ generic (with the meas-lim requirements
from the perspective of $K_\alpha$,
as in the proof of \ref{thm:tame_HOD_x}). We continue in this manner until
producing
$(\Tt_1,\Uu_1)$ of length $\delta_1$.

At limit stages of $(\Tt_1,\Uu_1)$ (and $(\Tt,\Uu)$ in general) we use
$(\Lambda_{\tame}^{\Momone^M},\Lambda_{\tame}^\canN)$ to select branches.
Thus, we need to verify that this makes sense, i.e.~that the trees at those
stages are necessary.
Note also that $M|\delta_1$ and $\canN|\delta_1$ satisfy $\ZFC^-$, contain
$R_0,S_0,\gamma_1$,
and moreover, $M|\delta_1$ and $\canN|\delta_1$ are inter-definable from
parameters, so the extender selection process just described
is definable from parameters over both.

Let $\lambda\leq\delta_1$ be a limit with $N=M((\Tt,\Uu)\rest\lambda)$
not a Q-structure for itself, and let $\delta=\delta((\Tt,\Uu)\rest\lambda)$.
We claim:
\begin{enumerate}
\item\label{item:M|delta,N|delta_generic} $M|\delta$ and $\canN|\delta$ are
meas-lim extender algebra generic
over $N$ at $\delta$,
 \item\label{item:M|delta,N|delta_sat_ZFC-+V=HC} $M|\delta$ and $\canN|\delta$
satisfy $\ZFC^-+$``$V=\HC$'',
\item $N\sats$ ``There are no Woodin cardinals'',
 \item \label{item:Tt,Uu_def}
 $\lambda=\delta$
and $(\Tt,\Uu)\rest\delta\sub(M|\delta)\inter(\canN|\delta)$ and
$(\Tt,\Uu)\rest\delta$ is definable from parameters over $M|\delta$
and  $\canN|\delta$,
\item\label{item:Q-struc_comp} letting $\gamma\geq\delta$ be least with
$\rho_\om^{M|\gamma}=\om$,
the P-construction $Q$ of $M|\gamma$ over $N$ is defined,
$\OR^Q=\gamma$,
and $Q$ is a Q-structure for $N$; and likewise for
$\canN$
and the least $\gamma'\geq\delta$ with $\rho_\om^{\canN|\gamma'}=\om$,
which yields a Q-structure $Q'$ for $N$,
\item\label{item:Qs_match_below_delta_1} if $\delta<\delta_1$ then $Q=Q'$,
where $Q,Q'$ are as above.
\end{enumerate}
This is by induction on $\lambda$, and much as in the proof of
\ref{thm:tame_HOD_x}. Items \ref{item:M|delta,N|delta_generic} and
\ref{item:M|delta,N|delta_sat_ZFC-+V=HC}
are as there.

Now suppose that $N$ has no Woodins, and we deduce
items \ref{item:Tt,Uu_def}, \ref{item:Q-struc_comp} and
\ref{item:Qs_match_below_delta_1}.
The parameter we need to define the trees is $(R_0,S_0)$,
which we have in the relevant segments of $M,\canN$ because we initially
folded in linear iteration past $\gamma_1$.
As mentioned above, the extender selection process is definable from
parameters over $M|\delta_1$ and $\canN|\delta_1$, and  in fact,
$(\Tt,\Uu)\rest\lambda$ is definable from parameters over $M|\delta$ and
$\canN|\delta$. For because $N$ has no Woodins,
the Q-structures $Q_\xi,Q_\xi'$
used at limit stages $\xi<\lambda$ in $\Tt,\Uu$
to determine $[0,\xi)_\Tt$ and $[0,\xi)_\Uu$ respectively
are identical and are proper segments of $N$. By induction,
these are computed as in item \ref{item:Q-struc_comp},
and the segments of $M,\canN$ used to compute them have height ${<\delta}$,
so $M|\delta,\canN|\delta$ can determine them, and hence $[0,\xi)_\Tt$
and $[0,\xi)_\Uu$. So $M|\delta,\canN|\delta$ can compute
$(\Tt,\Uu)\rest\lambda'$ as long as
$(\Tt,\Uu)\rest\lambda'\sub(M|\delta)\inter(\canN|\delta)$.
But if this fails for some $\lambda'<\lambda$, we contradict the fact that
$M|\delta$ and $\canN|\delta\sats\ZFC^-$. Item \ref{item:Tt,Uu_def} now
follows.

It also follows that $\Tt\rest\delta$ and $\Uu\rest\delta$ are necessary,
so $\Lambda_{\tame}^{\Momone^M}(\Tt\rest\delta)$ and
$\Lambda_{\tame}^\canN(\Uu\rest\delta)$
are
defined,
and the process continues. Let $\gamma$ be as in item \ref{item:Q-struc_comp}.
Let $Q$ be the result of the P-construction of $M$ above
$N$ (recall this stops as soon as it reaches a Q-structure or projects across
$\delta$).
Because $\delta$ is regular in $Q[M|\delta]$, we cannot have $M|\gamma\in
Q[M|\delta]$, so $\OR^Q\leq\gamma$.
But if $\OR^Q<\gamma$ then we reach a contradiction as in the proof of Claim
\ref{clm:no_cofinal_branches} in the proof of \ref{thm:tame_HOD_x}.
So $\OR^Q=\gamma$. It is analogous for $\canN$.

For item \ref{item:Qs_match_below_delta_1}, by item \ref{item:Q-struc_comp} and
by the agreement of $M|\delta_1$ and $\canN|\delta_1$,
if $\delta<\delta_1$ then  $Q=Q'$.
(Note here $\gamma,\gamma'<\delta_1$,
as $M|\delta_1,\canN|\delta_1$ have largest cardinal $\om$.)

It remains to verify that $N$ has no Woodins.
So suppose $N\sats$ ``$\eta$ is
Woodin'' and let $\eta$ be least such.
Then because we have folded in meas-lim genericity iteration,
$M|\eta,\canN|\eta$ are $(N,\BB_{\ml,\eta}^N)$-generic, so
 $M|\eta$ and $\canN|\eta$ satisfy $\ZFC^-$.
 Let $\lambda'<\lambda$ be least such that
$\delta((\Tt,\Uu)\rest\lambda')\geq\eta$. Then note that by $\ZFC^-$
and as before, $M|\eta$ and $\canN|\eta$ can compute $(\Tt,\Uu)\rest\lambda'$,
and we get $\lambda'=\eta$. But since $N|\eta$ has no Woodins,
the preceding applies with $\lambda$ replaced by $\lambda'=\eta<\lambda$.
In particular, $Q_\eta=Q_\eta'$, where these are the Q-structures
determining $[0,\eta)_\Tt,[0,\eta)_\Uu$. Since $\eta$ is Woodin in $N$,
$E^\Tt_\eta$ or $E^\Uu_\eta$ must come from $Q_\eta=Q_\eta'$.
But then $E^\Tt_\eta=E^\Uu_\eta$, so this extender is being used
for linear iteration or genericity iteration purposes, and $Q_\eta\pins
K_\eta$. But $\eta$ is a strong cutpoint of $Q_\eta$, so $E^\Tt_\eta$ causes a
drop in model
to some $P\ins Q_\eta$. But then $E^\Tt_\eta$ is not $K_\eta$-total,
a contradiction.

This completes the induction, giving
$(\Tt,\Uu)\rest(\delta_1+1)$.
Now suppose $\lambda=\delta_1$.  By item \ref{item:Q-struc_comp},
letting $b,c$ be the branches chosen in $\Tt,\Uu$,
 $Q(\Tt,b)$ results from the P-construction
of $R_1$ above $N=M((\Tt,\Uu)\rest\delta_1)$,
and has height $\OR^{R_1}$,
and $Q(\Uu,c)$ is  that of $S_1$ above $N$,
of height $\OR^{S_1}$.
But $\eps_1$ indexes the least disagreement between $R_1,S_1$ above $\delta_1$.
Now
\[ Q(\Tt,b)||\eps_1=Q(\Uu,c)||\eps_1
\text{ but }
 Q(\Tt,b)|\eps_1\neq Q(\Uu,c)|\eps_1.\]
For if $Q(\Tt,b)|\eps_1=Q(\Uu,c)|\eps_1$ then
because $Q(\Tt,b)[M|\delta_1]$ and $Q(\Tt,b)[\canN|\delta_1]$
have the same universe and the forcing  is small
relative to the active extenders, there is a unique possible extension
of the extenders to the extensions, so $R_1|\eps_1=S_1|\eps_1$,
contradiction.

So
the overall comparison now reduces to a comparison of $Q(\Tt,b)$ with
$Q(\Tt,c)$,
and therefore $\delta_1$ will be the least Woodin cardinal, and hence (by
tameness, or in this case, just that $\delta_1$ is the least such Woodin) also a
strong cutpoint of the final model.

Now suppose $\alpha>0$ and we have defined
$(\Tt_\alpha,\Uu_\alpha)$, of length $\delta_\alpha+1$,
and the P-constructions of $R_{\alpha+1},S_{\alpha+1}$ yield
the Q-structures $Q(\Tt\rest\delta_\alpha,b')$ and $Q(\Uu\rest\delta_\alpha,c')$
etc.
We then define $(\Tt_{\alpha+1},\Uu_{\alpha+1})$ extending
$(\Tt_\alpha,\Uu_\alpha)$,
above $\delta_\alpha$, of length $\delta_{\alpha+1}+1$.
Here we again have two stages.
In the first we fold in linear iteration past $\gamma_{\alpha+1}$,
at the least measurable $>\delta_\alpha$,
and in the second we fold in genericity iteration.
Everything is analogous to the case $\alpha=1$
(there are now Woodin cardinals in
$M((\Tt_{\alpha+1},\Uu_{\alpha+1})\rest\lambda)$,
but they are exactly the $\delta_\beta$ for $\beta\leq\alpha$).

Given $\left<\Tt_\alpha,\Uu_\alpha\right>_{\alpha<\eta}$ for a limit $\eta$,
this gives
$(\Tt,\Uu)\rest\lambda$ where $\lambda=\sup_{\alpha<\eta}\delta_\alpha$.
Note $\lambda=\delta((\Tt,\Uu)\rest\lambda)$.
Because $M((\Tt,\Uu)\rest\lambda)$ satisfies ``There is a proper class of
Woodins'' by induction,
it is a Q-structure for itself, so $\Tt\rest\lambda$ and $\Uu\rest\lambda$ are
necessary (as they are in $M$),
and hence
 in the domains of the iteration strategies.
This yields $(\Tt,\Uu)\rest(\lambda+1)$.
We get $M^\Tt_\lambda\nins M^\Uu_\lambda\nins M^\Tt_\lambda$.
Since $\lambda$ is a limit of strong cutpoints of
$M^\Tt_\lambda,M^\Uu_\lambda$,
the comparison now reduces to a comparison of
$M^\Tt_\lambda,M^\Uu_\lambda$,
above $\lambda$. Note that $(\Tt,\Uu)\rest(\lambda+1)$ is definable
from parameters
over $M|\gamma_\eta$, and over $\canN|\gamma_\eta$
(or at least, $(\Tt,\Uu)\rest\lambda$ is definable from parameters
over those segments, and $[0,\lambda)_\Tt,[0,\lambda)_\Uu$
are also, so the models $M^\Tt_\lambda,M^\Uu_\lambda$
are definable ``in the codes'', but might literally have ordinal height
$>\gamma_\eta$).
At this stage we fold in linear iteration past $\gamma_\eta$, at the least
measurable $\mu>\lambda$, if there is such, and then genericity iteration, to
produce $(\Tt,\Uu)\rest(\delta_\eta+1)$
much as before.

This completes the description of the comparison. We produce trees $(\Tt,\Uu)$
of length $\om_1^M$, and $\left<\delta_\alpha\right>_{\alpha<\om_1^M}$
enumerates the Woodins of $M(\Tt,\Uu)$, cofinal in $\om_1^M$.
By tameness, we get $\Tt$-cofinal and $\Uu$-cofinal branches $b,c\in M$
(this doesn't require any further iterability assumptions).
One now reaches a contradiction as in the proof of
Theorem \ref{thm:tame_HOD_x}.
\end{proof}

\begin{proof}[Proof of Theorem \ref{thm:E_almost_def_tame_L[E]}]
 By Lemma \ref{lem:def_of_Ppp_in_tame_mice}, $\mathscr{P}^M$ is definable over
$(\her_{\om_2})^M$ without parameters.
So  by Lemma \ref{lem:tame_es_above_om_1_def_from_Ppp},
$\es^M\rest[\om_1^M,\OR^M)$ is definable over $\univ{M}$ without parameters.
\end{proof}

\section{HOD in tame mice}\label{sec:HOD_tame_mice}

In Theorem \ref{tm:HOD_tame_mouse} we analyse  $\HOD^{L[\es]}$
above $\om_2^{L[\es]}$, for tame $L[\es]$. This uses Vopenka:
\begin{dfn}\label{dfn:Ggg}
 Let $M$ be a $(0,\om_1+1)$-iterable tame premouse satisfying $\ZFC$. Write $\mathscr{G}=\mathscr{G}_{\tame}$ (see \ref{dfn:Ggg}).
 Then $\Vop_{*\mathscr{G}}^M$ denotes the Vopenka forcing
 corresponding to non-empty $\OD^{\univ{M}}$ subsets of $\mathscr{G}^M$,
 coded in the usual manner with ordinals as conditions.
 (Let $\PP_0$ be the forcing whose conditions are non-empty
$\OD^{\univ{M}}$
 subsets of $\mathscr{G}^M$, with $A\leq B$ iff $A\sub B$.
 Then $\Vop_{*\mathscr{G}}^M$ is
  the natural isomorph of $\PP_0$, using standard ordinal codes
 for conditions in $\PP_0$.)
\end{dfn}

\begin{rem}Note that $\Vop_{*\mathscr{G}}^M$ is definable over $\univ{M}$
without parameters. Once
we have proved the following lemma,
 we will define $\Vop_{\mathscr{G}}^M$ as a \emph{more} natural isomorph of
$\Vop_{*\mathscr{G}}^M$, which is a subset of $\om_2^M$, and is definable over
$(\her_{\om_2^M})^M$ without parameters.
\end{rem}

\begin{lem}\label{lem:tame_Vopenka}
 Let $M$ be a $(0,\om_1+1)$-iterable tame premouse satisfying $\ZFC$.
 Let $\PP=\Vop_{*\mathscr{G}}^M$ and $\delta=\om_2^M$.
 Let $H=\HOD^{\univ{M}}$.
 Then:
 \begin{enumerate}
  \item\label{item:Vop_delta-cc} $\PP\in H$ and $H\sats$ ``$\PP$ is a
$\delta$-cc
complete Boolean algebra''.
  \item\label{item:Vop_sub_delta} $\PP\iso$  some $\PP'\sub\delta$
which is $(\Sigma_3\wedge\Pi_3)^{\her_{\delta}^M}$-definable
without parameters,
  \item\label{item:H[G]=H[M|om_1]} There is $G$ which is $(H,\PP)$-generic,
with
$H[G]=H[\Momone^M]$ having universe $\univ{M}$.
  \item\label{item:every_Vop_cond_generic} For every $p\in\PP$ there is an
  $(H,\PP)$-generic
$G'\in M$ such that $p\in G'$ and $H[G']$ has universe $\univ{M}$.
 \end{enumerate}
\end{lem}
\begin{proof}
Part \ref{item:Vop_delta-cc}: We have $\PP\in H$ and $H\sats$ ``$\PP$ is a
complete Boolean algebra''
by the usual proof for Vopenka forcing. We have $H\sats$ ``$\PP$ is $\delta$-cc''
because by Lemma \ref{lem:def_of_Ppp_in_tame_mice}, in $M$, $\mathscr{G}^M$ has cardinality $\leq\om_1^M$,
and all maximal antichains of $\PP$ in $H$ correspond
to partitions of $\mathscr{G}^M$ in $M$.

Part \ref{item:Vop_sub_delta}: A \emph{nice code}
is a triple $(\alpha,\beta,\varphi)$ such that $\alpha<\beta<\delta$
and $\varphi$ is a formula.
The nice code $(\alpha,\beta,\varphi)$
codes the set
\[
A_{\alpha\beta\varphi}=\{\canN\in\mathscr{G}
^M\bigm|\sJs(\canN)|\beta\sats\varphi(\alpha)\}. \]

\begin{clmfour}\label{clm:every_OD_set_has_nice_code}A set $A\sub\mathscr{G}^M$
is $\OD^{\univ{M}}$ iff $A$ has a nice code.\end{clmfour}
\begin{proof}
Each $A_{\alpha\beta\varphi}$ is $\OD^{\univ{M}}$
since by Remark \ref{rem:Ggg_def}, $\mathscr{G}^M$ and
 $\canN\mapsto\sJs(\canN)$ are $\univ{M}$-definable.

So suppose $A\sub\mathscr{G}^M$ is $\OD^{\univ{M}}$ but has no nice code.
Let $\lambda\in\OR^M$ be a
limit cardinal of $M$
and  $\xi<\lambda$ and $\varphi$ be a formula (in the language of set theory)
such that
$\canN\in A$ iff $\her_\lambda^M\sats\varphi(\canN,\xi)$.
In fact, because we are arguing by contradiction, we may assume
$\xi=0$ (take the least $\xi$ such that $\varphi(\cdot,\xi)$
yields a set with no nice code, and then by substituting another formula
for $\varphi$, we can take $\xi=0$).

Let $\canN\in\mathscr{G}^M$. Then
$N=\cs(\canN)$ is well-defined, has universe $\univ{M}$,
and satisfies standard condensation,
by Lemma \ref{lem:tame_es_above_om_1_def_from_Ppp}. Also, as in the proof
of that lemma, $N$ can be translated into an iterable $x$-mouse $N_x$
for some $x\in\RR^M$. Let
\[ H^\canN=\Hull_1^{N|(\lambda+\om)}(\{\lambda\}\cup\om_1^M),
\]
$C^\canN$ be its transitive collapse and
and $\pi^\canN:C^\canN\to N|(\lambda+\om)$ the uncollapse.
By the iterability of  $N_x$ (as an $x$-mouse), and since $x\in\rg(\pi^\canN)$
and $\lambda$ is an $M$-cardinal, then $C^\canN$ is $1$-sound with
$\pi^\canN(p_1^{C^\canN})=\{\lambda\}$.
So by standard condensation,  $C^\canN\pins N$, so in fact
$C^\canN\pins\sJs(\canN)$. But the elements of $H^\canN$ are
independent of $\canN$,
because given $\canN'\in\mathscr{G}^M$, $(\canN,\canN')$ are interdefinable
from parameters, so
$(\cs(\canN)|\lambda,\cs(\canN')|\lambda)$
are also (as they have the same extender sequence
above $\om_1^M$).
So $\OR(H^\canN)$ and $\pi^\canN$ are also independent of $\canN$.
Let $\pi=\pi^\canN$ and $\pi(\lambdabar)=\lambda$.
Let $\psi_\varphi$ be the formula, in the language of premice,
asserting $\varphi(L[\es]|\om_1)$.
Then
\[ \canN\in
A\iff\her_\lambda^M\sats\varphi(\canN)\iff
\cs(\canN)|\lambda\sats\psi_\varphi\iff
\sJs(\canN)|\lambdabar\sats\psi_\varphi.\]
So $(0,\lambdabar,\psi_\varphi)$ is a nice code for $A$, a contradiction.
\end{proof}
So let $\PP'$ be the coding of $\PP$ via nice codes
(for non-empty subsets of $\mathscr{G}^M$).
Then $\PP'\sub\delta^3$
and because $\mathscr{G}^M$ is $\Pi_2^{\her_{\delta}^M}$,
the set of conditions
$(\alpha,\beta,\varphi)\in\PP'$ is $\Sigma_3^{\her_\delta^M}$
(to assert $A_{\alpha\beta\varphi}\neq\emptyset$),
and the ordering restricted to these conditions is $\Pi_3^{\her_\delta^M}$.

Parts \ref{item:H[G]=H[M|om_1]}, \ref{item:every_Vop_cond_generic}: As usual,
for every $\canN\in\mathscr{G}^M$ we have the generic filter
\[ G_\canN=\{(\alpha,\beta,\varphi)\in\PP'\bigm| \canN\in
A_{\alpha\beta\varphi}\}. \]

\begin{clmfour}
$H[\canN]\sub H[\canN^+]=H[G_\canN]=\univ{M}$.
\end{clmfour}
\begin{proof}
$G_\canN$ and $\canN^+=\sJs(\canN)$ are easily inter-computable,
so $H[\canN^+]=H[G_\canN]$.
By standard Vopenka facts, we have $H[G_\canN]=\HOD_{\canN}^{\univ{M}}$.
\footnote{\label{ftn:Vopenka_extension}That is, let $X\sub\eta\in\OR^M$
with $X\in\HOD_{\canN}^{\univ{M}}$, and fix a formula $\varphi$
and $\alpha\in\OR$ such that
$X=\{\beta<\eta\bigm|\univ{M}\sats\varphi(\canN,\alpha,\beta)\}$.
For $\beta<\eta$ let
$p^*_\beta=\{\canN'\in\mathscr{G}^M\bigm|\univ{M}\sats\varphi(\canN',\alpha,
\beta)\}$, and noting $p^*_\beta\in\PP$,
let $p_\beta\in\PP'$ be the corresponding element,
and letting $\tau:\eta\to V$ with $\tau(\beta)=p_\beta$,
note $\tau\in H$.
But $\tau$ is a $\PP'$-name and $\tau_{G_{\canN}}=X$.}
But by Lemma \ref{lem:tame_es_above_om_1_def_from_Ppp},
we have $\HOD_{\canN}^{\univ{M}}=\univ{M}$.
\end{proof}

\begin{clmfour}\label{clm:N_generates_N^+}
 $H[\canN]=H[\canN^+]$.
\end{clmfour}
\begin{proof}
It suffices to see that $\canN^+\sub H[\canN]$, because then $\canN^+$
is just the Jensen stack above $\canN$ in $H[\canN]$,
so $\canN^+\in H[\canN]$ also.
Fix $\xi\in(\om_1^M,\om_2^M)$ such that $\canN^+|\xi$ projects to $\om_1^M$.
It suffices to see that $\canN^+|\xi\in H[\canN]$,
and again via the Jensen stack, we may assume that
$\gamma=\om_2^{\canN^+|\xi}<\xi$
and $\canN^+|\gamma\in H[\canN]$ and
 there is some $\lambda\in(\gamma,\xi]$
such that $\canN^+|\lambda$ is active.

Let $\QQ=\PP'\inter \canN^+|\gamma$. Note that $\QQ$ is definable over
$\univ{\canN^+|\gamma}$
(just as $\PP'$ is defined over $(\her_{\om_2})^M=\univ{\canN^+}$).
We have $\QQ\in H$ as $\QQ=\PP'\inter\gamma^3$.
Let $\lambda$ be the supremum of all $\lambda'\leq\xi$
such that $\canN^+|\lambda'$ is active.
So $\gamma=\om_2^{\canN^+|\lambda}$.
So working over $\canN^+|\lambda$
(or equivalently, $M|\lambda$), let $R$ be the result of the P-construction
of $\canN^+|\lambda$ above $(\gamma^3,\QQ)$.
Then $R\in H$, because $\QQ\in H$, and given any $\canN'\in\mathscr{G}^M$, the
extender sequences of $(\canN')^+$
and $\canN^+$ agree above $\om_1^M$, so $\QQ$ is definable
over $(\canN')^+|\gamma$ just as over $\canN^+|\gamma$ (as they have the same
universe),
and their P-constructions yield the same output $R$.

As before, $R||\lambda\sats$ ``$\QQ$ is a $\gamma$-cc complete Boolean
algebra''
and
$G_{\canN,\gamma}=G_\canN\inter\gamma^3$
is $R||\lambda$-generic for $\QQ$.
Therefore the P-construction of $\canN^+|\lambda$ yields a
$(\gamma^3,\QQ)$-premouse (which is $R$),
and we have the usual fine structural correspondence
between segments of $\canN^+$ of height in $(\gamma,\lambda]$,
and the corresponding segments of $R$.

Now by induction, we have $\canN^+|\gamma\in H[\canN]$,
and $\canN^+|\gamma$ is inter-computable with $G_{\canN,\gamma}$.
But then the extender sequence of $\canN^+|\lambda$ above $\gamma$ is determined
by that of $R|\lambda$, as $\canN^+|\lambda$ is a small generic extension
thereof.
So $\canN^+|\lambda\in H[\canN]$, and therefore $\canN^+|\xi\in H[\canN]$, as
desired.
\end{proof}

There also is an alternate proof of this last claim,
which is actually quite different:

\begin{proof}[Sketch of alternate proof of Claim \ref{clm:N_generates_N^+}]
If our mice were Jensen-indexed, we could argue as follows:
Given $\alpha$ such that $\canN^+|\alpha$ is active, let
$\xi_\alpha=(\kappa^+)^{\canN^+|\alpha}$ where
$\kappa=\crit(F^{\canN^+|\alpha})$.
The sequence
\[
\mathscr{F}=\left<F^{\canN^+|\alpha}\rest\xi_\alpha\mid\alpha\in(\om_1^M,
\om_2^M)\text{ and }F^{\canN^+|\alpha}\neq\emptyset\right> \]
would be in $H$, because the sequence is independent of $\canN\in\mathscr{G}^M$.
But $(\canN,\mathscr{F})$ determines $\canN^+$,
by standard arguments. (Let $P$ be an active premouse with Jensen indexing.
Let $G=F^P\rest\kappa^{+P}$ where $\kappa=\crit(F^P)$.
Then $G,P||\OR^P$ determines $F^P$ as follows.
Let $X\sub\kappa$; we want to determine $i^P_F(X)$.
Let $\alpha<\kappa^{+P}$ be such that $X\in P|\alpha$
and $P|\alpha$ projects to $\kappa$.
Note that there is a unique elementary embedding
$\pi:P|\alpha\to P|G(\alpha)$ with
$\pi\rest\kappa=\id$, and $\pi$ is determined by
the first-order theory of $P|G(\alpha)$.
But then $\pi(X)=i^P_F(X)$, determining the latter, as desired.)

But we work with Mitchell-Steel indexing,
and it is not obvious to the author how to use the preceding kind of argument
directly with this indexing. So instead,
we convert indexing first.
Let $\widetilde{\canN^+}$ be the above-$\om_1^M$ Jensen-indexed
conversion of $\canN^+$. It isn't relevant here whether
the structure we get is actually a premouse, with sound segments etc.
It only needs to code the information in $\canN^+$ above $\om_1^M$
via a coherent sequence of Jensen-indexed extenders.
\footnote{Given a Mitchell-Steel indexed $P$ satisfying ``$\om_1$ exists'',
define $\widetilde{P}$ by induction on sequences of ultrapowers.
First set $\widetilde{P_0}=P_0$
where $P_0=P|(\om_1^P+\om)$.
If $P$ is active then
\[ \widetilde{P}=(\widetilde{U^P},F^P\rest(P|(\kappa^+)^P)) \]
where $U^P=\Ult(P|(\kappa^+)^P,F^P)$ and $\kappa=\crit(F^P)$.
If $P$ is passive then
$\widetilde{P}=\stack_{Q\pins P}\J(\widetilde{Q})$.}

Because the extender sequence of $\canN^+$ above $\om_1^M$ is independent of
$\canN\in\mathscr{G}^M$,
so is the extender sequence of $\widetilde{\canN^+}$ above $\om_1^M$.
Let $\widetilde{\mathscr{F}}$ be the restriction to ordinals
of $\es^{\widetilde{\canN^+}}$ above $\om_1^M$.
By a variant of the argument in parentheses above,
from $\canN$ and $\widetilde{\mathscr{F}}$ we can compute $\widetilde{\canN^+}$,
so $\widetilde{\canN^+}\in H[\canN]$.
(In the argument above we used that the proper segments of premice
are sound, but we don't need this property
of our Jensen-indexed structure. For if $\widetilde{\canN^+}|\alpha$ is active
with
extender $F$,
then we first convert $\widetilde{\canN^+}||\alpha$ to a Mitchell-Steel indexed
premouse
$Q$,
and then from $Q$ and $F\rest\OR$
we can compute $F$ much as before.)
So $\widetilde{\canN^+}\in H[\canN]$, but then we can (as above) invert
back to Mitchell-Steel indexing, so $\canN^+\in H[\canN]$.
\end{proof}

Applying the above with $\canN=\Momone^M$, we have established part
\ref{item:H[G]=H[M|om_1]}. To complete the proof of part
\ref{item:every_Vop_cond_generic},
observe that if $p\in\PP'$
then there is $\canN\in\mathscr{G}^M$ with $p\in G_\canN$
(because the forcing includes only nice codes for non-empty sets)
and we have just seen that $H[\canN]=H[G_\canN]=\univ{M}$, as desired.
\end{proof}

\begin{dfn}
$\Vop_{\mathscr{G}}^M$ denotes the forcing $\PP'$ of the previous lemma.
\end{dfn}

We finally use similar methods as part of the proof of  the following
theorem:

\begin{tm}\label{tm:HOD_tame_mouse}
 Let $M$ be a $(0,\om_1+1)$-iterable tame premouse satisfying $\ZFC$.
 Let $H=\HOD^{\univ{M}}$ and suppose that $H\neq\univ{M}$. Let $\delta=\om_2^M$
and
$t = \Th_{\Sigma_3}^{\her_{\delta}^M}(\delta)$. Then there are
$\mathbb{F},\her,W$ such that:
 \begin{enumerate}[label=--]
\item $\her=(H,\mathbb{F},t)$ is a $(0,\om_1+1)$-iterable $(\delta,t)$-premouse
with
universe $H$ and $\es^\her=\mathbb{F}$,
\item for every $\eta<\delta$ and every $X\in H$ with $X\sub\eta$, $X$
is encoded into $t$,  so
$X\in\J((\delta,t))$,
\item $\mathbb{F}$ is the restriction of $\es^M\rest[\delta,\infty)$ to $H$,
\item $\Hh$ is definable over $\univ{M}$ without parameters,
\item $\univ{M}$ is a generic extension of $H$ via a poset in $\J((\delta,t))$,
\item $\univ{M}=H[\canM^M]$,
\item $W$ is a premouse and lightface proper class of $\univ{M}$ and $W\sub H$,
\item $W\sats$ ``$\delta$ is the least Woodin cardinal'',
\item $t$ is generic for the meas-lim extender algebra of $W$ at $\delta$,
\item $\es^W\rest[\delta,\infty)$ is the restriction of $\mathbb{F}$ to $W$,
\item $H=\univ{W}[t]$, and
\item if $M=\Hull^M(\emptyset)$
then $W\ins X$ for some
correct iterate $X$ of $\Momone^M$.\footnote{Recall that when
we write $M=\Hull^M(\emptyset)$, the definability can refer to $\es^M$,
so this does not trivially imply that $\univ{M}\sats$ ``$V=\HOD$''.}
\end{enumerate}
\end{tm}
\begin{proof}
Let $D$ be the set of all $\gamma\in(\om_1^M,\om_2^M)$
such that $M|\gamma\sats\ZFC^-+$``$\om_1$ is the largest cardinal''.
Let $\vec{R}=\left<\PP_\gamma,R_\gamma\right>_{\gamma\in D}$
be $\PP_\gamma=\Vop_{\mathscr{G}}^M\inter\gamma^3$
and $R_\gamma$ is the output of the P-construction
of $M|\lambda$ above $\PP_\gamma$, where $\xi$ is least such that $\xi>\gamma$
and $\rho_\om^{M|\xi}=\om_1^M$ and $\lambda$ is the supremum of $\gamma$
and all $\lambda'\leq\xi$ such that $M|\lambda'$ is active.
By the proof of Lemma \ref{lem:tame_Vopenka}, $D$, $\vec{R}$ and
$\Vop_{\mathscr{G}}^M$
are $\Sigma_3^{(\her_{\om_2})^M}$,
and hence, encoded into $t$.
Let $R$ be the output of the
P-construction of $M$ above $(\delta,t)$.
Also like in \ref{lem:tame_Vopenka},
 $R$ is definable
without parameters over $\univ{M}$,
so $R\sub H$. We have $\Vop_{\mathscr{G}}^M\in R$,
and for each $\canN\in\mathscr{G}^M$,
$G_\canN$ is $R$-generic for $\Vop_{\mathscr{G}}^M$,
and $R[G_\canN]=R[\canN]$ has universe $\univ{M}$.
By \ref{lem:tame_Vopenka}, for each $p\in\Vop_{\mathscr{G}}^M$ we have some
such
$\canN\in\mathscr{G}^M$
with $p\in G_\canN$. It follows that $H\sub R$ ($R$ computes the theory of
ordinals
in $\univ{M}$
by considering what is forced by $\Vop_{\mathscr{G}}^M$).
So $\univ{R}=H$. Setting $\mathbb{F}=\es^R$, we have the desired
$\her=(H,\mathbb{F},t)$.

The fact that every bounded $X\sub\delta$ in $H$
is encoded into $t$ is like in the proof of Claim
\ref{clm:every_OD_set_has_nice_code} of Lemma \ref{lem:tame_Vopenka}
part \ref{item:Vop_sub_delta}.

We now construct $W$. We get $W|\delta$ from a certain simultaneous
comparison/genericity iteration
of all $\canN\in\mathscr{G}^M$, and then $\es^W\rest[\delta,\infty)$ is the
restriction
of $\es^M\rest[\delta,\infty)$. The details of the comparison
are similar to those in the proof of  Theorem \ref{thm:tame_HOD_x},
so we just give a sketch.
For $\canN\in\mathscr{G}^M$,
let $\Tt_\canN$ be the tree on $\canN$ produced by the comparison.
Given $\Tt_\canN\rest(\alpha+1)$ for all $\canN$, let $F^{\Tt_\canN}_\alpha$
be the least disagreement extenders,
indexed at $\ell_\alpha$ when non-empty, and
$K_\alpha=M^{\Tt_\canN}_\alpha||\ell_\alpha$.
For $\Tt_\canN\rest(\om_1^M+1)$, we compare, subject to folding in linear
iteration
at the least measurable of $K_\alpha$.
For $\Tt_\canN\rest(\om_1^M,\om_2^M)$, we compare, subject to folding in
meas-lim
genericity iteration for making
$t_{\canN}=\Th_{\Sigma_3}^{\canN^+}(\delta)$
generic (recall $\canN^+=\sJs(\canN)$,
a premouse with universe $(\her_{\om_2})^M$ in this case,
but the theory here can also refer to $\es^{\canN^+}$). For
each $\canN'\in\mathscr{G}^M$, since $\canN,\canN'$ are inter-definable
from
parameters
and the genericity iteration only begins above $\om_1^M$,
the theories $t_{\canN},t_{\canN'}$ are easily inter-computable,
and locally so
(ordinal-by-ordinal
modulo some fixed parameters $<\om_1^M$),
so genericity iteration with respect to $t_{\canN}$ is equivalent
to
that with
respect
to $t_{\canN'}$.

Let $\Lambda_{\tame,2}^\canN$ be the putative extension of
$\Lambda_{\tame}^\canN$ to trees
$\Tt$ of
length $<\om_2^M$,
which satisfy the other requirements of necessity,
but relative to $\canN^+$, and still using P-construction to compute
Q-structures.
Then $\Lambda_{\tame,2}^\canN$ is defined for all necessary trees,
and yields wellfounded models, by an easy
reflection
argument:
if not, then we can fix some $R\pins \canN^+$ witnessing this which projects to
$\om_1^M$,
and then use condensation to reflect to some hull $\bar{R}\pins \canN$,
and deduce that
$\Lambda_{\tame}^\canN$ is defective.\footnote{Here and below we use
the possibility that $\lgcd(N|\delta)=\om_1^{N|\delta}$
in Definition \ref{dfn:iterability-good}.}

We use $\Lambda_{\tame,2}^\canN$ to form $\Tt_\canN$. We stop the comparison
if it
reaches
length $\om_2^M$.
Let us verify that it in fact has length $\om_2^M$.
As usual, it cannot terminate early, in that we cannot reach a stage
$\alpha$
such that for some $\canN$, we have $M^{\Tt_\canN}_\alpha\ins
M^{\Tt_{\canN'}}_\alpha$ for
every $\canN'$.
So we just need to see that
$\Tt_\canN\rest\lambda\in\dom(\Lambda_{\tame,2}^\canN)$
for every limit $\lambda\leq\om_2^M$.
We also claim that $M(\Tt_\canN\rest\lambda)$ has no Woodin cardinals, and if
$M(\Tt_\canN\rest\lambda)$ is not a Q-structure for itself
then $\delta(\Tt_\canN\rest\lambda)=\lambda$ and
($*$)
$\canN^+||\lambda\elem_{\Sigma_3} \canN^+$.
Property ($*$) together with the usual fact that the earlier Q-structures are
retained,
ensures that $\Tt_\canN\rest\lambda$ (and in fact the entire comparison
through length $\lambda$) is definable over $\canN^+|\lambda$.
This is mostly as before, but ($*$) is new, so we focus on its verification.

Let $\left<\gamma_\alpha\right>_{\alpha<\om_2^M}$ enumerate the
set $C$ of
ordinals $\gamma<\om_2^M$
with
$M||\gamma\elem_{\Sigma_3}M|\om_2^M$, in increasing order.
Let $H_\beta=\Hull_{\Sigma_3}^{M|\om_2^M}(\beta)$.
Note that $C$ is club in $\om_2^M$ and $\om_1^M<\gamma_0$,
$H_{\gamma_\alpha}=M||\gamma_\alpha$,
and if $\gamma_\alpha<\gamma<\gamma_{\alpha+1}$ then
$H_\gamma=H_{\gamma_{\alpha+1}}$.
Moreover, if $\gamma_\alpha<\xi\leq\gamma_{\alpha+1}$ and
\[ t_\xi=\Th_{\Sigma_3}^{M|\om_2^M}(\xi), \]
then $t_\xi$ encodes a surjection of $(\gamma_\alpha+1)^{<\om}$
onto $\xi$. Write $t_{\Momone^M\xi}=t_\xi$
and $t_{\canN\xi}$ for the corresponding theory for other
$\canN\in\mathscr{G}^M$;
so when $\om_1^M\leq\xi$, there is a simple translation between
$t_{\Momone^M\xi}$
and $t_{\canN\xi}$.

Now suppose that $M(\Tt_\canN\rest\lambda)$ is not a Q-structure for itself.
We claim that
\[ \xi\eqdef\delta(\Tt_\canN\rest\lambda)=\gamma_\alpha \]
for some limit $\alpha$. For suppose that
$\gamma_\alpha<\xi\leq\gamma_{\alpha+1}$
for some $\alpha$ (or it is similar if $\xi\leq\gamma_0$).
Then $t_{\canN\xi}$ is meas-lim extender algebra generic over $M(\Tt)$,
and $\xi$ is regular in
$\J(M(\Tt))[t_{\canN\xi}]$.
But
$t_{\canN\xi}$ encodes a surjection of
$(\gamma_\alpha+1)^{<\om}$ onto
$\xi$,
collapsing $\xi$ in
$\J(M(\Tt))[t_{\canN\xi}]$, a contradiction.

By the previous paragraph, combined with the standard arguments,
we now get that $\lambda=\xi$ and $M|\xi\elem_{\Sigma_3}M|\om_2^M$
and $M|\xi\sats\ZFC^-$ and $\Tt_\canN\rest\xi$ is definable over
$M|\xi$.
So the arguments from earlier proofs now go through.

So we get a comparison of length $\delta=\om_2^M$.
Let $W|\delta$ be the resulting common part model. Note that $W|\delta\in H$,
and in fact, $W|\delta$ is definable without parameters over $\her_\delta^M$.
It follows that $W|\delta$
is in fact definable (in the codes) over $(\delta,t)$, via consulting what is
forced by
$\Vop_{\mathscr{G}}^M$. (Note here that
because $\Vop_{\mathscr{G}}^M$ has
the $\delta$-cc in $H$,
every bounded subset of $\delta$ in $M$
has a name
in $H$ given by some bounded $X\sub\delta$,
and since each such $X$ is encoded into $t$,
$\Th_{\Sigma_n}^{\her_\delta^M}$
is definable over $(\delta,t)$ for each $n<\om$.)
Also, each
$t_{\canN\delta}$ is meas-lim extender algebra generic over
$W|\delta$,
but $t$ is easily locally computed from any $t_{\canN\delta}$,
and hence is also generic over $W|\delta$.
So $W|\delta$ and $(\delta,t)$ are generically equivalent,
so we can build the P-construction of $\her$ above $W|\delta$,
or equivalently, the P-construction of $\cs(\canN)^{\univ{M}}$ above $W|\delta$,
for any $\canN\in\mathscr{G}^M$.
Let $W$ be the resulting model.
Because $W$ was produced by comparison,
the P-construction cannot reach a Q-structure,
so $W\sats$ ``$\delta$ is Woodin'',
and note $H=W[t]$ and $\mathbb{F}$ is induced by $\es^W\rest[\delta,\infty)$.

Finally suppose that $M=\Hull^M(\emptyset)$.
Then $\J(M)$ is an $\om$-mouse. In particular,
$M$ is countable.
The tree $\Tt=\Tt_{\Momone^M}$ is on $\Momone^M$,
via the correct strategy, and has countable length, since $M$ is countable.
Let $b=\Sigma_{\Momone^M}(\Tt)$
and $Q=Q(\Tt,b)$. By tameness,
$\delta$ is a strong cutpoint of $Q$,
and it follows that $W\ins \J(W)=Q\ins M^\Tt_b$,
as desired.
\end{proof}

\begin{rem}
 We actually now get another alternate proof of the fact
 that $H[\Momone^M]=\univ{M}$: We have $H=W[t]$,
and note that in  $W[t][\Momone^M]$, we can
recover the tree on $\Momone^M$ which leads to $W|\delta$,
by comparing $\Momone^M$ with $W|\delta$,
and noting
that since $\delta$ is the least Woodin of $W$, all the
Q-structures guiding this tree are available for this.
But then starting from $\Momone^M$,
we can then inductively recover $M|\delta$ by
translating the Q-structures over to segments of $M|\delta$
extending $\Momone^M$. We will also use a variant
of this later, in the non-tame context.\end{rem}

\section{$\star$-translation}\label{sec:star}

We now prepare to deal more carefully
with non-tame mice, by discussing the basics
of $\star$-translation and its inverse, the latter being  the generalization
of P-construction to non-tame mice. This section is essentially
a summary of results from \cite{closson}, slightly adapted.

\begin{dfn}
Let $N$ be an $n$-sound premouse. Fix some constant symbol $\dot{p}\in
V_\om\cut\OR$.
For $\alpha\leq\OR^N$ we write $t_{n+1}^N(\alpha)$ for the theory in the
language of premice
with constants in $\alpha\un\{\dot{p}\}$, which results
by modifying
$\Th_{n+1}^N(\alpha\un\{\pvec_{n+1}^N\})$
by replacing $\pvec_{n+1}^N$ with $\dot{p}$. We write $t_{n+1}^N$ for
$t_{n+1}^N(\rho_{n+1}^N)$.
\end{dfn}

\begin{dfn}
Let $P$ be a sound premouse. We say that
$\Tt$  is \emph{$P$-optimal} iff
\begin{enumerate}[label=--]
\item $\Tt$ is $\om$-maximal on some $\om$-premouse
$N\pins P|\om_1^P$,
\item $\Tt$ has limit length $\delta=\delta(\Tt)$,
\item $\delta$ is a successor cardinal of $P$,
 \item $\J(M(\Tt))\sats$ ``$\delta$ is Woodin'',
 \item $\Tt$ is definable from parameters over $P$, and
  \item $\rho_\om^P\leq\delta$ and $t_{k+1}^P(\delta)$ is
$\BB_{\measlim,\delta}^{M(\Tt)}$-generic
over $M(\Tt)$, where $k$ is least with $\rho_{k+1}^P\leq\delta$.
\end{enumerate}
Given $M\in\pm_1$, we say that $\Tt$ is \emph{$P$-optimal for $M$}
iff $\Tt\in M$ and $P\pins M$ and $\Tt$ is $P$-optimal
and $\delta(\Tt)$ is a  cutpoint (hence strong cutpoint) of $M$.
\end{dfn}

\begin{lem}\label{lem:P-optimal_uniqueness}
Let $M$ be a pm.
Let $\Tt$ be both $P$- and $P'$-optimal for $M$.
Then $P=P'$.
\end{lem}
\begin{proof}
Suppose $P\pins P'$. Then $\rho_1^{\J(P)}\leq\delta=\delta(\Tt)=\rho_\om^P$.
Let $k$ be least with $\rho_{k+1}^{P'}\leq\delta$.
Let $R=M(\Tt)$ and $t= t_1^{\J(\R)}(\delta)$
and $u= t_{k+1}^{P'}(\delta)$.
Then $t$
is computable from $t_1^{\J(P)}(\delta)$
(since $R$ is $P$-parameter-definable),
hence computable from $u$, since
 $\J(P)\ins P'$.
So $t\in\J(\R)[u]$ (recall $u$ is $\BB_{\measlim,\delta}^{R}$-generic
over $\J(\R)$).

Now $\delta$ is $\bfrSigma_\om^{\J(\R)}$-regular because $\delta$ is regular in
$\J(\R)[u]$ and
$t\in\J(\R)[u]$.
We claim $\rho_1^{\J(\R)}=\delta$.
So suppose $\rho_1^{\J(\R)}<\delta$. Let
 \[ H=\Hull_1^{\J(\R)}(\rho_1^{\J(\R)}\un\{p_1^{\J(\R)}\}) \]
 and $\gamma=\sup(H\inter\delta)$. Then $\gamma<\delta$ by the
$\bfrSigma_\om^{\J(\R)}$-regularity
of $\delta$.
 Let
 \[ H'=\Hull_1^{\J(\R)}(\gamma\un\{p_1^{\J(\R)}\}). \]
Then $H'\inter\delta=\gamma$ by a familiar argument, but then the transitive
collapse of $H'$ is in $R$,
a contradiction. (It follows that $\rho_\om^{\J(\R)}=\delta$; otherwise we get
an $\bfrSigma_{n+1}^{\J(\R)}$-singularization of $\delta=\rho_n^{\J(\R)}$ with
some $n\in[1,\om)$.)

Now for $n<\om$
let $t_n=\{\varphi\in t\bigm|\Ss_n(R)\sats\varphi\}$,
so $t_n\in\J(\R)$ and $t=\bigcup_{n<\om}t_n$.
Let $\tau\in\J(\R)$ be a name such that $\tau_G=t$, where $G$ is the generic
filter associated to
$u$.
Let $p\in\BB_{\measlim,\delta}^{R}$ be the Boolean value of ``$\tau$ is a
consistent theory in
parameters in $\delta\un\{\dot{p}\}$''. For each $n<\om$, let
$p_n\in\BB^R_{\measlim,\delta}$ be the conjunction of $p$
with the Boolean value of ``$\check{t_n}\sub\tau$''. So $p_n\in R$ and
$\left<p_n\right>_{n<\om}$ is $\bfrSigma_1^{\J(\R)}$.
In fact $\left<p_n\right>_{n<\om}\in R$,
since $\J(\R)$ does not definably singularize $\delta$ and
$\rho_1^{\J(\R)}=\delta$. So $q=\bigwedge_{n<\om}p_n\in\BB^R_{\measlim,\delta}$.
Now $q\neq 0$ and $q\in G$,
since $\tau_G=t=\bigcup_{n<\om}t_n$.
But then $t=\{\varphi\bigm|q\forces\varphi\in\tau\}$. So $t\in\J(\R)$, which is
impossible.
\end{proof}

\begin{dfn}\label{dfn:transcendent}
A premouse $M$ is \emph{transcendent} iff
$M\in\pm_1$, $M$ is an $\om$-mouse and for all $\Tt,P,\delta\in M$ and $k<\om$,  if
\begin{enumerate}[label=--]
 \item $P\pins M$ and $\delta=\rho_\om^P=\rho_{k+1}^P=\om_1^M$,
\item $\Tt$ is on $\Momone^M$, via $\Sigma_{\Momone^M}$,
and $\lh(\Tt)=\delta=\delta(\Tt)$,
\item $\Tt$ is $P$-optimal for $M$ and
 \item $\J(M(\Tt))\sats$ ``$\delta$ is a Woodin cardinal'',
\end{enumerate}
letting $Q=Q(\Tt,\Sigma_{\Momone^M}(\Tt))$ and $n<\om$, then
$\Th_{\Sigma_{n+1}}^M(\emptyset)$
is not definable from parameters over $Q[t_{k+1}^P]$.
Given an $\om$-mouse $R\pins M$,
\emph{transcendent above $R$} is the relativization to parameter $R$
and trees above $R$.
\end{dfn}

\begin{rem}\label{rem:transcendence_examples}
Note that $M_n^\#$ is transcendent for $n\leq\om$.
Many other such standard ``minimal'' mice are transcendent;
for example,
we will also observe in Remark \ref{rem:Wlim} that
$M_{\wlim}^\#$
(the sharp for a Woodin limit of Woodins) is transcendent,
as is the minimal mouse $M$ with an active superstrong extender.
But $(M_1^\#)^\#$ is not transcendent, which is easily seen via
genericity iteration. However, $(M_1^\#)^\#$ is trivially
transcendent above $M_1^\#$.
But the sharp of the model $S$ of
 Example \ref{exm:M_1^sharp-closed_reals} is not
transcendent above any $\om$-mouse $R\pins\Momone^S$.
For let $\Tt$ on
$M_1^\#(R)$ be
as there, and note that $\Tt$ is $S|\om_1^S$-optimal,
but we get $Q=M^\Tt_b$ is the output of the P-construction
of $S$ above $M(\Tt)$, and $\OR^Q=\OR^S$.
\end{rem}

\begin{rem}
 Let $\Tt$ be $P$-optimal
 and $\delta=\delta(\Tt)$.
We next define the
\emph{$\star$-translation} $Q^\star=Q^\star(\Tt,P)$ of certain premice $Q$
extending $M(\Tt)$
(in the right context). This is a simple variant
of the procedure in \cite{closson}.
The goal is to convert $Q$, which
may have extenders $E\in\es_+^Q$ with $\crit(E)\leq\delta$,
into a premouse $Q^{\star}$
extending $P$, having $\delta$ as a strong cutpoint,
but containing essentially the same information
(modulo the generic object $P$).
The overlapping extenders $E$ are converted
into ultrapower maps, which can be recovered
by $M$ by computing the corresponding core maps.
The differences with \cite{closson} are (i) we define
$R^\star$ for all \emph{valid}
segments of $R\ins Q$, which begins with $M(\Tt)$ itself
(instead of waiting for
the least admissible beyond $M(\Tt)$; \emph{valid} is defined presently
and pertains to condition
(iii) below), (ii) we set $M(\Tt)^\star=P$ (so $P$ is the starting
point, instead of basically $M|\delta$),
and
(iii) we allow
$\delta$ to
be the critical
point of extenders in $\es_+^Q$. Items (i) and (ii) only involve
slight fine structural changes, just at the bottom of the hierarchy, and
are straightforward. To translate the extenders as in (iii), one  takes
ultrapowers just as for other extenders, the difference being that the
ultrapower
is formed of some segment of $Q$ instead of a segment of a model of $\Tt$.
Otherwise things are very similar to \cite{closson}.
We give the definition now in detail, and will then state
some facts about it, but a proof of those facts is beyond the scope of the
paper, so we will just take them as a hypothesis throughout this section.
\end{rem}
\begin{dfn}\label{dfn:star}
 Let $\Tt$ be $P$-optimal and $\delta=\delta(\Tt)$.

Let $Q$ be a premouse. A \emph{$\delta$-measure} of  $Q$ is an $E\in\es_+^Q$
such that $\crit(E)=\delta$
 and $E$ is $Q$-total. Let $\mu^Q_\delta$ denote the least such, if it exists.
Say $Q$ is \emph{$\star$-valid}
iff
\begin{enumerate}[label=(\roman*)]
 \item  $M(\Tt)\ins Q$
 and if $M(\Tt)\pins Q$ then $Q\sats$ ``$\delta$ is Woodin'', and
 \item if $Q$
has a
$\delta$-measure
then $Q$ is $\delta$-sound and there is
$q<\om$ such that $\rho_{q+1}^Q\leq\delta$.
\end{enumerate}

 Given $\kappa<\delta$, let $\beta_\kappa$ be the least
 $\beta<\lh(\Tt)$ such that $\kappa<\nu(E^\Tt_\beta)$,
 let $M^*_{\kappa}$ be the largest $N\ins M^\Tt_\beta$  such that $N\inter\pow(\kappa)\sub
M(\Tt)$,
and $n_\kappa=$ the largest $n<\om$
such that $\kappa<\rho_n^{M^*_\kappa}$.

 Assuming $R$ is $\star$-valid,
we (attempt to) define the \emph{$\star$-translation}
$R^{\star}$ of $R$, by recursion as follows:
 \begin{enumerate}
  \item $M(\Tt)^\star=P$.
  \item\label{item:*-trans_Ult_delta-measure} If $R$ has a $\delta$-measure
  and $\rho_{r+1}^R\leq\delta(\Tt)<\rho_r^R$,
  then $R^\star=\Ult_r(R,\mu^R_\delta)^\star$ (note that if wellfounded,
  $\Ult_r(R,\mu^R_\delta)$
  is $\star$-valid and has no $\delta$-measure).
 \end{enumerate}
 Suppose from now on that $R$ has no $\delta$-measure. Then:
 \begin{enumerate}[resume*]
  \item\label{item:*-trans_R_active_with_crit(F^R)<delta} If $R$ is active with $\kappa=\crit(F^R)<\delta$
  then:
  \begin{enumerate}
   \item If $R$ is type 2 and $\delta=\lgcd(R)$
and $U=\Ult(R,F^R)$
  has a $\delta$-measure, then $R^\star=\Ult_0(R,\mu^U_\delta)^\star$.
  \item Otherwise $R^\star=\Ult_{n_\kappa}(M^{*}_\kappa,F^R)^\star$
  \end{enumerate}
 \item If $R$ is passive
 and $R=\J(S)$ (note then $S$ is $\star$-valid) then
 $\J(R)^\star=\J(S^\star)$.
 \item If  $R$ is passive of limit type
then $R^\star$ is the stack of all $S^\star$ for all such $S$
such that $S$ has no $\delta$-measure (note there
are cofinally many $\star$-valid $S\pins R$).
 \item If $R$ is active with $\crit(F^R)>\delta$ and
 \begin{enumerate}
 \item  the universe of $(R^\passive)^\star$
 is that of $R[P]$ (a meas-lim extender algebra extension), and
 \item
the canonical
 extension $F^*$ of $F^R$ to the generic extension
 induces a premouse $((R^\passive)^\star,F^*)$,
 \end{enumerate}
then we set $R^\star=$ this premouse.
\item Otherwise, $R^\star$ is left undefined.
\end{enumerate}
This definition proceeds by recursion along a natural linear order
(we leave the details of this to the reader,
but it is implicit in Remark
\ref{rem:closson_correction} below).
If this linear order is illfounded, then $R^\star$ is left undefined. Also, if any of the structures $S$
arising in the recursion leading to $R^\star$
fails to produce a premouse $S^{\star}$ extending $P$
(for example, if there is a $\star$-valid
$S$ such that $S\pins R$ but $S^\star$ is not a premouse extending $P$,
or if one of the ultrapowers arising in clauses \ref{item:*-trans_Ult_delta-measure} and \ref{item:*-trans_R_active_with_crit(F^R)<delta} are illfounded), then $R^\star$ is left undefined.
\end{dfn}

\begin{rem}\label{rem:closson_correction}
 If the phalanx $\Phi(\Tt)\conc(Q,q)$ is iterable
where either $q=0$
or $Q$ has a $\delta$-measure and $\rho_{q+1}^Q\leq\delta<\rho_q^Q$ (where $q$ indicates
the degree associated to $Q$ in the phalanx),
then it is straightforward to see that the definition of $Q^\star$ is
by recursion along a wellorder (consider  \label{pg:finite_versus_infinite_trees_*-trans} degree-maximal trees $\Uu$ on
$\Phi(\Tt)\conc(Q,q)$
such that $\crit(E^\Uu_\alpha)\leq\delta$ for all $\alpha+1<\lh(\Uu)$).
The fact that $Q^\star$ is a well-defined premouse, however,
takes  fine structural calculation, as
in \cite{closson}.
But there are some small issues in \cite{closson} which need correction;
most significantly (as far as the author is aware), the description of the
relationship between
the standard parameters of $Q$ and those of $Q^\star$
is incorrect in some cases, which come up for example in
 \cite[Theorem 1.2.9(d''), with $j=1$]{closson}.\footnote{In the notation used at that point of \cite{closson}, assuming
$\mathcal{P}$ is $1$-sound and Dodd-sound, it should be
$p_{n+1}(\mathcal{P}[g]^*)=j(p_{n+1}^R\cut\kappa)\conc q$
where $j:R\to\Ult_n(R,F^\mathcal{P})$ is the ultrapower map
for the relevant $R$
and $\kappa=\crit(j)$, and $q=t^\mathcal{P}\cut\delta$
where $t^{\mathcal{P}}$ is the Dodd parameter of $\mathcal{P}$.}

The author
intends to write an account of this, incorporating of
course the modifications
(i)--(iii). But this is beyond the scope of the present paper,
and here we will just summarize the features we need,
make the assumption that these do indeed work out
and complete the proofs of the current paper using this assumption.
\end{rem}

\begin{dfn}\label{dfn:Q^*_when_Q_correct}
The \emph{$\star$-translation hypothesis} (\emph{STH})
is the following assertion: Let $\Tt$ be $P$-optimal,
$\delta=\delta(\Tt)$ and $Q$ be $\star$-valid.
Let $k<\om$ be least such that $\rho_{k+1}^P\leq\delta$.
Let $Q^\star=Q^\star(\Tt,P)$.
Then:
\begin{enumerate}
\item\label{item:STH_part_1} If $Q^\star$ is a well-defined premouse,
then
$\pow(\delta)\inter
Q[t^P_{k+1}(\delta)]=\pow(\delta)\inter Q^\star$, and letting $n<\om$ and $x\in
Q^\star$ and
\begin{enumerate}
\item $\theta=\rho_\om^Q$, if $Q$ is sound with
$\delta\leq\rho_\om^Q$, or
\item $\theta=\delta$ otherwise,
\end{enumerate}
the theory $\Th_{\Sigma_{n+1}}^{Q^\star}(\theta\cup\{x\})$
is definable from parameters over $Q[t^P_{k+1}(\delta)]$.
 \item\label{item:STH_part_2} If $\Tt$ is on an $\om$-mouse $N$,
 of countable length,
 via $\Sigma_N$, and
$Q=Q(\Tt,\Sigma_N(\Tt))$,
then $Q^\star$ is a well-defined premouse,
is $\delta$-sound and above-$\delta$-$(q,\om_1+1)$-iterable
whenever $\delta<\rho_q(Q^\star)$.\qedhere
\end{enumerate}
\end{dfn}
The proof of STH is almost as in \cite{closson},
though see Remark \ref{rem:closson_correction}.

We now invert the $\star$-translation, also using
a small modification of \cite{closson}.
\begin{dfn}\label{dfn:black_hole}
Let $M\in\pm_1$ be a premouse and $\Tt$ be
$P$-optimal for $M$.
Let $\delta=\delta(\Tt)$.
Let  $q<\om$ and $Q$ be a $q$-sound, $(q+1)$-universal premouse
such that
$M(\Tt)\ins Q$
and $Q\sats$ ``$\delta$ is Woodin'', $\rho_{q+1}^Q\leq\delta\leq\rho_q^Q$,
and $\core_{q+1}(Q)$ is $(q+1)$-solid.

For $\kappa\in[\rho_{q+1}^Q,\rho_q^Q]$,
recall that $Q$ has the \emph{$(q+1)$-hull
property at $\kappa$}
iff
\[\pow(\kappa)\inter Q\sub C_\kappa=\cHull_{q+1}^Q(\kappa\cup\pvec_{q+1}^Q).\] (So $Q$ has
the $(q+1)$-hull property at $\rho_{q+1}^Q$,
by $(q+1)$-universality.)
 Let $\pi_\kappa:C_\kappa\to Q$ be the uncollapse map.

Say $Q$ is
\emph{$\star$-$\delta$-critical}
iff
\begin{enumerate}
\item $Q$  is $(\delta+1)$-sound but  non-$\delta$-sound (hence $\delta<\rho_q^Q$ and $\crit(\pi_\delta)=\delta$),
\item $Q$ has the $(q+1)$-hull property at
$\delta$, and
\item letting $\mu$ be the normal measure on $\delta$
derived from $\pi_\delta$, either
\begin{enumerate}[label=(\roman*)]
 \item \label{pg:mu_delta_first}
 $\mu\in\es_+^{C_\delta}$
(hence $Q=\Ult_q(C_\delta,\mu)$
and $C_\delta||\lh(\mu)=Q||\lh(\mu)$ and $\lh(\mu)=\delta^{++Q}$),
or
\item \label{pg:mu_delta_second}$C_\delta$ is active type 2 with $\lgcd(C_\delta)=\delta$ and
$\mu\in\es^U$ where $U=\Ult(C_\delta,F^{C_\delta})$
(hence $q=0$ and $Q=\Ult_0(C_\delta,\mu)$
and $U||\lh(\mu)=Q||\delta^{++Q}$
and $C_\delta^\passive=Q||\delta^{+Q}$).
\end{enumerate}
\end{enumerate}
Say $Q$  \emph{$\star$-successor-projects across $\delta$}
iff
\begin{enumerate}
\item $\rho_{q+1}^Q<\delta<\rho_q^Q$,
\item there is a largest $\kappa<\delta$
such that $Q$ has the $(q+1)$-hull property at $\kappa$; fix this $\kappa$,
\item  $C_\kappa=M^*_\kappa$ and $q=n_\kappa$,
\item
 letting  $E$ be the $(\kappa,\pi_\kappa(\kappa))$-extender
derived from $\pi_\kappa$, there is $\nu\in[\kappa^{+Q},\pi_\kappa(\kappa)]$
such that
$E\rest\nu$
is non-type Z and the trivial completion of $E\rest\nu$
is not in $\es_+^Q$, and taking $\nu$ least such,
we have $\delta\leq\nu$ and $Q=\Ult_{n_\kappa}(M^*_\kappa,E\rest\nu)$.
\end{enumerate}
Suppose $Q$ $\star$-successor-projects across $\delta$ and fix notation as above.
The \emph{extender-core} of $Q$ is
\[ N=(Q||\nu^{+Q},E') \]
where $E'$ is the trivial completion of $E\rest\nu$
(so $N^\passive=N||\nu^{+N}=Q||\nu^{+Q}\ins Q^\passive$). Note that $Q$ has the $(q+1)$-hull property at
$\delta$ iff $\nu=\delta$ iff $Q$ is
$\delta$-sound.

Say $Q$ is \emph{$\star$-terminal} iff either
\begin{enumerate}[label=(\roman*)]
 \item $Q$ is fully sound with $\rho_\om^Q=\delta$ and $Q$
is a Q-structure for $\delta$,
or
\item $Q$ is $\delta$-sound and there is $r\in[q,\om)$ such that $Q$ is $r$-sound but non-$(r+1)$-sound,
$\rho_{r+1}^Q<\delta\leq\rho_r^Q$,
$Q$ is $(r+1)$-universal and $\core_{r+1}(Q)$
is $(r+1)$-solid,
and there are cofinally many $\kappa<\delta$
such that $Q$ has the $(r+1)$-hull property at $\kappa$.\footnote{Note
that if $M(\Tt)\ins R\ins Q=Q(\Tt,b)$
where $M^\Tt_b$ is wellfounded,
then $R$ is $\star$-terminal iff $R=Q$.}
\end{enumerate}

Let $R\ins M$ with $P\ins R$.
We will (attempt to) define the \emph{black hole construction} of $R$ with respect to $\Tt,P$. It is a  kind of background construction using all
extenders
in $\es_+^R$ beyond $P$ (as far as the construction is defined),
but with a modified coring process which allows
the appearance of extenders $E$ with $\crit(E)\leq\delta$.
The intent is to invert the
$\star$-translation.

For $R$
such that $P\ins R\ins M$ we (attempt to) define
models $R^{\moon}_n$, for $n<\om$,
and then
$R^{\moon}$, by recursion on $(R,n)$, as follows.
Set $P^\moon_0=M(\Tt)$.
Suppose we have defined $R^\moon_0$.
We attempt to define models $R^\moon_{n+1}$ for $n<\om$,
and then set $R^\moon=\lim_{n<\om}R^\moon_n$.
Suppose we have $R'=R^\moon_n$.
If $R'$ is sound and $\delta\leq\rho_\om^{R'}$
then we define $R^\moon=R^\moon_m=R'$ for all $m\in[n,\om)$.
Otherwise let $q<\om$ be least such that $R'$ is $q$-sound and
either $\rho_{q+1}^{R'}<\delta$ or $R'$ is non-$(q+1)$-sound. Let $\rho=\rho_{q+1}^{R'}$.
We assume the following and proceed as follows; otherwise we give up and leave
$R^\moon_{n+1}$ undefined:
\begin{enumerate}[label=bh\arabic*.,ref=bh\arabic*]
\item $\rho\leq\delta$ and   $R'$ is non-$(q+1)$-sound,
but  $R'$ is $(q+1)$-universal and $\core_{q+1}(R')$ is
$(q+1)$-solid,
\item\label{item:if_R'_fails_hull_prop_at_delta} If $R'$ fails the $(q+1)$-hull property at $\delta$ (so by $q$-soundness and $(q+1)$-universality, we have $\rho=\rho_{q+1}^{R'}<\delta<\rho_q^{R'}$)
then $R'$ $\star$-successor projects across $\delta$, and   we set $R^\moon_{n+1}=$ the
extender-core of $R'$.
\item\label{item:if_R'_has_hull_prop_at_delta} If $R'$ has the $(q+1)$-hull property at $\delta$ then:
\begin{enumerate}
\item If $R'$ is non-$\delta$-sound (so $\delta<\rho_q^{R'}$, by $q$-soundness), then $R'$ is $\star$-$\delta$-critical and   we set $R^\moon_{n+1}=$ the $\delta$-core of $R'$.
\item If $R'$ is $\delta$-sound (so $\rho=\rho_{q+1}^{R'}<\delta$, by choice of $q$) then:
\begin{enumerate}
\item If there are only boundedly many $\kappa<\delta$ such that $R'$ has the hull property at $\kappa$, then $R'$ $\star$-successor projects across $\delta$, and  we set $R^\moon_{n+1}=$ the
extender-core of $R'$.
\item If there are unboundedly many $\kappa<\delta$
such that $R'$ has the hull property at $\kappa$,
then $R'$ is $\star$-terminal, and  we set $R^\moon=R'$ (and the construction
goes no further).
\end{enumerate}
\end{enumerate}
\end{enumerate}

This completes the description of $R^\moon_{n+1}$. We claim that  if $R^\moon_n$ exists for all $n<\om$
then $\lim_{n<\om}R^\moon_n$ also exists,
so we have defined $R^\moon$. \label{page:prove_R^moon_n_stabilizes_in_n}For suppose that $R^\moon_n$ and $R^\moon_{n+1}$ exist but $R^\moon_{n+1}\neq R^\moon_n$  for all $n<\om$.
Then for every $n<\om$, either $R^\moon_n$ is $\star$-$\delta$-critical or $R^\moon_n$ $\star$-successor projects across $\delta$.
Note that if $R^{\moon}_n$ $\star$-successor projects across $\delta$ then $\OR(R^{\moon}_{n+1})\leq\OR(R^{\moon}_n)$
and if $\OR(R^{\moon}_{n+1})=\OR(R^{\moon}_n)$
then $R^{\moon}_{n}$ is active type 2
and $R^{\moon}_{n+1}$ has superstrong type,
so $\nu(F(R^{\moon}_{n+1}))<\nu(F(R^{\moon}_{n}))$.
It easily follows that there is $n<\om$ such that $R^\moon_n$ is $\star$-$\delta$-critical. Fix such an $n$.
Then $R^\moon_{n+1}$, which is the $\delta$-core of $R^\moon_n$, is $\delta$-sound. So $R^\moon_{n+1}$ is not $\star$-$\delta$-critical,
so it $\star$-successor projects across $\delta$.
So letting $E=F(R^\moon_{n+2})$,
we have $R^\moon_{n+1}=\Ult_q(C,E)$
where $\kappa=\crit(E)$
and $C$ is the $\kappa$-core of $R^\moon_{n+1}$,
and $\rho_{q+1}^C=\rho_{q+1}^{R^\moon_{n+1}}\leq\kappa<\rho_q^C$. But since $R^\moon_{n+1}$
is $\delta$-sound,
$\nu(E)\leq\delta$.
So in fact $\nu(E)=\delta$ and $E$ is type 3,
so $R^\moon_{n+2}$ is active type 3
with largest cardinal $\delta$. But then it is easy to see that $R^\moon_{n+2}$ is not $\star$-$\delta$-critical and does not $\star$-successor project across $\delta$,
a contradiction.

Now let $R=M||\alpha$ or $R=M|\alpha$
for some $\alpha$,
and suppose we have
successfully defined
$S^\moon$ for all $S\pins R$,
these are  sound premice, and none are $\star$-terminal.
If $R=\J(S)$ then we set $R^\moon_0=\J(S^\moon)$.
If $R$ is passive of limit type
then $R^\moon_0=\liminf_{S\pins R}S^\moon$
(note that this exists, like with standard background constructions).
And if $R$ is active, hence with $\delta<\crit(F^R)$,
then we assume that $F^R$ restricts to an extender $E$
such that $S=((R^\passive)^\moon,E)$
is a premouse, and we set
$R^\moon_0=S$ (and otherwise $R^\moon_0$
is undefined).
\end{dfn}

The following lemma, saying in particular that the $\Moon$-construction
and $\star$-translation are inverses,
are straightforward to verify
by induction:
\begin{lem}\label{lem:bh_inverts_star}
Let $\Tt$ be $P$-optimal for $M$.

Adopting the notation of Definition \ref{dfn:star} \tu{(}$\star$-translation\tu{)},
suppose that
$Q$ is $\star$-valid,
$Q^\star=Q^\star(\Tt,P)$ is well-defined and $Q^\star\ins M$.
Then there is $n<\om$ such that
$(Q^\star)^\moon_n(\Tt,P)$ is well-defined and
$(Q^\star)^\moon_n=Q$.

Conversely, let $R$ and $r<\om$ be such that $P\ins R\ins M$ and
$R^\moon_r=R^\moon_r(\Tt,P)$ is well-defined.
Then $R^\moon_r$ is $\star$-valid,
$(R^\moon_r)^\star=(R^\moon_r)^\star(\Tt,P)$
is well-defined and $(R^\moon_r)^\star=R$.
Moreover, if $(P,0)\ins(S,s)\ins(R,r)$
then $S^\moon_s||\delta^{+S^\moon_s}=R^\moon_r||\delta^{+S^\moon_n}$.
\end{lem}

\begin{dfn}\label{dfn:just_beyond_delta-proj}
 Let $\Tt$ be $P$-optimal for $M$, $\delta=\delta(\Tt)$
 and $P\pins R\ins M$ with $\delta$ an $R$-cardinal.
 We say that $R$ is \emph{just beyond $\delta$-projection}
 iff there is $S$ such that $P\ins S\pins R$
 and $\rho_\om^S=\delta$ and there is no admissible
 $R'$ such that $S\pins R'\ins R$.
\end{dfn}

So if $R$ is just beyond $\delta$-projection
then $\rho_1^R\leq\delta$.
The $\Moon$-construction is almost completely
local, but it seems maybe not quite completely
at the level of measurable Woodins, because of the
requirement of computing cores which project to $\delta$
(if there is such a non-trivial core, then there are $\delta$-measures,
hence measurable Woodins).
To handle this we split into two cases in what follows, making use of the two formulas $\psi_{\moon}$
and $\psi'_{\moon}$.
\footnote{The construction \emph{is} completely local \emph{in the codes},
but it seems maybe not literally.
More precisely, if  $\rho_\om^{R^\moon}=\delta$
but $R^\moon$ is not sound, and $\alpha\in\OR$,
then while it is not clear that the model $\J_\alpha(\core_\om(R^\moon))$
is definable from parameters over $\J_\alpha(R^\moon)$, the theory
$\Th_{\rSigma_{n+1}}^{\J_\alpha(\core_\om(R^\moon))}(\delta\cup\{x\})$
is definable from parameters over $\J_\alpha(R^\moon)$,
for each $n<\om$ and $x\in\J_\alpha(\core_\om(R^\moon))$. However, if
$\alpha\geq(\om\cdot\OR(\core_\om(R^\moon)))$,
then we do have $\J_\alpha(\core_\om(R^\moon))$ literally
definable from parameters over $\J_\alpha(R^\moon)$.}
\begin{lem}\label{lem:compute_*_q_translation}
Assume STH.\footnote{Actually the lemma only uses part \ref{item:STH_part_1} of STH.}
Then there are formulas $\psi_{\moon}$ and $\psi'_{\moon}$ of the language of premice
such that
 for all  $M\in\pm_1$,
 all $\Tt,P,R\in M$ such that $\Tt$ is $P$-optimal for $M$,
 $\delta=\delta(\Tt)$ is an $R$-cardinal,
$P\pins R \ins\ M$
and $R^\moon_0=R^\moon_0(\Tt,P)$ is well-defined,
we have:
\begin{enumerate}
\item\label{item:card_agmt} $R$ and $R^\moon_0$ have the same cardinals
$\kappa\geq\delta$,
and for each such $\kappa>\delta$ \tu{(}so $\kappa<\OR^R$\tu{)}, we have
$R^\moon_0|\kappa=(R|\kappa)^\moon_0=(R|\kappa)^\moon$
\tu{(}whereas $R^\moon_0|\delta=P^\moon_0=P^\moon$\tu{)}.
\item\label{item:proj_agmt} If $\rho_\om^R=\OR^R$ then $R^\moon=R^\moon_0\sub
R$ and
$\rho_\om^{R^\moon}=\OR^{R^\moon}=\OR^R$.
\item If $R$ is not just beyond $\delta$-projection
then $R^\moon_0\sub R$ and $\Th_{\rSigma_0}^{R^\moon_0}(R^\moon_0)$
is defined over $R$ by $\psi_{\moon}$
from the parameter $\Tt$, and
\item\label{item:R_just_beyond_delta}If $R$ is just beyond $\delta$-projection
then
$\rho_1(R^\moon_0)\leq\delta$, $R^\moon_0$ is $\delta$-sound,
and $t^{R^\moon_0}_1(\delta)$ is defined over $R$ by
$\psi_{\moon}'$ from the parameter $\Tt$.
\end{enumerate}
\end{lem}
\begin{proof}[Proof sketch]
The formula $\psi_\moon$ basically says to perform the
$\Moon$-construction, whereas $\psi'_\moon$ says to do that
up to a point, and then to perform a coded version of the construction,
working with theories $\sub\delta$ instead of the actual models.
 The construction is defined from the parameter $(\Tt,P)$, and we are given the parameter $\Tt$. But using $\Tt$, we can identify $P$
 in a $\Sigma_1$ fashion over $M$,
 by Lemma \ref{lem:P-optimal_uniqueness}.
Now we won't write down the formulas $\psi_{\moon}$ and $\psi'_\moon$ explicitly, but just sketch out some main
considerations and an explanation of parts \ref{item:card_agmt},
\ref{item:proj_agmt} and \ref{item:R_just_beyond_delta}.
 The proof that everything works is  by induction on $R$.

 If $R$ has no largest cardinal,  it is easy by induction,
 leading to (the relevant clause of) $\psi_\moon$ for this case.

 Suppose $R$ has a largest cardinal $\kappa>\delta$.
 We can compute $(R|\kappa)^\moon$
 definably over $R|\kappa$, and it has height $\kappa$.
 We claim that for each $S\pins R$ such that $\rho_\om^S=\kappa$,
 we have $\rho_\om(S^\moon)=\kappa$ and $S^\moon\pins R^\moon_0$,
 and hence, $(R^\passive)^\moon_0$ is the stack of $\J(S^\moon)$
 over all such $S$. Given this, we get an appropriate
 definition of $(R^\passive)^\moon_0$ over $R$,
 and then if $R$ is active, we just add the restriction of $F^R$,
 leading to $\psi_\moon$ for this case.
 So supposing
$\rho_\om(S^\moon)<\kappa$,
 then by induction, we can take a hull of $S$ to which condensation
 applies, producing some $\bar{S}\pins S$,
 such that $\rho_\om(\bar{S}^\moon)=\rho_\om(S^\moon)$
 and $\bar{S}^\moon$ defines the set missing from $S^\moon$,
 which gives a contradiction. Conversely,
 since $(S^\moon)^\star=S$,
 STH part \ref{item:STH_part_1} implies the set $t\sub\kappa$ missing from
$S$ is definable from parameters over $S^\moon[P]$.
But then if $\kappa<\rho_\om(S^\moon)$,
then there is a set in $\pow(\kappa)\inter S^\moon$
coding the relevant forcing relation, which implies $t\in S^\moon[P]\sub
S$, contradiction.

Now suppose $\lgcd(R)=\delta$. If $R$ is not just beyond $\delta$-projection,
then $R$ is admissible or a limit of admissible proper segments, and it is
then easy to define $R^\moon_0$ over $R$.
So suppose there is $S$ such that $P\ins S\pins R$ and $\rho_\om^S=\delta$
but there is no admissible $S'$ with $S\pins S'\ins R$,
and let $S$ be least such. Then
we can take $n<\om$
such that $S^\moon_n$ is sound
and $\rho_\om(S^\moon_n)=\delta$,
and then $S^\moon=S^\moon_n$,
and  note $R^\moon_0=\J_\alpha(S^\moon_n)$,
where $R=\J_\alpha(S)$. Let $k<\om$ be such that
$\rho_{k+1}(S^\moon_n)=\delta$, and note
that $t=t_{k+1}^{S^\moon_n}$ is definable from $\Tt$
over $S$. Starting from the parameter $t$, it is straightforward to uniformly
define
$t_1^{\J_\beta(S^\moon_n)}(\delta)$ over $\J_\beta(S)$,
for $\beta\in(0,\alpha]$. This leads to $\psi_\moon'$.

Finally let us observe that $R_0^\moon$ is $\delta$-sound in part \ref{item:R_just_beyond_delta}.
Note that $R_0^\moon=\Hull_1^{R_0^\moon}(\delta\cup\{\gamma\})$
where $\gamma=\OR(S^\moon)$, and let $\xi$ be least
such that $\gamma\in\Hull_1^{R_0^\moon}(\delta\cup\{\xi\})$.
If $\xi=0$ then we are done, so suppose $\xi\geq\delta$.
Then $R^0_\moon=\Hull_1^{R_0^\moon}(\delta\cup\{\xi\})$,
and note that $\Hull_1^{R_0^\moon}(\xi)\inter\OR=\xi$,
and it follows that $p_1^{R_0^\moon}\cut\delta=\{\xi\}$
and $R_0^\moon$ is  $1$-solid above $\delta$
and is $\delta$-sound.
\end{proof}

A full analysis of
$\star$-translation and proof of STH needs a sharper, more extensive
version of the preceding lemma.

\begin{lem}\label{lem:*-trans_gives_seg_of_M}Assume STH. Let $\Tt$ be  $P$-optimal for $M$ where
$M$ is $(0,\om_1+1)$-iterable. Let $N\pins \Momone^M$ be such that $\Tt$ is on $N$ and let $\Gamma$ be an \tu{(}so in fact the unique\tu{)} $(\om,\theta+1)$-strategy for $N$, where $\theta$ is some regular uncountable cardinal.  Let $\delta=\delta(\Tt)$ and $Q=Q(\Tt,\Gamma(\Tt))$. Then:
\begin{enumerate}
\item\label{item:Q^star_projects_<=delta} $Q^{\star}$ is a well-defined, $\delta$-sound premouse, projects $\leq\delta$, with $\delta$ a strong cutpoint,
 \item
either $Q^\star\pins M$ or \tu{[}$M||\delta^{+M}=Q^\star||\delta^{+M}$
and $\delta$ is a successor cardinal in $M$\tu{]}, and
\item\label{item:if_M_om-mouse_then_Q^star_ins_M} if $M$ is an $\om$-mouse then
 $Q^\star\ins M$.
 \end{enumerate}
\end{lem}
\begin{proof}
We have $\delta\leq\theta$, since $\Tt$ is via $\Gamma$,  an $(\om,\theta+1)$-strategy,
and $\delta$ is a limit ordinal.
Note then that by taking a countable hull, we may assume that $\delta<\om_1$ and that $M$ is countable.
(In so doing, the transitive collapse $\bar{\Tt}$ of $\Tt$ is also via $\Gamma$,
by the uniqueness of $\Gamma$
and since the uncollapse map allows us to lift the phalanx of $\bar{\Tt}$ to the phalanx of $\Tt\rest\xi$
for some $\xi<\theta$.)

Now
by STH,
$Q^\star$ is a $\delta$-sound premouse,  $\delta$ is a strong cutpoint and a successor cardinal of $Q^{\star}$, and for each $q<\om$, if $\delta<\rho_q^{Q^{\star}}$
then $Q^{\star}$ is above-$\delta$, $(q,\om_1+1)$-iterable. So it suffices to see that $Q^{\star}$ projects $\leq\delta$. (If $M$ is an $\om$-mouse
then we can't have $M\pins Q^{\star}$,
since $\delta$ is a cardinal in $Q^{\star}$.)

But $\Tt$ is $P$-optimal for $Q^{\star}$, so by Lemma \ref{lem:bh_inverts_star} (applied with $Q^{\star}$ replacing $M$ there), there is $n<\om$
such that $(Q^{\star})^{\moon}_n=Q$,
so by Lemma \ref{lem:compute_*_q_translation} (with $Q^{\star}$ replacing $M$), $t_{d+1}^Q(\delta)$ is definable
from parameters over $Q^{\star}$,
where $d$ is such that $\rho_{d+1}^Q\leq\delta<\rho_d^Q$. So it suffices to see that $t^Q_{d+1}(\delta)\notin Q^{\star}$.
But otherwise, by STH, we would have $t^Q_{d+1}(\delta)\in Q[t^P_{k+1}]$ where $k$ is as there
(and recall $t^P_{k+1}$ is meas-lim extender algebra generic over $Q$ at $\delta$). But this is impossible, like in the proof of Lemma \ref{lem:P-optimal_uniqueness}.
 \end{proof}

\begin{rem}\label{rem:Wlim} Assume STH
and $M_{\wlim}^\#$ exists and is $(\om,\om_1+1)$-iterable.
Then $M_{\wlim}^\#$ is
transcendent. For suppose not, and let
$\Tt,P\in M$ be a counterexample; so
$t=\Th_{\rSigma_1}^{M^\#_{\wlim}}(\emptyset)$
is in $\J(Q[P])$ where $Q=Q(\Tt,\Sigma_{\Momone}(\Tt))$.
But then if $Q^\star\pins M$ then $Q\in M$, so $Q[P]\in M$,
so $t\in M$, contradiction. So $M\ins Q^\star$,
which implies $M=Q^\star$. But note then that
$M^\moon_0$ is  produced by iterating
the phalanx $\Phi(\Tt)\conc\left<Q\right>$
finitely many steps (via extenders with critical points $\leq\delta$), so
$M^\moon_0$ is also an iterate
of $\Momone$ or a segment thereof. But $M^\moon_0[P]$,
a generic extension via the meas-lim extender algebra,
 has universe that of $M$, and the extenders
 in $\es_+(M^\moon_0)$ with critical point $>\delta$ are
 exactly the level-by-level restrictions of those of $\es_+^M$.
 So $M^\moon_0$ inherits all the Woodin cardinals of $M$,
 and the active sharp, and this contradicts the minimality of $M$.

 The argument for the least mouse with an active superstrong extender
 is very similar. And obviously there are many such variants.
\end{rem}

\section{$\HOD$ in non-tame mice}\label{sec:HOD_in_non-tame}

We can now begin our analysis of ordinal definability in non-tame mice.
All the results will assume STH.
Recall that \S\ref{sec:candidates} applies.

\begin{dfn}
 Let $\canN$ be a premouse satisfying ``$\ZFC^-+V=\HC$''.
 Then
 $\Lambda^\canN$ denotes the partial $(\om,\OR^\canN)$-iteration strategy
$\Lambda$
for $\canN$, defined over $\canN$ as follows. We define $\Lambda$ by induction
on the
length of trees. Let $\Tt\in\canN$. We say that $\Tt$ is \emph{necessary} iff
$\Tt$ is an iteration
tree via
$\Lambda$, of limit length, and letting $\delta=\delta(\Tt)$, either $M(\Tt)$
is
a Q-structure
for itself, or $\Tt$ is $P$-optimal for $\canN$,
with some $P\pins\canN$.
Every $\Tt\in\dom(\Lambda)$ is necessary. Let $\Tt$ be necessary,
and $P$-optimal for $\canN$ if such $P$ exists.
Then $\Lambda(\Tt)=b$ iff $b\in\canN$ and letting $Q=Q(\Tt,b)$,
if $M(\Tt)\pins Q$ then
  $Q^\star=Q^\star(\Tt,P)$ is well-defined and $Q^\star\pins\canN$.
  (Note that if $\Lambda(\Tt)=b$  then
$b,Q\in\J_\lambda(Q^\star)$, where $\J_\lambda(Q^\star)$ is admissible,
and the
assertion that ``$\Lambda(\Tt)=b$'' is uniformly
$\Sigma_1^{\J_\lambda(Q^\star)}(\{\Tt\})$,
by Lemmas \ref{lem:bh_inverts_star}, \ref{lem:compute_*_q_translation} and
 \ref{lem:Q_computes_b}. So $\Lambda$ is $\Sigma_1$-definable over
$\canN$.\footnote{Here of course we can refer to
$\es^{\canN}$. Since $\canN\sats$ ``$V=\HC$'',
we can say that ``$\delta$ is a cutpoint of $\canN$''
by just saying it is a cutpoint of some segment of $\canN$
which projects to $\om$.})

We say that $\canN$ is \emph{iterability-good} iff all trees via
$\Lambda^\canN$ have wellfounded models,
and $\Lambda^\canN(\Tt)$ is defined for all necessary $\Tt$. (Note that
\emph{iterability-good} is
expressed by a first-order formula $\varphi$ (modulo $\ZFC^-$).)
\end{dfn}

By Lemma \ref{lem:*-trans_gives_seg_of_M}, we have:

\begin{lem}\label{lem:M_is_iterability-good}
Assume STH.  Let $M\in\pm_1$ be $(0,\om_1+1)$-iterable and $\canM=\canM^M$.
Then $\Lambda^{\canM}\sub\Sigma_{\canM}$
and $\canM$ is iterability-good.
\end{lem}

\begin{dfn}\label{dfn:G^M}
 Let $M\in\pm_1$. Then $\mathscr{G}^M$ denotes
 the set of all strong iterability-good $M$-candidates $\canN$
 such that   for every $P\pins\canN$,
 if $P$ has no largest cardinal then $P\sats$ ``I am
$\cs(P|\om_1^P)$'' (see \ref{dfn:cs}).\footnote{The clause regarding the $P\pins\canN$ and $\cs(P|\om_1^P)$
is not needed in the proof of Theorem \ref{tm:V=HOD_in_transcendent_mice}.}
\end{dfn}

\begin{proof}[Proof of Theorem \ref{tm:V=HOD_in_transcendent_mice}]
We are assuming STH  and $M\in\pm_1$ is a transcendent  strongly tractable  $\om$-mouse,
and want to see that $\Momone=\Momone^M$
is definable without parameters over $\her_\lambda^M$,
where $\lambda=\om_2^M$ (see \S\ref{subsec:terminology} and Definitions \ref{dfn:Q^*_when_Q_correct}, \ref{dfn:transcendent}, \ref{dfn:tractable}). We will show that
$\mathscr{G}^M=\{\Momone\}$, which suffices.
We will not use the assumption that $M$ is an $\om$-premouse,
nor that it is transcendent, until the very last paragraph of the proof.
So what we
establish prior to that point
(up to Claim \ref{clm:Q^star_reaches_M}, inclusive)
can and will also be used
in the proof of Theorem \ref{tm:HOD_non-tame_mouse}.

We know $\Momone\in\mathscr{G}^M$,
by Lemmas \ref{lem:iterable_tract_implies_strong} and
\ref{lem:M_is_iterability-good},
so suppose $\canN\in\mathscr{G}^M$
with $\canM\neq\canN$.
We will  form and analyse a genericity comparison of
$\canM$ with
$\canN$ to reach a contradiction.
(For the proof of Theorem \ref{tm:HOD_non-tame_mouse},
we need to adapt this
to a simultaneous
comparison
of all elements of $\mathscr{G}^M$.)

Let
$\widetilde{\canM}=\widetilde{\canM}(\canN)$
and $\widetilde{\canN}=\widetilde{\canN} (\canM)$
(see  \ref{dfn:ptilde(m)}). Recall  that  $\widetilde{\canM}\ins
\canM^+=M|\om_2^M$ and
$\widetilde{\canN}\ins\canN^+$
(and $\univ{\canN^+}=\her_\lambda^M$),
$\rho_1^{\widetilde{\canM}}=\om_1^M=\rho_1^{\widetilde{\canN}}$,
 $\univ{\widetilde{\canM}}=U=\univ{\widetilde{\canN}}$,
and there is
$\xi<\om_1^M$ such that
$\Sigma_1^{\widetilde{\canN}}(\{\xi,p_1^{\widetilde{\canN}}\})$
is recursively equivalent to
$\Sigma_1^{\widetilde{\canM}}(\{\xi,p_1^{\widetilde{\canM}}\})$,
meaning that there are recursive functions
$\varphi\mapsto\varphi'$ and $\varphi\mapsto\widehat{\varphi}$
such that for all $x\in U$ and $\Sigma_1$ formulas $\varphi$ in the passive
premouse language,
\[ \widetilde{\canM}\sats\varphi(\xi,p_1^{\widetilde{\canM}},x)\iff \widetilde{\canN}\sats\varphi'(\xi,p_1^{\widetilde{\canN}},x)\]
and
\[ \widetilde{\canN}\sats\varphi(\xi,p_1^{\widetilde{\canN}},x)\iff\widetilde{\canM}\sats\widehat{\varphi}(\xi,p_1^{\widetilde{\canM}},x).\]
We may assume that the $1$-solidity
witnesses for $\widetilde{\canM}$
are in
$\Hull_1^{\widetilde{\canM}}(p_1^{\widetilde{\canM}}\cup\xi)$,
and likewise for $\widetilde{\canN}$.
Moreover, since $M$ is strongly tractable, we in fact have $\OR^U=\OR^{\widetilde{\canM}}=\OR^{\widetilde{\canN}}<\om_2^M$,
since the definition of $\widetilde{\canM}$ and $\widetilde{\canN}$
gives  a $\bfSigma_1^{\her_\lambda^M}$ cofinal function $f:\om\to\OR^U$.
So
$\widetilde{\canM}\pins\canM^+=M|\om_2^M$ and $\widetilde{\canN}\pins\canN^+$.

Let
$t^{\widetilde{\canM}}=t_1^{\widetilde{\canM}}$ and
$t^{\widetilde{\canN}}=t_1^{\widetilde{\canN}}$.
Let $(A,B)$
be the least conflicting pair
with $A\pins\canM$ and $B\pins\canN$.
We construct a
$t^{\widetilde{\canM}}$-genericity comparison $(\Tt,\Uu)$ of
$(A,B)$, via  $(\Lambda^{\canM},\Lambda^{\canN})$,
folding in initial linear iteration past $(\xi,A,B)$,
and linear iterations past $\star$-translations of non-trivial Q-structures.
We now turn to the details.

We first set up some notation. For
$\eta\in(\xi,\om_1^M)$, let \[
H_\eta=\cHull_1^{\widetilde{\canM}}(\eta\cup\{p_1^{\widetilde{\canM}}\})\text{
and } \pi_\eta:H_\eta\to \widetilde{\canM}\text{ be the uncollapse},\]
\[
J_\eta=\cHull_1^{\widetilde{\canN}}(\eta\cup\{p_1^{\widetilde{\canN}}\})\text{
and }\sigma_\eta:J_\eta\to\widetilde{\canN}\text{ be the uncollapse}.\]
Note that $\rg(\pi_\eta)=\rg(\sigma_\eta)$ and $\OR^{H_\eta}=\OR^{J_\eta}$
and $H_\eta\pins\canM$ and $J_\eta\pins\canN$.
Let $C\sub\om_1^M$ be the club of all $\eta$
such that $\eta=\crit(\pi_\eta)=\crit(\sigma_\eta)$.
So for $\eta\in C$,  we have
$\rho_1^{H_\eta}\leq\om_1^{H_\eta}=\eta$
and $\rho_1^{J_\eta}\leq\om_1^{J_\eta}=\eta$ and
 $\pi_\eta(\eta)=\om_1^M=\sigma_\eta(\eta)$
and  $\pi_\eta(p_1^{H_\eta}\cut\eta)=p_1^{\widetilde{\canM}}$
and $\sigma_\eta(p_1^{J_\eta}\cut\eta)=p_1^{\widetilde{\canN}}$ and
\begin{equation}\label{eqn:t_rest_eta_is_theory_of_H_eta}
t^{\widetilde{\canM}}\rest\eta=\Th_1^{H_\eta}(\eta\cup\{p_1^{H_\eta}\})(p_1^{
H_\eta}/\dot{p}) \end{equation}
and likewise for $t^{\widetilde{\canN}}\rest\eta$ and $J_\eta$.
And given $\eta<\delta\leq\eta'$ with $\eta,\eta'\in C$  consecutive,
\begin{equation}\label{eqn:encoded_surj}t^{\widetilde{\canM}}\rest\delta\text{
encodes a surjection } (\eta+1)^{<\om}\to\delta.\end{equation}
If $\eta\in C$ and $\rho_1^{H_\eta}<\eta$ (equivalently,
$\rho_1^{J_\eta}<\eta$) then
(since $H_\eta$ is $1$-sound and
$\pi_\eta(p_1^{H_\eta}\cut\eta)=p_1^{\widetilde{\canM}}$),
\begin{equation}\label{eqn:encoded_surj_when_rho_1=om}t^{\widetilde{\canM}}
\rest\eta\text {
encodes a surjection } \om\to\eta.\end{equation}

We will construct a strictly increasing sequence
$\left<\eta_\beta\right>_{\beta<\om_1^M}$
and $(\Tt,\Uu)\rest(\eta_\beta+1)$, recursively in $\beta$.
The ordinals $\eta_\beta$ will be exactly those $\eta$
such that $M((\Tt,\Uu)\rest\eta)$ is not a Q-structure
for itself (and then $\eta=\delta((\Tt,\Uu)\rest\eta)$, but
$\eta$ need not be Woodin in the eventual $M(\Tt,\Uu)$).
We will see that each $\eta_\beta$
is a limit point of $C$ with $\rho_1^{H_{\eta_\beta}}=\eta_\beta$.

If we have constructed
$(\Tt,\Uu)\rest(\alpha+1)$ where $\alpha<\om_1^M$,
we let $F^\Tt_\alpha,F^\Uu_\alpha,K_\alpha$
be as usual, and will have $F^\Tt_\alpha\neq\emptyset$
or $F^\Uu_\alpha\neq\emptyset$.

We now begin the construction, considering first $\beta=0$. We construct
$(\Tt,\Uu)\rest\eta_0$
in 2 phases.  In the first phase (given $(\Tt,\Uu)\rest\alpha+1$ where
$\alpha<\eta_0$), we compare, subject to linear iteration of the least
measurable $\mu$ of $K_\alpha$,
until $\mu\geq\max(\xi,\OR^A,\OR^B)$. In the second phase, we compare,
subject to $t^{\widetilde{\Momone}}$-genericity iteration
for meas-lim extender algebra axioms of $K_\alpha$
(equivalently, $t^{\widetilde{\canN}}$-genericity).
Let $\eta_0$ be the least
 $\eta$ such that $M((\Tt,\Uu)\rest\eta)$
is not a Q-structure for itself.
The iteration strategies
$\Lambda^{\Momone},\Lambda^{\canN}$ apply trivially prior
to stage $\eta_0$,
and because $\widetilde{\canM},\widetilde{\canN}\in M$, an easy reflection argument shows that $\eta_0<\om_1^M$ exists.

Since $R=M((\Tt,\Uu)\rest\eta_0)$ is not a Q-structure for itself,
we need to see that $\Tt\in\dom(\Lambda^{\canM})$
and $\Uu\in\dom(\Lambda^{\canN})$.
Let $\delta=\delta((\Tt,\Uu)\rest\eta_0)$. So
$t^{\widetilde{\Momone}}\rest\delta$ and $t^{\widetilde{\canN}}\rest\delta$
are $\BB_{\measlim,\delta}^{\J(R)}$-generic over $\J(R)$,
and $\delta$ is regular in $\J(R)[t^{\widetilde{\Momone}}\rest\delta]$.
So by line (\ref{eqn:encoded_surj}), it follows
that $\delta$ is a limit point of $C$, so
$\delta=\om_1^{H_\delta}=\om_1^{J_\delta}$,
and by line
(\ref{eqn:encoded_surj_when_rho_1=om}), it follows
that $\rho_1^{H_\delta}=\delta$, and in fact note
$\rho_\om^{H_\delta}=\delta$ (since each $\rSigma_{n+1}$ theory
in parameters can be defined from $t^{\widetilde{\canM}}\rest\delta$).
Likewise, $\rho_1^{J_\delta}=\delta=\rho_\om^{J_\delta}$.
Note also that
$(\Tt,\Uu)\rest\eta_0\sub(H_\delta|\delta)\inter(J_\delta|\delta)$
and $(\Tt,\Uu)\rest\eta_0$ is definable from the parameter $(A,B,\xi)$
over $H_\delta$, and likewise over $J_\delta$, and so $\eta_0=\delta$ (the
most complex aspect of the definition
being the $t^{\widetilde{\canM}}$-genericity iteration,
but this is equivalent to $t^{\widetilde{\canM}}\rest\delta$ for this segment,
and that is definable over $H_\delta$ and over $J_\delta$). So $\eta_0$ is
indeed
a limit point of $C$ and $\rho_1^{H_{\eta_0}}=\eta_0=\rho_1^{J_{\eta_0}}$. Now it follows that $\canM\sats$ ``$\Tt\rest\eta_0$
is $H_{\eta_0}$-optimal'' and $\canN\sats$ ``$\Uu\rest\eta_0$ is
$J_{\eta_0}$-optimal'', and hence these trees are in the domains of our
strategies, as desired.

Now suppose we have constructed $(\Tt,\Uu)\rest(\eta_\beta+1)$
for some $\beta$, with $\delta((\Tt,\Uu)\rest\eta_\beta)=\eta_\beta$.
To reach
$(\Tt,\Uu)\rest(\eta_{\beta+1}+1)$,
we first determine whether there is $E\in\es(K_{\eta_\beta})$
which induces a bad meas-lim extender algebra axiom with $\nu(E)=\eta_\beta$.
If so, set $E^\Tt_{\eta_\beta}=E^\Uu_{\eta_\beta}=$ the least such.
After that, or otherwise, we proceed with comparison, again in two phases. The first phase is subject to iterating
the least measurable of $K_\alpha$ which is $>\eta_\beta$, to
$\geq\max(\OR(Q^\Tt)^\star),\OR((Q^\Uu)^\star))$,
where $Q^\Tt=Q(\Tt\rest\eta_\beta,[0,\eta_\beta)_\Tt)$
and likewise for $Q^\Uu$, and the superscript-$\star$ denotes the associated
$\star$-translation (using $H_{\eta_\beta}$ and $\Tt\rest\eta_\beta$
for the $\Tt$-side, and $J_{\eta_\beta}$ and $\Uu\rest\eta_\beta$ for
the $\Uu$-side). The second phase
is
subject to $t^{\widetilde{\Momone}}$-genericity iteration as before.
By induction, $(\Tt,\Uu)\rest\eta_\beta$ is definable from parameters
over $H_{\eta_\beta}$ and over $J_{\eta_\beta}$,
which are segments of $(Q^\Tt)^\star$ and $(Q^\Uu)^\star$ respectively.
So from $(Q^\Tt)^\star$ we can recover (from parameters)
first $\Tt\rest\eta_\beta$, and hence
also $\Tt\rest(\eta_\beta+1)$,
 the last step
 because
 $Q^\Tt=((Q^\Tt)^\star)^{\moon}$;
 likewise $\Uu\rest(\eta_\beta+1)$ from $(Q^\Uu)^\star$.
And the interval from $(\Tt,\Uu)\rest(\eta_\beta+1)$
through
$(\Tt,\Uu)\rest(\eta_{\beta+1}+1)$ is like for $(\Tt,\Uu)\rest(\eta_0+1)$.

Now suppose we have defined $(\Tt,\Uu)\rest\eta$ where
$\eta=\sup_{\beta<\zeta}\eta_\beta$ and $\zeta$ is a limit.
So $\eta=\delta((\Tt,\Uu)\rest\eta)$ and $\eta$ is a limit of limit points  of $C$.
Suppose first that $M((\Tt,\Uu)\rest\eta)$ is a Q-structure for itself
(hence we will set $\eta<\eta_\zeta$).
In this case we proceed directly with
 comparison subject to genericity iteration,
leading to $(\Tt,\Uu)\rest\eta_\zeta$.
We then have that $\eta_\zeta=\delta((\Tt,\Uu)\rest\eta_\zeta)$ is a limit point of $C$. We have $\Tt\rest\eta_\zeta\in\dom(\Lambda^{\Momone})$, etc,
since $(\Tt,\Uu)\rest\eta_\zeta$ is definable from $(A,B,\xi)$
over $H_{\eta_\zeta}$ and over $J_{\eta_\zeta}$,
as these structures can compute the genericity aspect as before,
and  we can uniformly recover the earlier Q-structures used to guide $(\Tt,\Uu)\rest\eta_\zeta$ by $\moon$-construction, since we inductively folded in
iteration past their $\star$-translations.
This leads to $(\Tt,\Uu)\rest(\eta_\zeta+1)$.
Finally, if $M((\Tt,\Uu)\rest\eta)$ is not a Q-structure
for itself, then $\eta_\zeta=\eta$, and
we can now proceed basically as in the previous case to see that
$\Tt\rest\eta_\zeta\in\dom(\Lambda^{\canM})$, etc, leading again to $(\Tt,\Uu)\rest(\eta_\zeta+1)$.

This completes the construction of the comparison. Note that $(\Tt,\Uu)\rest\om_1^M\in M$, since it is definable from parameters over $\widetilde{\canM}$.
So it lasts $\delta=\om_1^M$ stages, and $\eta_\beta<\om_1^M$ for each $\beta<\om_1^M$. Either $\Tt$ or $\Uu$
has no cofinal branch in $M$, as before.
 Let $b=\Sigma_A(\Tt)$ (the correct
$\Tt$-cofinal branch)
and $Q=Q(\Tt,b)$. Let $Q^\star=Q^\star(\Tt,\widetilde{\canM})$.

\begin{clmfive}\label{clm:Q^star_reaches_M}
 $Q^\star||\OR(\canM^+)=\canM^+$.\end{clmfive}
\begin{proof}
 Suppose not. By \label{pg:using_Lemma_*-trans_gives_seg_of_M}Lemma \ref{lem:*-trans_gives_seg_of_M}, it follows
 that $Q^\star\pins\canM^+$.
 And
$Q^{\star\moon}=Q^{\star\moon}(\Tt,\widetilde{\canM})=Q$,
so by Lemmas \ref{lem:compute_*_q_translation}
and \ref{lem:Q_computes_b}, we get $b\in M$, and hence there is no
$\Uu$-cofinal branch in $M$. (Our assumptions seem to allow the possibility that $Q\notin M$,
but still the relevant theory $t$ coding $Q$ is in $M$, so $b\in M$.)

\begin{sclmfive}
$(\canN^+)_0^\moon=(\canN^+)_0^\moon(\Uu,\widetilde{\canN})$
 is well-defined and satisfies ``$\delta$ is Woodin''
 (note if $M\sats$ ``$\om_2$ exists''
 then it follows that $\canN^{+\moon}=(\canN^+)_0^\moon$).
\end{sclmfive}
 \begin{proof}
Suppose not and let $R\pins\canN^+$ be least
such that $\widetilde{\canN}\ins R$
and either  (i)
$R^\moon=R^\moon(\Uu,\widetilde{\canN})$
 is ill-defined or not a premouse,
 or (ii) it is a well-defined premouse and is a Q-structure
 for $M(\Uu)$ or projects $<\delta$.

 If (i) holds then working in $\canN^+$,
 which has universe that of $\canM^+$,
we can use condensation to find
$\bar{R}\pins\canN$ and a sufficiently elementary
$\pi:\bar{R}\to R$ with $\crit(\pi)=\bar{\delta}=\om_1^{\bar{\R}}$,
$\widetilde{\canN},\Uu\in\rg(\pi)$,
$\pi(\bar{\widetilde{\canN}})=\widetilde{\canN}$,
 $\pi(\bar{\Uu})=\Uu$ and hence $\bar{\Uu}=\Uu\rest\bar{\delta}$.
Also,
$\bar{\widetilde{\canN}}\ins\bar{R}$,
and  $\bar{\Uu}$ is $\bar{\widetilde{\canN}}$-optimal.
By Lemma \ref{lem:compute_*_q_translation},
the ill-definedness
 of $R^\moon$ reflects to $\bar{R}^\moon(\bar{\Uu},\bar{\widetilde{\canN}})$,
 contradicting that $\canN$ is iterability-good.

 So (ii) holds. But then $R^\moon$ must determine a $\Uu$-cofinal
 branch, because otherwise, we can do a similar reflection argument
 to get a Q-structure for some $M(\bar{\Uu})$ with $\bar{\Uu}\pins\Uu$,
 produced by $\Moon$-construction, which does not yield a $\bar{\Uu}$-cofinal
branch, again contradicting that $\canN$ is iterability-good.
\end{proof}

By the subclaim,
$Q\npins(\canN^+)^\moon_0$.

\begin{sclmfive}\label{sclm:club_C_of_Woodins}
In $M$ (hence also in $\canN^+$) there is a
club $C\sub\delta$
 consisting of Woodin cardinals of $M(\Tt,\Uu)$,
 hence Woodin cardinals of $(\canN^+)_0^\moon$.
\end{sclmfive}
\begin{proof}
By Lemma \ref{lem:compute_*_q_translation},
$t=t_{q+1}^Q(\delta)\in\univ{\canM^+}=\univ{\canN^+}$,
where $\rho_{q+1}^Q\leq\delta<\rho_q^Q$.
 Fix the least $N\pins\canN^+$ such that
 $\widetilde{\canN}\ins N$ and $t\in\J(N)$,
so $\rho_\om^N=\delta$.
By STH and \ref{lem:compute_*_q_translation},
$N^\moon=N^\moon(\Uu,\widetilde{\canN})\pins(\canN^+)^\moon_0$,
$\rho_\om(N^\moon)=\delta$, $(N^\moon)^\star=N$ and
$t$
is definable from parameters over $N^\moon[t_1^{\widetilde{\canN}}]$.
We claim  that $N^\moon\npins Q$.
For suppose $R\pins Q$ and $t$
is  definable from parameters over $R[t_1^{\widetilde{\canN}}]$.
We have that $t_1^{\widetilde{\canN}}$ is also
generic over $Q$ for $\BB_{\measlim,\delta}^Q$,
and from $t$ and $t_1^{\widetilde{\canN}}$
one can compute the corresponding theory of $Q[t_1^{\widetilde{\canN}}]$
which could be denoted $t_{q+1}^{Q[t_1^{\widetilde{\canN}}]}$.
But that theory is not in $Q[t_1^{\widetilde{\canN}}]$
by a standard diagonalization.

So $N^\moon\npins Q$, but $Q\nins N^\moon$.
And we have $Q^\star\pins\canM^+$ and $N\pins\canN^+$.
So working in $M$,  we can fix $P\pins M$
with $\rho_\om^{P}=\delta$ and these objects
all in $\J(P)$, and a form a continuous, increasing
 chain $\left<P'_\alpha\right>_{\alpha<\om_1^M}$
of substructures $P'_\alpha\elem_n P$, with $n<\om$ sufficiently large,
and all relevant objects definable from parameters in $P'_0$,
and a club $C=\left<\delta_\alpha\right>_{\alpha<\om_1^M}$,
such that $P'_\alpha\inter\delta=\delta_\alpha$.
Let $P_\alpha$ be the transitive collapse
of $P'_\alpha$ and $\pi_\alpha:P_\alpha\to P$
the uncollapse,
so $\crit(\pi_\alpha)=\delta_\alpha$
and $\pi_\alpha(\delta_\alpha)=\delta$. By condensation,
we have $P_\alpha\pins\canM$. Let
 $Q_\alpha$, $\widetilde{\canM}_\alpha$, $Q^{\text{s}}_\alpha$ and $N_\alpha^{\text{bh}}$,
$\widetilde{\canN}_\alpha$, $N_\alpha$,
be the resulting ``preimages'' of
$Q$, $\widetilde{\canM}$, $Q^\star$,  and $N^\moon$, $\widetilde{\canN}$,
 $N$
 respectively.\footnote{\label{ftn:clarify_sides}Note $Q$ and $N^\moon$ are outputs of black-hole constructions, whereas $\widetilde{\canM}$, $Q^\star$, $\widetilde{\canN}$ and $N$
 are  outputs of  $\star$-translations.} Then (because $n$ is large enough), condensation and elementarity
give
that $\widetilde{\canM}_\alpha\ins Q^{\text{s}}_\alpha\pins\canM$
and $\widetilde{\canN}_\alpha\ins N_\alpha\pins\canN$
and the relevant first order properties
reflect down to these models at each $\alpha$,
along with $(\Tt,\Uu)\rest\delta_\alpha$, which is the preimage of $(\Tt,\Uu)$.
It follows that the Q-structures used at stage $\delta_\alpha$
in $\Tt,\Uu$ are distinct, and therefore $\delta_\alpha$
is Woodin in $M(\Tt,\Uu)$. So $C$ is a club
of Woodins of $M(\Tt,\Uu)$.
\end{proof}

We can now easily reach a contradiction.
We have $((\canN^+)^\moon_0)^\star(\Uu,\widetilde{\canN})=\canN^+$.
Let $R'\pins\canN^+$ be least such that $C\in\J(R')$,
so $\rho_\om^{R'}=\delta$.
Let $Q'=(R')^\moon$. So $\rho_\om^{Q'}=\delta$ and
$Q'\pins\canN^{+\moon}$ and
 $R'=(Q')^\star$.
So $C$ is definable from parameters over
$Q'[t_1^{\widetilde{\canN}}]$, so
$C\in(\canN^+)_0^\moon[t_1^{\widetilde{\canN}}]$.
But since $\delta$ is Woodin in $(\canN^+)^\moon_0$,
the forcing is
$\delta$-cc in $(\canN^+)^\moon_0$,
so there is a club $D\sub C$ with $D\in(\canN^+)^\moon_0$.
Letting $\eta$ be the least limit point of $D$,
then $\cof^{(\canN^+)^\moon_0}(\eta)=\om$,
so $\eta$ is not Woodin in $(\canN^+)^\moon_0$,
hence not Woodin in $M(\Tt,\Uu)$, a contradiction,
completing the proof of the claim.\end{proof}

Now $Q^\star\ins M$, since $M$ is an $\om$-mouse and
by Lemma \ref{lem:*-trans_gives_seg_of_M}(\ref{item:if_M_om-mouse_then_Q^star_ins_M}).
So $Q^{\star}=M$, by Claim \ref{clm:Q^star_reaches_M}
and Lemma \ref{lem:*-trans_gives_seg_of_M}(\ref{item:Q^star_projects_<=delta}).
But then by STH (\ref{dfn:Q^*_when_Q_correct}) part \ref{item:STH_part_1}, this contradicts the assumption that $M$ is transcendent (\ref{dfn:transcendent}).
\end{proof}

\begin{proof}[Proof of Theorem \ref{tm:HOD_non-tame_mouse}]
We are no longer assuming that $M$ is transcendent, nor an $\om$-mouse.
But we assume  $M\sats\ZFC$ and have $H=\HOD^{\univ{M}}$.
 Suppose $H\neq\univ{M}$; we want to analyse $H$. The analysis is analogous to
that in the tame case,
 Theorem \ref{tm:HOD_tame_mouse}.
However, we will not prove that $\es^W$ is the restriction of $\es^M$ above
$\om_3^M$ (or above anywhere); we will instead get that $M$ is a
$\star$-translation of some appropriate $W$.

Let $\canN\in\mathscr{G}^M\cut\{\Momone\}$ where $\Momone=\Momone^M$.
Recall
that everything in the proof of
Theorem \ref{tm:V=HOD_in_transcendent_mice} preceding its very last  paragraph applies. So we can compare
$(\canM,\canN)$ as there, producing a comparison $(\Tt,\Uu)$ of length
 $\bar{\delta}=\om_1^M$, and either $\Tt$ or $\Uu$
has no cofinal branch in $M$. (In the current proof
we write $\delta=\om_3^M$.)
 Let $b=\Sigma_A(\Tt)$ (the correct
$\Tt$-cofinal branch)
and $Q=Q(\Tt,b)$.  By Claim
\ref{clm:Q^star_reaches_M} of the preceding proof,
$\canM^+=Q^\star||\om_2^M$.

\begin{clmsix}\label{clm:card(om_1^M)=card(om_2^M)}$\om_1^M$ and $\om_2^M$
have the same cardinality (in $V$),
and therefore $\om_2^M<\om_2$.
\end{clmsix}
\begin{proof}
Since $M|\om_2^M=\canM^+=Q^\star||\om_2^M$, and  $Q^\star$, $Q$ and $\om_1^M$ have the same cardinality (in $V$),
and of course $\om_1^M\leq\om_1$.
\end{proof}

Recall
we are now also assuming that
$M\sats\ZFC$, so $\canM^+\sats\ZFC^-$,
so we can't have $\canM^+= Q^\star$
(but it seems $Q^\star$ might be active at $\om_2^M$).
We also have
$\Momone^{+\moon}=\Momone^{+\moon}(\Tt,\widetilde{\canM})=Q||\om_2^M$ is
well-defined, and satisfies ``$\bar{\delta}$ is Woodin''.

\begin{clmsix}\label{clm:1}
$\canN^{+\moon}=\canN^{+\moon}(\Uu,\widetilde{\canN})$ is
well-defined and
$\Momone^{+\moon}=\canN^{+\moon}$.
\end{clmsix}
\begin{proof}
A reflection argument like before shows that
either $\canN^{+\moon}$ is well-defined
(producing a model of height $\om_2^M$)
or it reaches a Q-structure. If it reaches a Q-structure,
we can argue as above
to produce a club of Woodins $\sub\bar{\delta}$ for a contradiction.
And if it does not reach a Q-structure but $\canM^{+\moon}\neq\canN^{+\moon}$,
we can again reflect a disagreement down, noting that it also
produces a club $\left<\delta_\alpha\right>_{\alpha<\om_1^M}$
with disagreement between the Q-structures at stage $\delta_\alpha$,
and hence again a club of Woodins.
\end{proof}

\begin{clmsix}\label{clm:M^moon=N^moon}  We have:
\begin{enumerate}
 \item\label{item:claim4} $M^\moon=M^\moon(\Tt,\widetilde{\canM})$
 is well-defined,
$M^\moon|\deltabar^{+M^\moon}=(\Momone^+)^\moon$, $M^\moon\sats$ ``$\bar{\delta}$ is Woodin'',
and $M^\moon$ is $\star$-valid (and hence has no $\bar{\delta}$-measure),
\item\label{item:claim5(i)} $\cs(\canN)^{\univ{M}}$ is well-defined,
and hence is a proper class premouse $N$ with $\univ{N}=\univ{M}$,
\item\label{item:claim5(ii)}
$N^\moon=N^\moon(\Uu,\widetilde{\canN})$ is well-defined, and
\item\label{item:claim5(iii)}
$M^\moon=N^\moon$.
\end{enumerate}
\end{clmsix}

\begin{proof} Part \ref{item:claim4}: The well-definedness
follows easily
from condensation, since
$(\Momone^+)^\moon$
is
well-defined. The next two clauses are by Lemma \ref{lem:compute_*_q_translation}.
 Finally, the $\star$-validity of $M^\moon$ is by Lemma \ref{lem:bh_inverts_star}.

Parts \ref{item:claim5(i)}--\ref{item:claim5(iii)}: Suppose not.
We have $M\sats\ZFC$. So fix a limit cardinal $\lambda$
of
$M$ such that either
$\cs(\canN)^{M|\lambda}$ or $(\cs(\canN)^{M|\lambda})^\moon$ is not
well-defined,
or $(\cs(\canN)^{M|\lambda})^\moon\neq
(M|\lambda)^\moon$.
We will again reflect the failure down to a segment
of $\canM^+$, and reach a contradiction. We have to be a little
careful how we form the hull to do this, however.

Note that standard condensation
holds for all segments of $M^\moon$, since otherwise
by condensation in $M$ we could reflect the failure
down to a segment of $(\Momone^+)^\moon$, where we do have condensation.
Let $R=\J(M|\lambda)^\moon$;  because $\lambda$ is an $M$-cardinal,
we have $R=\J((M|\lambda)^\moon)$ and $\OR^R=\lambda+\om$ and
$\rho_\om^R=\lambda$ and $R\pins M^\moon$.
Let $\alpha<\om_2^M$ with $\canN\in M|\alpha$
and $\alpha=\crit(\pi_\alpha)$ where
$\pi_\alpha:C_\alpha\to R$ is the uncollapse map for
$C_\alpha=\cHull_1^{R}(\alpha\un\{\lambda\})$.
(Note that $\widetilde{\canM},\widetilde{\canN},
\Tt,\Uu\in M|\alpha$ also.)
So $C_\alpha=\J(K)$ for some $K$, and $K\elem_1 C_\alpha$ (as
$R|\lambda\elem_1 R$ as $\lambda$ is a cardinal of $R$),
so $C_\alpha$ is $\alpha$-sound, with $\rho_1^{C_\alpha}\leq\alpha$ and
$p_1^{C_\alpha}\cut\alpha=\{\pi_\alpha^{-1}(\lambda)\}$.
Therefore $C_\alpha\pins R$.
Let
\[ C=\Hull_1^R(\bar{\delta}\un\{\lambda,\alpha\}) \]
and $\pi:C\to R$ the uncollapse. Note that $C_\alpha\in\rg(\pi)$,
since $C_\alpha\ins D$ where $D$ is the least segment of $R$ projecting to
$\deltabar$
with $\alpha\leq\OR^D$. Hence $\pi(C_\alpha)=C_\alpha$. It easily follows that
$C$ is $1$-sound with $\rho_1^C=\deltabar$ and
$p_1^C=\{\pi^{-1}(\lambda),\alpha\}$,
and so $C\pins R$.

Let
\[ C'=\cHull_1^{\J(M|\lambda)}(\deltabar\un\{\lambda,\alpha\}) \]
and $\pi':C'\to\J(M|\lambda)$ the uncollapse. Note that
$\rg(\pi')\inter\OR=\rg(\pi)\inter\OR$,
since all $\Sigma_1$ facts true in $R[\widetilde{\canM}]$ are $\Sigma_1$-forced
over
$R$
and $(\deltabar+1)\sub\rg(\pi)$, and
by ´ Lemma \ref{lem:compute_*_q_translation},
 $\J(M|\lambda)$ is $\Sigma_1^{R[\widetilde{\canM}]}
(\{\Tt,\widetilde{\canM},\lambda\})$ and, conversely, $R$
is $\Sigma_1^{\J(M|\lambda)}(\{\Tt,\widetilde{\canM},\lambda\})$. Much as
above,
$C'\pins M$,
and note that $C=(C')^\moon$, by  the elementarity of $\pi,\pi'$
and the corresponding
$\Sigma_1$-definability of the
$\star$-translation/$\Moon$-construction
of $C,C'$.\footnote{Alternatively,
we have $\rho_1^{(C')^\moon}=\rho_1^{C'}=\deltabar$,
and in the $\Moon$-construction, after projecting to $\bar{\delta}$,
a segment cannot later be lost, so $C\ins(C')^\moon$ or vice versa,
but $(\deltabar^+)^{C}=(\deltabar^+)^{C'}=(\deltabar^+)^{(C')^\moon}$,
so $C=(C')^\moon$.}

Now $C\pins\canM^{+\moon}=\canN^{+\moon}$; the equality is by Claim \ref{clm:1}.
Let $K$ be such that $C=\J(K)$
and $K'$ such that $C'=\J(K')$.
By the same claim and elementarity,
$K$ is $\star$-valid with respect to  $\Tt$, and hence also with respect to $\Uu$
(since $M(\Tt)=M(\Uu)$).
Writing
$K^\star_\canM=K^\star(\Tt,\widetilde{\canM})$
and $K^\star_{\canN}=K^\star(\Uu,\widetilde{\canN})$,
we have $K^\star_{\canM}=K'\ins \canM^+$ and since $K\pins C\pins\canN^{+\moon}$ and $K$ is $\star$-valid,
$K^\star_{\canN}$ is well-defined and $K^\star_{\canN}\ins\canN^+$. Because
$K$ has no largest cardinal, $K^\star_\canM=\canM^+|\OR^K$
has universe that of $K[\widetilde{\canM}]$,
and $K^\star_\canN=\canN^+|\OR^K$ that of $K[\widetilde{\canN}]$,
but note these universes are identical.
Because $\canM,\canN\in\mathscr{G}^M$ (\ref{dfn:G^M})
and by an easy reflection below $\om_1^M$,
it follows that $\canN^+|\OR^K\sats$ ``I am $\cs(\canN)$
and $\cs(\canN)^{\moon}$ is well-defined
and equals $K$''. But since also $K=(\canM^+|\OR^K)^{\moon}$
and $\pi':C'\to\J(M|\lambda)$ is sufficiently elementary,
this gives a contradiction, establishing the claim.
\end{proof}

Now we have
$\card^M(\mathscr{G}^M)\leq\om_2^M$ and
\footnote{In the
analogous situation in the tame
case, we had $\mathscr{G}^M\sub\mathscr{P}^M$
and $\card^M(\mathscr{P}^M)\leq\om_1^M$,
but
for non-tame, as far as the author knows,
we might have
$\mathscr{G}^M\not\sub\mathscr{P}^M$.}
$\delta=\om_3^M$.
For
 $\canN\in\mathscr{G}^M$, let $\canN^{++}=\sJs(\sJs(\canN))^{\her_\delta^M}$, so $\canM^{++}=M|\delta$,
and by Claim \ref{clm:M^moon=N^moon},
$\canN^{++}$ has universe $\her_\delta^M$. Also let
$t_{\canN}=\Th_{\rSigma_2}^{\canN^{++}}(\delta)$.
Then for all $\canN_1,\canN_2\in\mathscr{G}^M$,
there is $\betavec\in(\omega_2^M)^{<\om}$ such that for all
$\alphavec\in(\omega_3^M)^{<\om}$,
\[ t_{\canN_1}\rest\{\betavec,\alphavec\}\text{ is recursively
equivalent to }
t_{\canN_2}\rest\{\betavec,\alphavec\},\text{ uniformly in }\alphavec.
\]
For this, just choose $\vec{\beta}\in(\om_2^M)^{<\om}$ such that $\{\canN_2\}$ is $\Sigma_1^{\canN_1^+}(\{\vec{\beta}\})$
and $\{\canN_1\}$ is $\Sigma_2^{\canN_2^+}(\{\vec{\beta}\})$; we can certainly do this, since $\canN_1^+$ and $\canN_2^+$ each have universe $(\her_{\om_2^M})^M$.
This suffices,
since each $\canN^{++}$ has universe $\her_\delta^M$,
and $\canN^{++}$ (including its extender sequence) is $\Sigma_2^{\her_\delta^M}(\{\canN\})$,
since it is $\Sigma_1^{\her_\delta^M}(\{\canN,\om_2^M\})$. (Using the parameters $\canN$ and $\om_2^M$, $\canN^+=\sJs(\canN)^{\her_\delta^M}$
can easily be identified, and, similarly, $(\sJs(\canN^+))^{\her_\delta^M}$
is $\Sigma_1^{\her_\delta^M}(\{\canN^+\})$.)\footnote{A more complicated alternative here is to use the claims regarding the comparison above (comparing each of
$\canN_1,\canN_2$ with
$\canM=\Momone^M$ and considering the respective common $\Moon$-constructions,
to
translate between
$t_{\canN_i}$ and $t^{\canM}$).}
So for extenders with critical points $>\omega_2^M$, $t_{\canN}$-genericity
iteration
for some $\canN$ is equivalent to simultaneous $t_\canN$-genericity iteration
for all
$\canN$.

\label{pg:explain_domain_of_Sigma}
Now  for parts \ref{item:HOD_non-tame_mouse_H[M|delta]}--\ref{item:HOD_non-tame_mouse_H=univ(W)[t]} of the theorem to be proven,
we may assume that $M$ is countable,
by passing to a sufficiently elementary countable hull if necessary. (Well, if $M$ is proper class and we can only form a partially elementary hull, this doesn't quite suffice regarding the definability of $W$. But the proof to follow will bound the level of complexity needed to define $W$, and we could anticipate this in advance when forming the hull.) And in case $M$ is countable,
our $(0,\om_1+1)$-strategy $\Sigma=\Sigma_{\Momone^M}$ suffices, as usual, for the trees on $\Momone^M$ to be considered in what follows, and in particular, in that case we have $\delta<\om_1$. On the other hand, for part \ref{item:W|delta_is_seg_of_Sigma-iterate},
the extra assumptions there give an iteration strategy $\Sigma$ for $\Momone^M$
which, by Claim \ref{clm:card(om_1^M)=card(om_2^M)}, will also be sufficient for the purposes there
(without assuming that $M$ is countable).

Now consider the simultaneous comparison of all $\canN\in\mathscr{G}^M$, as
above, first interweaving iteration at least measurables until passing $\om_2^M$, and then interweaving
$t_\canN$-genericity
iteration (for the meas-lim extender algebra, with details executed essentially as for the comparison of just two premice),
using $\Lambda^{\canN^{++}}$ to iterate $\canN$;
that is, the strategy defined like
$\Lambda^{\canN}$, but over $\canN^{++}$. (Since $\canN\elem_1\canN^{++}$,
this works.) Since $H\psub M$,
we must have $\{\Momone\}\psub\mathscr{G}^M$,
so
the comparison cannot succeed. Let $\Tt_{\canN}$ be the tree on $\canN$.

We can analyse the comparison like we analysed the comparison of two models
earlier,
and we get similar results.
It lasts exactly $\om_3^M$ steps
\label{pg:explain_domain_of_Sigma_comment}(and we will have $\Tt_{\Momone^M}$ according to $\Sigma$ and in $\dom(\Sigma)$, by Claim \ref{clm:card(om_1^M)=card(om_2^M)}),
and letting
$\canN^{+\infty}=\cs(\canN)^{\univ{M}}$,
then
$\canN'=(\canN^{+\infty})^\moon(\Tt_{\canN},\canN^{++})$
is a proper class premouse
extending $M(\Tt_{\canN})$, satisfies ``$\delta$ is Woodin'',
and is independent of
$\canN$. So
 $W=\canN'$ is definable without parameters over
$\univ{M}$,
 and each $t_{\canN}$ is
$(W,\BB_{\measlim,\delta}^W)$-generic. In particular,
$W\sub H=\HOD^{\univ{M}}$. Let
$t=\Th_{\Sigma_2}^{\her_\delta^M}(\delta)$.

\begin{clmsix}\label{clm:6} $H=\univ{W}[t]$ and
$\univ{M}=H[\Momone^M|\delta]$.\end{clmsix}

\begin{proof}By the previous paragraph, $t$ is
$(W,\BB_{\measlim,\delta}^W)$-generic and
$W[t]\sub H$.
And letting $\QQ\in H$ be Vopenka for adding subsets of $\om_1^M$,
then
$G_{\Momone^M}$ is $(H,\QQ)$-generic. We need to examine more
closely the particular Vopenka needed to add $\Momone^M$.

\begin{sclmsix}Let $A\sub\mathscr{G}^M$ be $\OD^{\univ{M}}$.
Then $A$ is
$\Sigma_2^{\her_{\delta}^M}(\{\alpha\})$
for some $\alpha<\delta$.\end{sclmsix}

\begin{proof} Let $\lambda$ be some limit cardinal of $M$
such that $A$ is $\OD$ over $\her_\lambda^M$.
Let $\canN\in\mathscr{G}^M$ and
 choose
$\alpha<\delta$ such
that letting
\[
C_{\canN}=\cHull_1^{\J(\canN^{+\infty}|\lambda)}
(\omega_2^M\un\{\lambda,\alpha\}) \]
and $\pi_{\canN}:C_{\canN}\to\J(\canN'|\lambda)$
be the uncollapse, then $C_{\canN}$
is sound
with $\rho_1^{C_{\canN}}=\omega_2^M$ and
$p_1^{C_{\canN}}=\{(\pi_{\canN})^{-1}(\lambda),\alpha\}$
and $C_{\canN}\pins {\canN}^{++}$.
For $\canN_1,\canN_2\in\mathscr{G}^M$,
then $\canN_1^{+\infty}|\lambda$ is inter-definable with
$\canN_2^{+\infty}|\lambda$,
uniformly in  parameters $\canN_1,\canN_2\in\her_\gamma^M\sub
C_{\canN_1}\inter
C_{\canN_2}$, where $\gamma=\om_2^M$.
It follows that $\xi\eqdef\OR(C_{\canN_1})=\OR(C_{\canN_2})$ and
$C_{\canN_1},C_{\canN_2}$
have the same
universe.
But then note that $A$ is $\Sigma_2^{\her_{\delta}^M}(\{\xi,\alpha\})$
for some
 $\alpha<\xi$,
because $\mathscr{G}^M$ is definable over $\her_{\gamma}^M$ and
$\{\her_{\gamma}^M\}$ is
$\Sigma_2^{\her_\delta^M}$, so the function $\canN\mapsto C_{\canN}$ is
$\Sigma_2^{\her_\delta^M}(\{\xi\})$, and this suffices.
This proves the subclaim.\end{proof}

Let $\PP\in H$ be the Vopenka corresponding to $\OD^M$ subsets of
$\mathscr{G}^M$, taking
 ordinal codes $<\delta$  in
the natural form
given by the foregoing proof, as conditions.
Note then that $\PP$ (with its ordering) is
$\Sigma_2^{\her_\delta^M}$,
and $\PP\in\univ{W}[t]$.

Given $\canN\in\mathscr{G}^M$, note that
 $\canN^{++}$ can be computed from $(G_\canN,t)$,
 so $\canN^{++}\in W[t][G_\canN]$.
 Conversely, easily $G_\canN\in W[t][\canN^{++}]$.
Since $\canN^{+\infty}=W^\star(\Tt_{\canN},\canN^{++})$,
therefore
$\univ{\canN^{+\infty}}=W[t][\canN^{++}]=W[t][G_\canN]$.
In particular,
\[ \univ{M}=W[t][G_{\Momone^M}]=H[G_{\Momone^M}]. \]

It follows that $\univ{W}[t]=H$, just by the general $\ZFC$ fact
that if $N_1\sub N_2$ are proper class transitive models of $\ZFC$
and there is $\PP\in N_1$ and $G$ which is both $(N_1,\PP)$-generic
and $(N_2,\PP)$-generic and $N_1[G]=N_2[G]$, then $N_1=N_2$.
This
proves Claim \ref{clm:6}.
\end{proof}
We have now completed the proof
except for one more fact when below a Woodin limit of Woodins:
\begin{clmsix}
Suppose $M$ is below a Woodin limit of Woodins.
Then  there is $\alpha<\om_3^M$
such that $\univ{M}=H[M|\alpha]$,
and hence some $X\sub\om_2^M$ with $\univ{M}=H[X]$.
\end{clmsix}
\begin{proof}For this, let
$\alpha_0$ be a proper limit stage of $\Tt=\Tt_{\Momone^M}$
such that the Woodins of $W|\delta$
are bounded strictly below $\delta(\Tt\rest\alpha_0)$, and let
$\alpha>\delta(\Tt\rest\alpha_0)$ be such that $\Tt\rest(\alpha_0+1)\in
M|\alpha$.
Then $M|\delta$ can be inductively recovered from $M|\alpha$
and $W|\delta$, by comparing $\Tt\rest(\alpha_0+1)$ (as a phalanx)
against $W|\delta$, using the
$\star$-translations
$Q^\star$ of the  Q-structures $Q=Q(\Tt\rest\lambda,b)\ins W$
to compute
projecting mice $N\pins M|\delta$
(noting that if $Q\neq M(\Tt\rest\lambda)$
then $\rho_\om(Q^\star)=\om_2^M$,
because otherwise $\lambda=\om_3^{\J(Q^\star)}$, and
as $Q$ is the common Q-structure
for all trees at stage $\lambda$,
working inside $\J(Q^\star)$,
we can compute $\Tt_{\canN}\rest\lambda$-cofinal branches
for all $\canN\in\mathscr{G}^M$,
which contradicts comparison termination there).
\end{proof}

This proves the theorem.
\end{proof}

\section*{Questions}

\begin{enumerate}
\item Let $M$ be a $(0,\om_1+1)$-iterable premouse modelling ZFC. Recall that $\univ{M}$ is the universe of $M$. Let $H=\HOD^{\univ{M}}$.
\begin{enumerate}
\item Is there $x\in\RR^M$ such that $H=\HOD^{\univ{M}}_{\{x\}}$?

By Corollary \ref{cor:tame_V=HODx},
if $M$ is tame, the answer is ``yes''. By \cite[Theorem 3.11]{V=HODX_pub},
if there is $x\in(\pow(\om_1^M))^M$ which satisfies the equation (but not the demand that $x\in\RR^M$); in fact $x=M|\om_1^M$ does.

\item What is the least $\alpha$
such that $\univ{M}=H[M|\alpha]$?

By Theorem \ref{tm:HOD_non-tame_mouse},
$\alpha\leq\om_3^M$,
and if $M$ is below a Woodin limit of Woodins
then $\alpha<\om_3^M$.
By Theorem \ref{tm:HOD_tame_mouse}, if $M$ is tame then $\alpha\leq\om_1^M$.
\end{enumerate}

\item Do the results of this paper extend to long extender mice?
\end{enumerate}
\section*{Acknowledgements}

Thanks to
Henri Menke for the latex code for $\moon$;
see  \url{https://www.henrimenke.de/} and \url{https://tex.stackexchange.com/questions/231517/how-can-i-write-a-spiral-symbol}.

Work partially funded by the Deutsche Forschungsgemeinschaft (DFG,
German Research Foundation) under Germany's Excellence Strategy EXC
2044 -- 390685587, Mathematics Münster: Dynamics–Geometry–Structure.

Editing funded by the Austrian Science Fund (FWF) [10.55776/Y1498].

\bibliographystyle{plain}
\bibliography{../bibliography/bibliography}

\end{document}